\documentclass[11pt]{amsart}

\usepackage{amsmath}
\usepackage{amssymb}
\usepackage{graphicx}
\newtheorem{theorem}{Theorem}[section]
\newtheorem{corollary}[theorem]{Corollary}
\newtheorem{lemma}[theorem]{Lemma}

\theoremstyle{definition}
\newtheorem{definition}[theorem]{Definition}

\newtheorem{remark}[theorem]{Remark}

\numberwithin{equation}{section}
\newcommand{\eps}{\varepsilon}

\DeclareMathOperator*{\argmin}{\arg\min}

\newcommand{\bbC}{\mathbb C} 
 \newcommand{\bbN}{\mathbb N}
\newcommand{\bbP}{\mathbb P} 
\newcommand{\bbR}{\mathbb R} 
\newcommand{\bbZ}{\mathbb Z} 
\newcommand{\bbT}{\mathbb T}

\newcommand{\bzero}{{\mathbf 0}}

\newcommand{\bc}{\mathbf c}

 \newcommand{\bx}{\mathbf x} 
\newcommand{\by}{\mathbf y}

\newcommand{\cA}{\mathcal A} \newcommand{\cB}{\mathcal B}
\newcommand{\cC}{\mathcal C}  
 \newcommand{\cF}{\mathcal F}
\newcommand{\cG}{\mathcal G} 
 
\newcommand{\cK}{\mathcal K} \newcommand{\cL}{\mathcal L}
 
\newcommand{\cO}{\mathcal O}  
 
\newcommand{\cS}{\mathcal S} 
\newcommand{\cU}{\mathcal U}

\newcommand{\aver}[1]{\langle {#1} \rangle}

\title[Learn PDE from its solution]{{How much can one learn a partial differential equation from its solution?}}

\author[Y. He]{Yuchen He}

\address[Y. He]{School of Mathematical Sciences, Shanghai Jiao Tong University, Shanghai 200240, China}
\email{{\tt yuchenroy@sjtu.edu.cn}}

\author[H. Zhao]{Hongkai Zhao}
\address[H. Zhao]{Department of Mathematics, Duke University, NC 27705}
\email{\tt zhao@math.duke.edu}

\author[Y. Zhong]{Yimin Zhong}
\address[Y. Zhong]{Department of Mathematics and Statistics, Auburn University, AL 36830}
\email{\tt yimin.zhong@auburn.edu}

\keywords{PDE learning, elliptic differential operator, parabolic PDE, hyperbolic PDE, operator spectrum, regression}

\subjclass[2010]{35R30, 35G10, 47K10, 65M32}

\begin{document}

\begin{abstract}
In this work, we study the problem of learning a partial differential equation (PDE) from its solution data. PDEs of various types are used to illustrate how much the solution data can reveal the PDE operator depending on the underlying operator and initial data. A data-driven and data-adaptive approach based on local regression and global consistency
is proposed for stable PDE identification. Numerical experiments are provided to verify our analysis and demonstrate the performance of the proposed algorithms.
\end{abstract}

\maketitle

%%%%%%%%%%%%%%%%%%%%%%%%%%%%%%%%%%%%%%%%%%%%%%%%%%%%%%%%%%%%%%%%%%%%%%%%%%%%%%%%

\section{Introduction}
Partial differential equations (PDE) have been used as a powerful tool in modeling, studying, and predicting in science, engineering, and many real-world applications.
%~\cite{phillipson2009modeling,ghergu2011nonlinear,fefferman2018partial}. 
Many of them are naturally time-dependent, i.e., modeling the dynamics of an underlying system that is evolving in time,  such as heat/diffusion equation, convection/transport equation, Schr\"{o}dinger equation, %~\cite{fibich2015nonlinear}, 
Navier-Stokes equation.
%~\cite{fefferman2000existence}. 
In particular, a PDE can be an effective way to model the underlying dynamics or represent a map from the initial input to the terminal output.
%~\cite{hamilton1982inverse}. 
Effectiveness comes from the fact that, while a PDE model typically does not have many terms, it can capture various physical laws, diverse mechanisms, and rich dynamics.
%~\cite{rudy2017data,rudy2019data,he2020robust,schaeffer2013sparse}. 
Moreover, a PDE model is easy to interpret and the parameters (coefficients) are  meaningful.  

In the past, most of these equations are derived from basic physical laws and assumptions, such as Newton's laws of motion~\cite{marion2013classical}, conservation laws~\cite{phillipson2009modeling}, Fick's laws of diffusion~\cite{fick1855ueber}, etc., plus simplification/approximation to various extent. They have been extremely successful in modeling and studying physical, biological, environmental, and social systems and solving real-world problems. With the advance of technologies, abundant data are available from measurements and observations in many complex situations where the underlying model is not yet available or not accurate enough. Whether one can learn a PDE model directly from measured or observed data becomes an interesting question. Here, data mean the observed/measured solution and its (numerically computed) derivatives, integrals, and other transformed/filtered quantities. 

Many methods have been proposed for PDE learning~\cite{Wu-Xiu-2020,bongard2007automated,long2018pde,shea2021sindy,schmidt2009distilling,messenger2021weak,schaeffer2017learning,he2020robust,kang2021ident,long2019pde,rudy2017data,rudy2019data,raissi2018hidden,schaeffer2013sparse,he2021asymptotic,xu2019dl,fasel2021ensemble,kaheman2020sindy,xu2022physics}. In general, there are two approaches. One is to directly approximate the operator/mapping, i.e. $u(t)\rightarrow u(t+\Delta t)$,  learned/trained from solution data. We call this approach differential operator approximation (DOA). The approximation is restricted to a chosen finite dimensional space such as on discretized meshes in the physical domain, or a transformed space, e.g., Fourier/Galerkin space. In particular, several recent works proposed using various types of neural networks to approximate the operator/mapping  \cite{li2020fourier, Wu-Xiu-2020,long2018pde,long2019pde}. Although this approach is general and flexible, it does not explicitly reconstruct the differential operator and hence does not help to understand the underlying physics or laws. Also, it does not take advantage of the compact parameterized form of a differential operator and hence requires a large degree of freedom to represent the mapping. Moreover, it means more data and more computational costs are required to train such a model. 
Another approach is to approximate the underlying PDE using a combination of candidates from a dictionary of  basic differential operators and their functions. We call this approach differential operator identification (DOI), which determines the explicit form of differential operators. Since DOI uses terms from a predetermined dictionary (prior knowledge) to approximate the underlying PDE, it involves much fewer degrees of freedom, data, and computational cost compared to the previous approach. Moreover, using the explicit form of differential operator allows approximation using local measurements in space and time, which is more practical in real applications. Once found it is also more helpful to understand the underlying physics or laws.  Several models and methods \cite{bongard2007automated,long2018pde,schmidt2009distilling,messenger2021weak,schaeffer2017learning,he2020robust,kang2021ident,long2019pde,rudy2017data,rudy2019data,raissi2018hidden,schaeffer2013sparse} have been proposed along this line. 
However, most of the previous works are focused on PDE identification with constant coefficients, which has only a few unknowns.
The problem is formulated as a regression problem from a prescribed dictionary, e.g., polynomials of the solution and its derivatives. Model selection by promoting sparsity may be applied.
Typically single or multiple solutions to the PDE are sampled on a global and dense rectangular grid in space and time (from $t=0$) to identify a PDE model. 

A common issue in most of the previous studies is that PDE learning is posed in an ideally controlled setup in terms of both initial data and the number of solution trajectories, which are as diverse and as many as one wants. Moreover, the solution data are sampled on a uniform and dense grid in space and time, including $t=0$. In many real applications, one only has the chance to observe a phenomenon and its dynamics once it happens. 
For examples, recording seismic waves~\cite{douglas2016recent}, satellite data of weather development~\cite{hasler1998high}, remote sensing of air pollution~\cite{beran1974remote}, etc.  Either the event does not happen often or when it happens next time, the environment and setup are very different.  This means that the observation data is the solution to a PDE corresponding to one initial condition, i.e, a single trajectory.  Moreover, one may not afford if not impossible to measure or observe the data everywhere in space and time (especially close to initial time). In other words, only one solution trajectory with uncontrollable initial data measured by local sensors (for some time duration) at certain locations after some time lapses may be available for PDE learning in practice. In addition to these realities,  heterogeneous environment or material properties requires PDE learning with unknown coefficients varying in space and time, which is even more challenging due to the coupling of the unknown functions with the solution data. 
Last but not least, not all solution data should be used in PDE learning indifferently since PDE solution may be locally degenerate, e.g., zero, or singular in certain regions in space and time,
which may lead to significant errors due to measurement noise and/or numerical discretization errors.

In this work, we will study a few basic questions in PDE learning using various types of PDE models. In particular, we will characterize the data space, the dimension of the space spanned by all snapshots of a solution trajectory 
%-- a family of functions in space parametrized by time,  
with certain tolerance, and show how it is affected by the underlying PDE operator and initial data. This information is of particular importance for DOA approaches. Since the operator approximation is trained by snapshots along a trajectory, it means the operator can be possibly learned only in this data space. Then we focus on the PDE identification problem using a single trajectory. We first study the identifiability issue and then propose a data-driven and data-adaptive approach for general PDEs with variable coefficients. The main goal is to identify the PDE type correctly and robustly using a minimal amount of local data. The required data are a few patches of local measurements that can resolve the variation of the coefficients and the solution. Our proposed Consistent and Sparse Local Regression (CaSLR) method finds a differential operator which is 1) globally consistent, 2) built from as few terms as possible from the dictionary, and 3) a good local fit to data using different linear combinations at different locations. Once the PDE type is determined, a more accurate estimation of coefficients can be achieved by using regression on (more refined) local measurements and/or proper regularization. Numerical examples are used to verify our analysis and demonstrate the performance of the proposed algorithms. 

We note that our CaSLR method has some similarities to the method proposed in ~\cite{rudy2019data} which uses group sparsity to enforce global consistency.
 However,~\cite{rudy2019data} uses solution data on a dense grid in space and time and restrained their exploration to PDEs with coefficients varying either in space or in time, but not both. When the coefficients depend only on space, they propose to identify a constant coefficient PDE for each time using the corresponding snapshot data in whole space, and they address the case where coefficients depend only on time analogously. The CaSLR is different from this previous work in several major aspects. First, our method identifies PDE with coefficients varying both in space and time based on patches of local data.  Second, a different feature identification process that does not require additional thresholding is used. We also provide an identification guarantee analysis. 

In our theoretical analysis, we use linear evolution partial differential equations of the following form as examples:
\begin{equation}\label{EQ: MODEL}
\begin{aligned}
         \partial_t u (x, t) &= -\cL u(x, t), \quad (x,t)\in\Omega\times [0, T],\\
         u(x, 0) &= u_0(x).
\end{aligned}
\end{equation}
To avoid the complication of different possible boundary conditions for different operator $\cL$ in our study,  we consider the periodic setup: $\Omega = \bbT^d:= \bbR^d \slash (2\pi\bbZ)^d$ as the $d$-torus. The time-independent linear differential operator $\cL$ is parameterized by the unknown coefficients $\{ p_{\alpha}(x) \}_{|\alpha|=0}^n\subset \cK(\Omega; \bbR)$,
\begin{equation}\label{EQ: OP}
    \cL f(x) = \sum_{|\alpha|=0}^{n} p_{\alpha}(x) \partial^{\alpha} f(x),
\end{equation}
where $\cK(\Omega;\bbR)$ is a certain set of real-valued periodic functions. 
\begin{remark}
When one learns a PDE operator, a source term should not be part of the unknown. Otherwise, there is no unique solution. 
\end{remark}

\section{Data space spanned by solution trajectory }\label{SEC: SOL SPACE}
As discussed before, a single solution trajectory, i.e., the solution $u(x,t)$ corresponding to an initial condition, is likely what one can observe in practice. Hence a basic question in PDE learning from its solution is how large is the data space spanned by all snapshots $u(x,t), 0\le t\le T<\infty$, along a solution trajectory. In the DOA approach, snapshot pairs $(u(x,t), u(x,t+\Delta t))$ are typically used to train the approximation, which implies that the action of the operator restricted to this data space can be observed. In the following study, we characterize the information content of a solution trajectory by estimating the least dimension of a linear space (in $L^{2}(\Omega)$) where all snapshots $u(\cdot,t)$, $0\le t\le T$ are $\eps$ close to in $L^2$ norm. This characterization is dual to the Kolmogorov $n$-width \cite{kolmogoroff1936uber} of the solution trajectory as a family of functions in $L^2(\Omega)$ parameterized by time $t$.

As examples, we use two different types of operators where $\cL$ is, 1) a strongly elliptic operator, and 2) a first-order hyperbolic operator, to show different behaviors. These two types of operators are ubiquitous for physical models such as diffusion equation, viscous Stokes flow, transport equation, etc. Intuitively, when $\cL$ is the Laplace operator, the solution stays close to a low dimensional space since $u(\cdot, t)$, $t>0$ is analytic in space due to the smoothing effect.  We show in Theorem~\ref{THM: RANK} that for a general elliptic operator $\cL$, all snapshots of any single trajectory $u(x, t)$, i.e., $u(\cdot,t)$, $0<t<T$, stays $\eps$ close to a linear space of dimension at most of the order $O(|\log\eps|^2)$. This implies the intrinsic difficulty in a direct approximation of the mapping, $u(x,t)\rightarrow u(x,t+\Delta t)$ for small $\Delta t$, i.e., the generator, for a parabolic differential operator by a single trajectory since the test function space spanned by all snapshots of a solution trajectory is very small. 
On the other hand, if $\cL$ is a first-order hyperbolic operator, the data space spanned by all snapshots of a single trajectory stays $\eps$ close to linear space of dimension  $O(\eps^{-\gamma})$, where $\gamma$ depends on the regularity of the initial data.   

\subsection{Strongly elliptic operator}
\label{sec:elliptic}
\subsubsection{Preliminaries}
Let $X$ be a Banach space and $\cA: X\to X$ is a linear, densely defined, closed operator  with the spectral set $\text{sp}(\cA)$ and the resolvent set $\rho(\cA)$.
\begin{definition}
Let the sector region $\Sigma_{\delta}\subset\bbC$ be 
\begin{equation}
    \Sigma_{\delta} = \{ z\in\bbC:|\arg (z) |\le \delta\},\quad \delta\in (0, \frac{\pi}{2})\,,
\end{equation}
We say the operator $\cA$ is admissible if 
\begin{equation}
    \|(z - \cA)^{-1}\|_{X\to X} \le \frac{M}{1+|z|} \quad \text{for all } z\in \bbC\slash \Sigma_{\delta}\,,
\end{equation}
where $M$ is a positive constant.
\end{definition}
When the operator $\cA$ is admissible, one can take a suitable contour integration along $\cC$ for the following Dunford-Cauchy integral
\begin{equation}
    e^{-\cA t} = \frac{1}{2\pi i}\int_{\cC} e^{-zt} (z - \cA)^{-1} d z\,,
\end{equation}
the evaluation of $e^{-\cA t}$ then can be approximated through a well-designed quadrature scheme on $\cC$. It has been shown in~\cite{lopez2006spectral} that the approximation can have exponential convergence. Similar results are found in~\cite{gavrilyuk2002, gavrilyuk2004data} as well. Such exponential convergence also enables fast algorithms for solutions of~\eqref{EQ: MODEL}, see~\cite{jiang2015efficient}.
\begin{theorem}[Theorem 1,~\cite{lopez2006spectral}]
If the operator $\cA$ is admissible,
then there exists an operator $\cA_N$ in the form of
\begin{equation}
    \cA_N(t) = \sum_{k=-N}^N c_k e^{-z_kt} (z_k - \cA)^{-1}
\end{equation}
with constants $c_k, z_k\in \bbC$ and 
\begin{equation}\label{EQ: ESTIMATE}
    \| e^{-\cA t}- \cA_N(t) \|_{X\to X}= \cO(e^{-c N})
\end{equation}
uniformly on the time interval $[t_0, \Lambda t_0]$ with $c = \cO(1/\log \Lambda)$.
\end{theorem}
It should be noticed that in~\eqref{EQ: ESTIMATE} the exponent's dependence on time is quite weak. By a simple modification of the above theorem, we have tailored the following corollary for later use.
\begin{corollary}\label{COR: DIM}
For $t\in [t_0, T]$ that $t_0 = \eps^{\kappa}$, $\kappa>0$, one can take $N = C_{\cA}(\kappa)|\log\eps|^2$ such that
\begin{equation}\label{EQ: COR ESTIMATE}
   \| e^{-\cA t}- \cA_N(t) \|_{X\to X} \le \eps
\end{equation}
for certain constant $C_{\cA}(\kappa) > 0$ depending on $\cA$ and $\kappa$.
\end{corollary}
\subsubsection{Properties of strongly elliptic operator}
Assume $\cL$ is a strongly elliptic operators of order $n=2m$, that is, its principal part
\begin{equation}
    (-1)^{n/2} \sum_{|\alpha| = n} p_{\alpha }(x) \xi^{\alpha} \ge \nu |\xi|^n
\end{equation}
for some constant $\nu > 0$ for all $x\in\overline{\Omega}$, $\xi\in\bbZ^d$. Let the Hilbert space $H^{s}(\Omega)$ 
\begin{equation}
    H^s(\Omega):= \Big\{f\in L^2(\Omega): \sum_{\xi\in\bbZ^d} (1 + |\xi|^{s})^2 |\widehat{f}(\xi)|^2 <\infty \Big\},
\end{equation}
we cite the following two classical lemmas~\cite{agmon2010lectures, pazy2012semigroups} for our case.
\begin{lemma}\label{LEM: CLOESD}
There exists a constant $C > 0$ such that 
\begin{equation}
    \|u\|_{H^{2m}(\Omega)} \le C \left(\|\cL u\|_{L^2(\Omega)} + \|u\|_{L^2(\Omega)}\right)
\end{equation}
for all $u\in H^{2m}(\Omega)$.
\end{lemma}

\begin{lemma}\label{LEM: SECTOR}
Let $\cL$ be a strongly elliptic operator of order $n = 2m$, then there exists constants $C, R > 0$ and $\theta\in(0, \frac{\pi}{2})$ that
\begin{equation}
    \|u\|_{L^2(\Omega)} \le \frac{C}{|z|} \|(z + \cL) u\|_{L^2(\Omega)}
\end{equation}
for all $u\in H^{2m}(\Omega)$ and $z\in\bbC$ satisfying $|z|\ge R$ and $\theta - \pi<\arg z<\pi - \theta$.
\end{lemma}
From Lemma~\ref{LEM: CLOESD}, $\cL$ is a closed operator and because the domain of $\cL$ includes all $C^{\infty}(\Omega)$ functions, therefore $\cL$ is also densely defined. The following is a direct corollary of Lemma~\ref{LEM: SECTOR}. 
\begin{corollary}
Let $\cL$ be a strongly elliptic operator of order $n = 2m$, then there exists a positive constant $\mu>0$ such that $\cL+\mu$ is admissible.
\end{corollary}
Since the parameters $\{ p_{\alpha} \}_{|\alpha|=0}^n$ are real-valued, then the strongly elliptic operator 
$\cL$ permits a decomposition $\cL = \cL_0 + \cB$ such that $\cL_0$ is a self-adjoint operator of order $n$ and $\cB$ is a differential operator of order $<n$, therefore the spectrum of $\cL$ is discrete, there are only a finite number of eigenvalues outside the sector $\Sigma_{\delta}$ and {the eigenfunctions may form a complete basis in $L^2(\Omega)$ under certain circumstances~\cite{ browder1953eigenfunctions, agmon1962eigenfunctions}. Let $\mu > 0$ such that $\cL_{\mu}:=\cL + \mu$ is admissible, in the following we assume the initial condition $u_0\in L^2(\Omega)$ can be represented by eigenfunctions}
\begin{equation}
    u_0(x) = \sum_{k=1}^{\infty} c_k \phi_k(x), 
\end{equation}
where $\{\lambda_k, \phi_k\}_{k\ge 1}$ are the eigenpairs of $\cL_{\mu}$ sorted by $\Re \lambda_k$ in ascending order including multiplicity. Furthermore, the eigenvalues $\{ \lambda_k \}_{k\ge 1}$ satisfy the growth rate $\Re\lambda_k = \cO(k^{\beta})$ with $\beta=n/d$~\cite{gaarding1953asymptotic}.
 
\begin{lemma}\label{LEM: REP}
Let the eigenpairs of ${\cL}_{\mu}$ be $\{\lambda_k, \phi_k\}_{k\ge 1}$, then the solution of~\eqref{EQ: MODEL} can be represented by 
\begin{equation}\label{EQ: EXPAND}
    u(x, t) = e^{\mu t}\sum_{k=1}^{\infty} c_k e^{-\lambda_k t} \phi_k(x),
\end{equation}
where the coefficients $\{c_k\}_{k=1}^{\infty}$ satisfy $    u_0(x) = \sum_{k=1}^{\infty} c_k \phi_k(x)$. 
\end{lemma}

\begin{theorem}\label{THM: RANK}
Suppose $\cL_{\mu}$ is admissible that the spectrum sits in the interior of the sector $\Sigma_{\delta}$ and  the  coefficients in~\eqref{EQ: EXPAND} decay as $|c_k|\le \theta k^{-\gamma}$, $\theta > 0, \gamma > 1$, then there exists a linear space $V\subset L^2(\Omega)$ of dimension $C_{\cL}(\kappa) |\log \eps|^2$ such that 
\begin{equation}
    \| u(\cdot, t) - P_{V} u(\cdot, t) \| \le C \eps \|u_0\|, \; \forall t\in [0,T].
\end{equation}
where $P_V$ is the projection operator onto $V$ and $C=C(\theta, \gamma)$. The constant $C_{\cL}(\kappa)$ is chosen from Corollary~\ref{COR: DIM} and $\kappa=\cO(\beta/(\gamma-1))$.
\end{theorem}

\begin{proof}
Without loss of generality, we assume that $\|u_0\|=1$. Let $u_M$ be the truncated series from Lemma~\ref{LEM: REP},
\begin{equation}
    u_M(x, t) = e^{\mu t} \sum_{k=1}^{M} c_k e^{-\lambda_k t} \phi_k(x)\,,
\end{equation}
then $$\|u(\cdot, t) - u_M(\cdot, t) \|_{L^2(\Omega)}\le\theta e^{\mu t} { \frac{M^{1-\gamma}}{\gamma - 1} }e^{-\lambda_{M+1} t }\le \theta e^{\mu t}{ \frac{M^{1-\gamma}}{\gamma - 1} }e^{-\lambda_{M} t }.$$ 
Let $q\in\bbN$ such that $\Re \lambda_q > \mu$, we denote $M_{\eps} = q + \eps^{1/(1-\gamma)}$,  $L = |\log\eps|$ and define the following approximation
\begin{equation}
\begin{aligned}
       w_{\eps}(x, t) &= \sum_{k=1}^{M_{\eps}} c_k \sum_{l=0}^{L} (-1)^l \frac{( \lambda_k-\mu)^l t^l}{l!} \phi_k(x) \\
       &=\sum_{l=0}^{L}  t^l \left( \sum_{k=1}^{M_{\eps}} c_k (-1)^{l}  \frac{(\lambda_k-\mu)^l }{l!} \phi_k(x)\right)\,,
\end{aligned}
\end{equation}
then for each $t$, $w_{\eps}$ sits in the linear space $$V_1=\text{span}\left\{ \sum_{k=1}^{M_{\eps}} c_k (-1)^{l}  \frac{ (\lambda_k-\mu)^l }{l!} \phi_k(x), \; l=0,\dots, L\right\}.$$ 
For $t\in [0,  \frac{1}{C_{\delta}\Re\lambda_{M_{\eps}}}]$, $C_{\delta} = 1+\tan|\delta|$ since the spectrum is included in the sector region, we have  
\begin{equation}
\begin{aligned}
|\lambda_k  - \mu| t &\le \frac{\sqrt{|\Re \lambda_k-\mu|^2 + |\Im\lambda_k|^2}}{C_{\delta}\Re \lambda_{M_{\eps}}}\le \frac{\sqrt{|\Re \lambda_k-\mu|^2 + \tan^2|\delta||\Re\lambda_k|^2}}{C_{\delta}\Re \lambda_{M_{\eps}}}\\
&\le \frac{\sqrt{|\Re \lambda_{M_{\eps}}|^2 + \tan^2|\delta||\Re\lambda_{M_{\eps}}|^2}}{C_{\delta}\Re \lambda_{M_{\eps}}} < 1
\end{aligned}
\end{equation}
 for $ k=1,\dots, M_{\eps} $, then by the error estimate of Taylor expansion on $[0,   \frac{1}{(1+\delta)\Re\lambda_{M_{\eps}}}]$,
\begin{equation}
\begin{aligned}
       \|u_{M_{\eps}}(x, t) - w_{\eps}(x, t)\| &= \| \sum_{k=1}^{M_{\eps}} c_k \phi_k(x) \sum_{l=L+1}^{\infty} (-1)^l \frac{(\lambda_k-\mu)^l t^l}{l!}   \| \\
       &\le \left( \sum_{k=1}^{M_{\eps}} |c_k| \right) \frac{1}{(L+1)!} \\
       &\le {\frac{ \theta }{\gamma - 1}}e^{-(L+1)} \le \frac{ \theta }{\gamma - 1} \eps \,.
\end{aligned}
\end{equation}
which implies 
\begin{equation}
\begin{aligned}
        \|u(x, t) - P_{V_1}u(x, t)\|&\le \|u(x, t) - u_{M_{\eps}}(x, t)\| + \|u_{M_{\eps}}(x, t) - w_{\eps}(x, t)\| \\
        &\le \frac{2\theta }{\gamma - 1}\eps.
\end{aligned}
\end{equation}
On the interval $[ \frac{1}{(1+\delta)\Re\lambda_{M_{\eps}}},T]$, since $\Re \lambda_{M_{\eps}} = \cO(M_{\eps}^{\beta})$, we take $\kappa = \log(\lambda_{M_{\eps}})/|\log\eps| = \cO(\beta / (\gamma - 1))$, then by the Corollary~\ref{COR: DIM}, there exists a linear space $V_2$ of dimension $C_{\cL}(\kappa)|\log\eps|^2$, i.e.,  spanned by $(z_k-\cA)^{-1}u_0, k=1, \ldots, N=C_{\cL}(\kappa)|\log\eps|^2$, such that 
\begin{equation}
    \|u(x, t) - P_{V_2}u(x, t)\|\le\eps,
\end{equation}
Now we define $V = V_1\cup V_2$, then $\dim V\le \dim V_1+\dim V_2 = \cO(|\log\eps|^2)$ and 
\begin{equation}
     \|u(x, t) - P_{V}u(x, t)\| \le \left(2\theta \frac{1}{\gamma - 1} + 1\right)\eps, \quad \forall t\in [0, T].
\end{equation}
\end{proof}
\begin{remark}
If the elliptic operator $\cL_{\mu}$ in Theorem~\ref{THM: RANK} has eigenfunctions that form an orthonormal basis, e.g., a self-adjoint elliptic operator, the statement is still true for $|c_k|\le \theta k^{-\gamma}$ for some $\theta>0, \gamma>\frac{1}{2}$. 
\end{remark}
The following Lemma~\ref{le:multi-trajectories} shows that, for a self-adjoint elliptic operator $\cL$, even if multiple trajectories are available, the data space stays $\eps$ close to the space spanned by the first $\cO((\tau^{-1}\log|\eps|)^{d/n})$ eigenfunctions of $\cL$ after certain $\tau >0$. This poses two difficulties for the direct operator approximation approach. First, unless diverse initial data $u(x,0)$ can be used and the corresponding solution $u(x,t)$ can be observed at $t\ll 1$, an accurate approximation of the mapping is not possible. Second, although all solution trajectories stay close to a low dimensional space spanned by a few leading eigenfunctions of $\cL$, it is not known \emph{a priori} unless in the special case of constant coefficients and simple geometry.
  
\begin{lemma}
\label{le:multi-trajectories}
If $\cL$ is a self-adjoint strongly elliptic operator, then there exists a linear space $V\subset L^2(\Omega) $ such that for any solution data $u(x, t)$ to the equation 
\begin{equation}
    \partial_t u = -\cL u
\end{equation}
with initial condition $u_0\in L^2(\Omega)$, 
\begin{equation}
  \min_{f\in V}\| f(x) - u(x, t) \|_{L^2(\Omega)}\le \eps \|u_0\|_{L^2(\Omega)},\quad\forall t\in [\tau, T]
\end{equation}
where $\dim V = \cO((\tau^{-1}\log|\eps|)^{d/n})$.
\end{lemma}
\begin{proof}
Let $\phi_1, \phi_2, \dots$ be the eigenfunctions of $\cL$, which forms an orthonormal basis for $L^2(\Omega)$, with eigenvalues $\lambda_1, \lambda_2, \dots$ and $V = \text{span} \{\phi_1, \dots, \phi_M\}$ is the subspace formed by the first $M$ eigenfunctions of $\cL$. For each $t\in[\tau, T]$,
\begin{equation}
\begin{aligned}
       \min_{f\in V} \|f(x) - u(x, t)\|_{L^2(\Omega)} &\le \|e^{\mu t}\sum_{k=M+1}^{\infty} c_k e^{-\lambda_k t}\phi_k(x)\|_{L^2(\Omega)} \\&\le e^{(-\lambda_{M+1}+\mu)\tau} \|u_0\|_{L^2(\Omega)}.
\end{aligned}
\end{equation}
We select $M$ such that $e^{(-\lambda_{M+1}+\mu)\tau}\le\eps$, then $\lambda_{M+1} \ge \frac{|\log\eps|}{\tau} + \mu$. From the growth rate that $\lambda_k\sim \cO(k^{n/d})$, we see $M = \cO((\tau^{-1}|\log \eps|)^{d/n}) $ would suffice.
\end{proof}

\subsection{Hyperbolic PDE}
Next, we show that the behavior of solution trajectory for hyperbolic PDEs can be quite different. The data space spanned by snapshots of a single solution trajectory depends on the regularity of the initial data and can be quite rich. 
We use the following first-order hyperbolic PDE defined on a torus $x\in\Omega=[0,2\pi]^d$ with the periodic condition and $t\in[0, T]$ as an example 
\begin{equation}\label{EQ: HYPERBOLIC}
\begin{array}{l}
    \partial_t u(x, t) + \bc(x) \cdot \nabla u(x,t) = 0
    \\
    u(x,0)=u_0(x).
    \end{array}
\end{equation}
Define the following two correlation functions in space and time of a solution trajectory,
\begin{equation}
\label{eq:correlation}
\begin{aligned}
 K(x,y) &:= \int_0^T u(x,s) u(y, s) ds,    && (x,y)\in \Omega\times \Omega, \\
 G(s,t) &:= \int_{\Omega} u(x,t) u(x, s) dx,    && (s,t)\in [0,T]\times [0,T],
\end{aligned}
\end{equation}
where $K(x,y)$ and  $G(s,t)$ define two symmetric semi-positive compact integral operators on $L^2(\Omega)$ and $L^2[0,T]$ respectively. They have the same non-negative eigenvalues $\lambda_1\ge \lambda_2 \ge \ldots \ge\lambda_j\ge\ldots$ with $\lambda_j\to 0$ as $j\to\infty$. Their normalized eigenfunctions form an orthonormal basis in $L^2(\Omega)$ and $L^2[0, T]$. Define $V_K^k$ and $V_G^k$ to be the linear space spanned by their $k$ leading eigenfunctions of $K(x,y)$ and  $G(s,t)$ respectively, which provides the best $k$-dimensional linear spaces that approximate the family of functions $u(\cdot,t)$ (in $L^2(\Omega)$) and $u(x, \cdot)$ (in $L^2([0,T])$). We have
\begin{equation}
\label{eq:approx}
    \int_0^T\|u(\cdot,t)-P_{V_K^k}u(\cdot,t)\|^2_{L^2(\Omega)}dt=\int_{\Omega}\|u(x,\cdot)-P_{V_G^k}u(x,\cdot)\|^2_{L^2[0,T]} dx =\sum_{j=k+1}^{\infty}\lambda_j,
\end{equation}
where $P_{V_K^k}$ and $P_{V_G^k}$ denotes the projection operator to $V_K^k\subset L^2(\Omega)$ and $V_G^k\subset L^2[0,T]$. 

Below we show an upper bound for the dimension of the best linear subspace in $L^2(\Omega)$ that can approximate all snapshots of a single trajectory to $\eps$ tolerance in $L^2(\Omega\times[0,T])$. 
\begin{lemma}\label{LEM: REGULAR}
Let $\bc(x)\in C^p(\Omega)$ be a velocity field and  $u_0 \in C^p(\Omega)$, then there exists a subspace $V\subset L^2(\Omega)$ of dimension $o(\eps^{-2/p})$ that 
\begin{equation}
  \sqrt{ \int_0^T \|P_V u(\cdot, t) - u(\cdot ,t)\|^2_{L^2(\Omega)} dt} \le \eps 
\end{equation}
\end{lemma}
\begin{proof}
From \eqref{eq:approx}, the linear space spanned by the leading $k$ eigenfunctions of the compact operator induced by the kernel function $K(x,y)$ is the best approximation of the family of functions $u(\cdot,t)$ and satisfies
\begin{equation}
    \int_0^T\|u(\cdot,t)-P_{V_K^k}u(\cdot,t)\|^2_{L^2(\Omega)}dt =\sum_{j=k+1}^{\infty}\lambda_j,
\end{equation}
where $\lambda_j$ is the eigenvalues of the compact operator induced by kernel function $K(x,y)$ or $G(s,t)$ defined in \eqref{eq:correlation}.

Let $Z(t; x)$ solve the ODE
\begin{equation}\label{EQ: FLOW}
    \dot{Z}(t; x) = -\bc(Z(t; x)),\quad Z(0; x) = x.
\end{equation}
Since $\bc(x)\in C^p(\Omega)$, then the solution $Z(t; x)\in C^p[0, T]$, 
The solution to the hyperbolic PDE~\eqref{EQ: HYPERBOLIC} is $u(x,t) = u_0(Z(t;x))$, therefore the correlation
\begin{equation}
   G(s, t) := \int_{\Omega} u(x,t) u(x, s) dx = \int_{\Omega} u_0(Z(s; x)) u_0(Z(t; x)) dx.
\end{equation}
Since $G(s,t)\in C^p([0,T]^2)$, its eigenvalue decay follows $\lambda_n = o(n^{-(p+1)})$~\cite{reade1983eigenvalues,reade1984eigenvalues}. Therefore
\begin{equation}
    \sum_{j=k+1}^{\infty} \lambda_j = o(k^{-p}). 
\end{equation}
which completes the proof.

\end{proof}

\begin{remark}
\label{re:compact}
For hyperbolic PDE, the trajectory of a solution corresponding to initial data with less regularity contains more information about the underlying differential operator. For example, when the initial data $u_0(x)$ with $\|u_0\|_2=1$ has a compact support of size $h$. As the support size $h$ goes to zero, the correlation between two snapshots at two times separated by $\cO(h)$ is zero assuming the magnitude of the velocity field $\bc(x)$ is $\cO(1)$. Hence the dimension of the data space spanned by a single trajectory with any fixed tolerance is at least $\cO(h^{-1})$. On the other hand, if $\bc(x)$ is analytic, the dimension of the data space spanned by a single trajectory corresponding to an analytic initial data with tolerance $\eps$ would become ${\cO(|\log\eps|^{d})}$.
However, the low regularity of the solution $u(x,t)$ can cause large discretization errors in the numerical approximation of derivatives of the solution which leads to poor PDE identification. If initial condition design is part of the PDE learning process, it should be a  function with its spatial and Fourier support as large as possible while being resolved by the measurement and computation resolution.
\end{remark}

\begin{remark}
\label{re:example}
Now we use an example to give a lower bound for the dimension  of the best linear subspace in $L^2(\Omega)$ that can approximate all snapshots of a single trajectory to $\eps$ tolerance. Let $\bc(x)$ be the unit vector parallel to $x_1$ axis, $T=2\pi$, $u_0(x) = \sum_{n=1}^{\infty} \sqrt{a_n} \cos n x_1$ with $0<a_n = o(n^{-2(p+1)})$, then $\partial_{x_1}^p u_0\in C(\Omega)$, $u(x,t)=u_0(x_1-t)$. Assume $s>t>0$,
\begin{equation}
\begin{aligned}
      G(s,t)=\int_{\Omega} u(x,t)u(x,s) dx&= (2\pi)^{d-1}\sum_{n=1}^{\infty}\int_{0}^{2\pi} a_n \cos n x_1 \cos n(x_1 - |s-t|) dx_1 \\
      &= \frac{1}{2}(2\pi)^{d} \sum_{n=1}^{\infty} a_n \cos(n (s-t) )\,.
\end{aligned}
\end{equation}
The eigenvalues of $G(s,t)$ on interval $[0,2\pi]$ is $\frac{\pi}{2}(2\pi)^{d} a_n$ with multiplicity of two. Hence we have 
\begin{equation}
    \int_0^{2\pi}\|u(\cdot,t)-P_{V_K^k}u(\cdot,t)\|^2_{L^2(\Omega)}dt=\sum_{j=k+1}^{\infty}\lambda_j =o(k^{-2p-1})\,.
\end{equation}
Since $\max_{0\le t \le 2\pi}\|u(\cdot,t)-P_{V_K^k}u(\cdot,t)\|_{L^2(\Omega)}\ge \sqrt{\frac{1}{2\pi}\int_0^{2\pi}\|u(\cdot,t)-P_{V_K^k}u(\cdot,t)\|^2_{L^2(\Omega)}dt}$, 
we see that the dimension of the best linear subspace in $L^2(\Omega)$ that can approximate all snapshots of a single trajectory with this chosen initial condition $u_0$ to $\eps$ tolerance has to be of order $ \cO(\eps^{-1/(p+\frac{1}{2})})$. 
\end{remark}

\begin{remark}
For a hyperbolic operator with multiple trajectories, unlike the case for elliptic operator stated by Lemma~\ref{le:multi-trajectories},  the solution data space on an interval $[0, T]$ is as rich as the solution data space on $[\tau, T+\tau]$ due to the diffeomorphism induced by the flow $X$ in~\eqref{EQ: FLOW}.
\end{remark}

The above study shows two possible challenges for a DOA approach in practice: 1) limited data space to train the approximation if the underlying differential operator is compressive or smoothing, such as an elliptic operator, or 2) a large number of parameters and a large amount of data as well as an expensive training process are required to approximate a 
differential operator, such as a hyperbolic operator,  with rich trajectory dynamics.

\subsection{Numerical examples}\label{SEC: DIM}
Here we use a few numerical examples to corroborate our analysis above of the dimension estimates of the space spanned by all snapshots along a single solution trajectory for different types of PDEs. 

First, we show how the dimension of the data space corresponding to a single solution trajectory depends on the PDE operator.

\begin{enumerate}
    \item [I.] Transport equation.\begin{equation}\label{EQ: TRANS EX}
\begin{aligned}
       &u_t(x,t) = 4u_{x}(x,t)\;,\quad(x,t)\in[-8,8)\times (0,5]\\ %\label{eq_svdtrans_pde}\\
&u(-8,t) = u(8,t)\;,\quad t\in (0, 5]\\%\label{eq_svdtrans_bcond}\\
&u(x,0) =\begin{cases}\exp(-\frac{1}{1-x^2})\,&\quad x\in(-1,1)\\0\;,&\text{otherwise.}
\end{cases}%\label{eq_svdtrans_init}
\end{aligned}
\end{equation}
\item [II.] Heat equation. \begin{equation}\label{EQ: DIFF EX}
\begin{aligned}
&u_t(x,t) = 4u_{xx}(x,t)\;,\quad(x,t)\in[-8,8)\times (0,5]\\%\label{eq_svddiff_pde}\\
&u(-8,t) = u(8,t),\quad u_x(-8, t) = u_x(8, t)\;,\quad t\in (0, 5]\\%\label{eq_svddiff_bcond}\\
&u(x,0) =\begin{cases}\exp(-\frac{1}{1-x^2})\,&\quad x\in(-1,1)\\0\;,&\text{otherwise.}
\end{cases}%\label{eq_svddiff_init}
\end{aligned}
\end{equation}
\end{enumerate}

Taking the maximal time $T=5$, we split the observation into two phases: the early phase $t\in [0, 2.5)$ and the late phase $t\in(2.5, 5]$. Within each phase, we compute the singular values of the solution matrix $u_{jk}=u(x_j, t_k)$ sampled on the space-time grid with grid size $\Delta x = 16/500$ and $\Delta t = 5/5000$.  In Figure~\ref{FIG: DIMENSION}, (a) shows the singular value distribution of the solution space for the transport equation~\eqref{EQ: TRANS EX}; and (b) shows that for the heat equation~\eqref{EQ: DIFF EX}. We see that the data space spanned by a single solution trajectory of a hyperbolic operator has a significantly larger dimension than that of a parabolic operator. Also, we see that the dimension of the data space at a later time interval does not decrease for a hyperbolic PDE, whereas it decreases considerably for a parabolic PDE.
 
\begin{figure}[!htb]
    \centering
    \begin{tabular}{cc}
    \includegraphics[width = 0.45\textwidth]{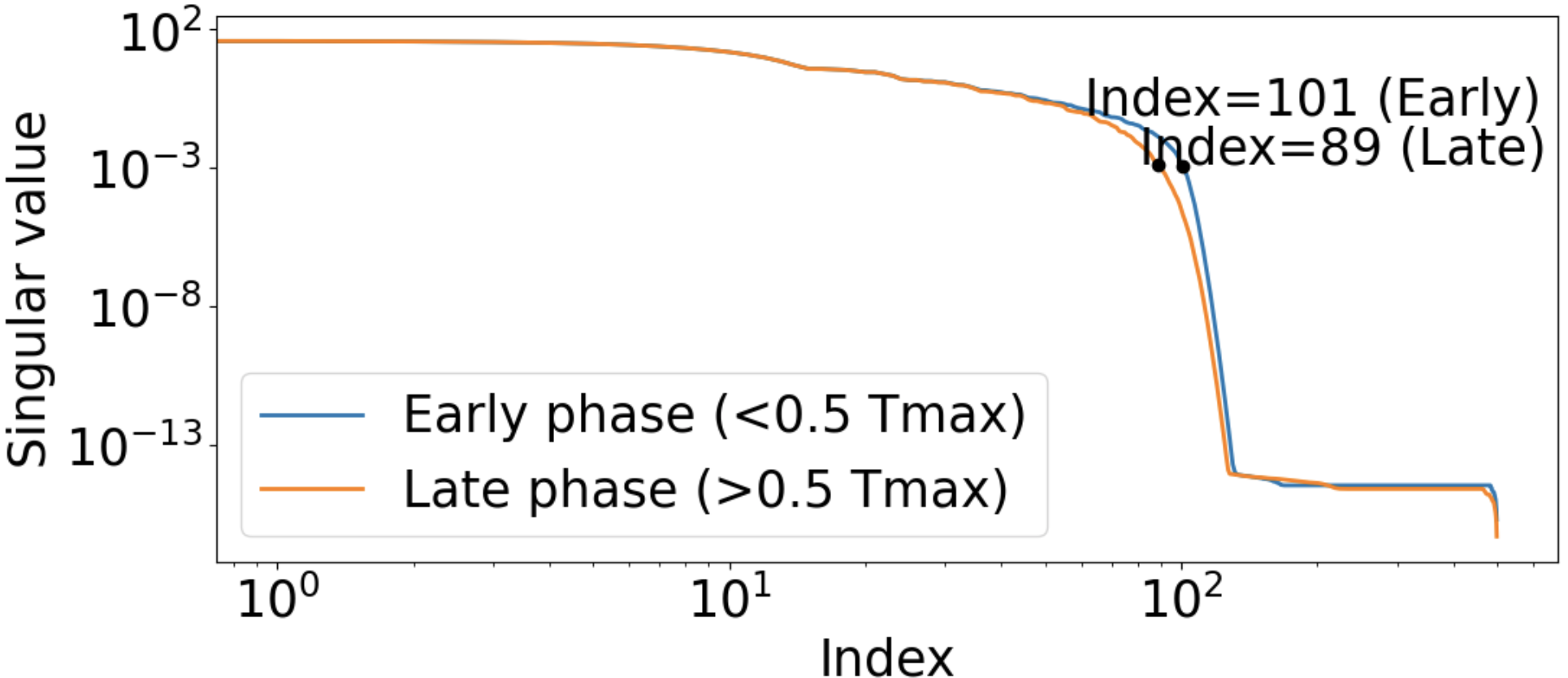}&
    \includegraphics[width = 0.45\textwidth]{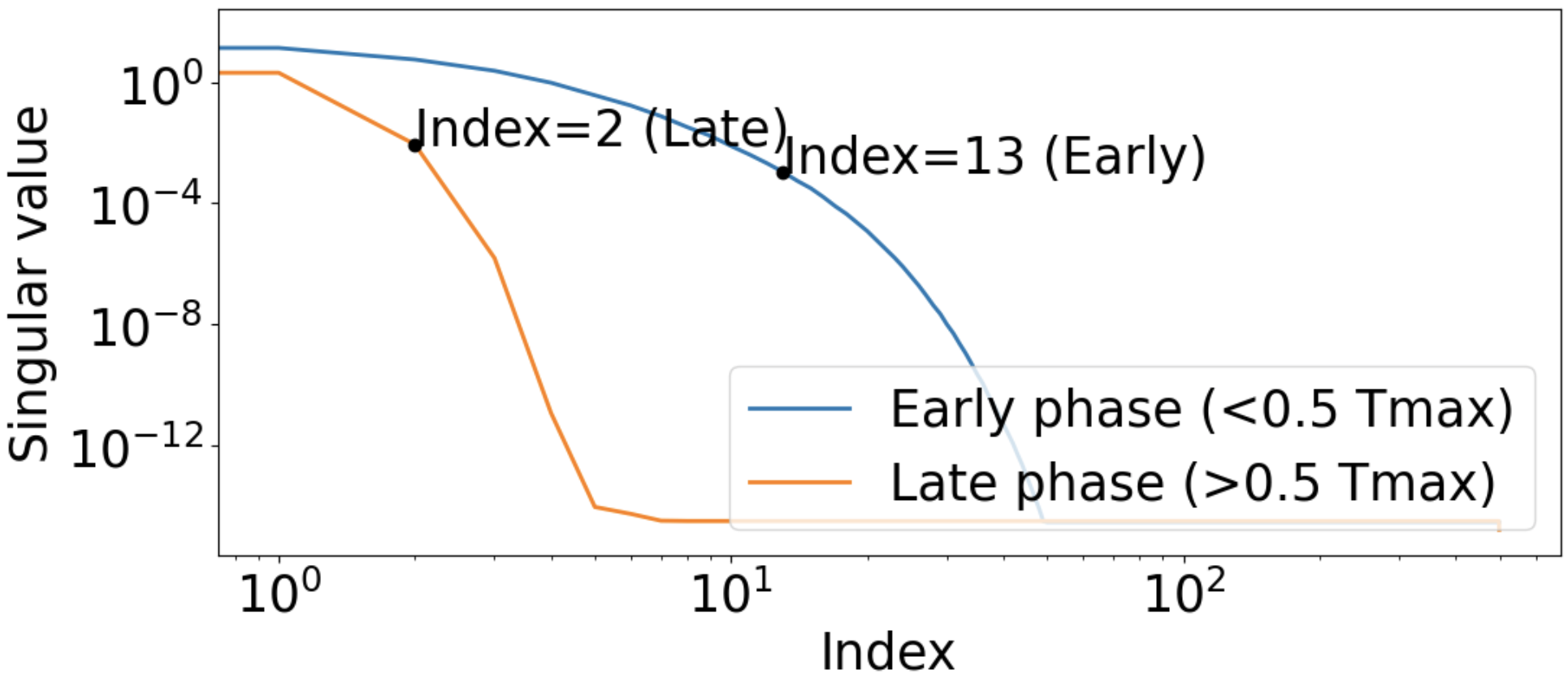}\\
    \qquad\; (a)&\qquad\; (b)
    \end{tabular}
    \caption{ Singular value distribution of the solution matrix $u_{jk}=u(x_j,t_k)$ restricted to the first half and the second half of the time interval of (a) the transport equation~\eqref{EQ: TRANS EX}; (b) the diffusion equation~\eqref{EQ: DIFF EX}. The indices annotated in the figures indicate the numbers of singular values that are greater than $10^{-3}$. In both cases, $T_{\max}=5$ is the maximal observation time.}
    \label{FIG: DIMENSION}
\end{figure}

Now we test the following PDEs with variable coefficients over a  time-space grid with widths $\Delta x = 16/500$ and $\Delta t = 5/5000$. 
\begin{enumerate}
    \item [I'.] Transport equation.\begin{equation}\label{EQ: TRANS EX2}
\begin{aligned}
       &u_t(x,t) = (2+\cos(\frac{\pi x}{8}))u_{x}(x,t)\;,\quad(x,t)\in[-8,8)\times (0,5]\\ %\label{eq_svdtrans_pde}\\
&u(-8,t) = u(8,t)\;,\quad t\in (0, 5]\\%\label{eq_svdtrans_bcond}\\
&u(x,0) =\sin(\frac{\pi x}{8}).%\label{eq_svdtrans_init}
\end{aligned}
\end{equation}
\item [II'.] Heat equation. \begin{equation}\label{EQ: DIFF EX2}
\begin{aligned}
&u_t(x,t) = (2+\cos(\frac{\pi x}{8}))u_{xx}(x,t)\;,\quad(x,t)\in[-8,8)\times (0,5]\\%\label{eq_svddiff_pde}\\
&u(-8,t) = u(8,t),\quad u_x(-8, t) = u_x(8, t)\;,\quad t\in (0, 5]\\%\label{eq_svddiff_bcond}\\
&u(x,0) =\sin(\frac{\pi x}{8}).%\label{eq_svddiff_init}
\end{aligned}
\end{equation}
\end{enumerate}

\begin{figure}[!htb]
    \centering
    \begin{tabular}{cc}
    \includegraphics[width = 0.45\textwidth]{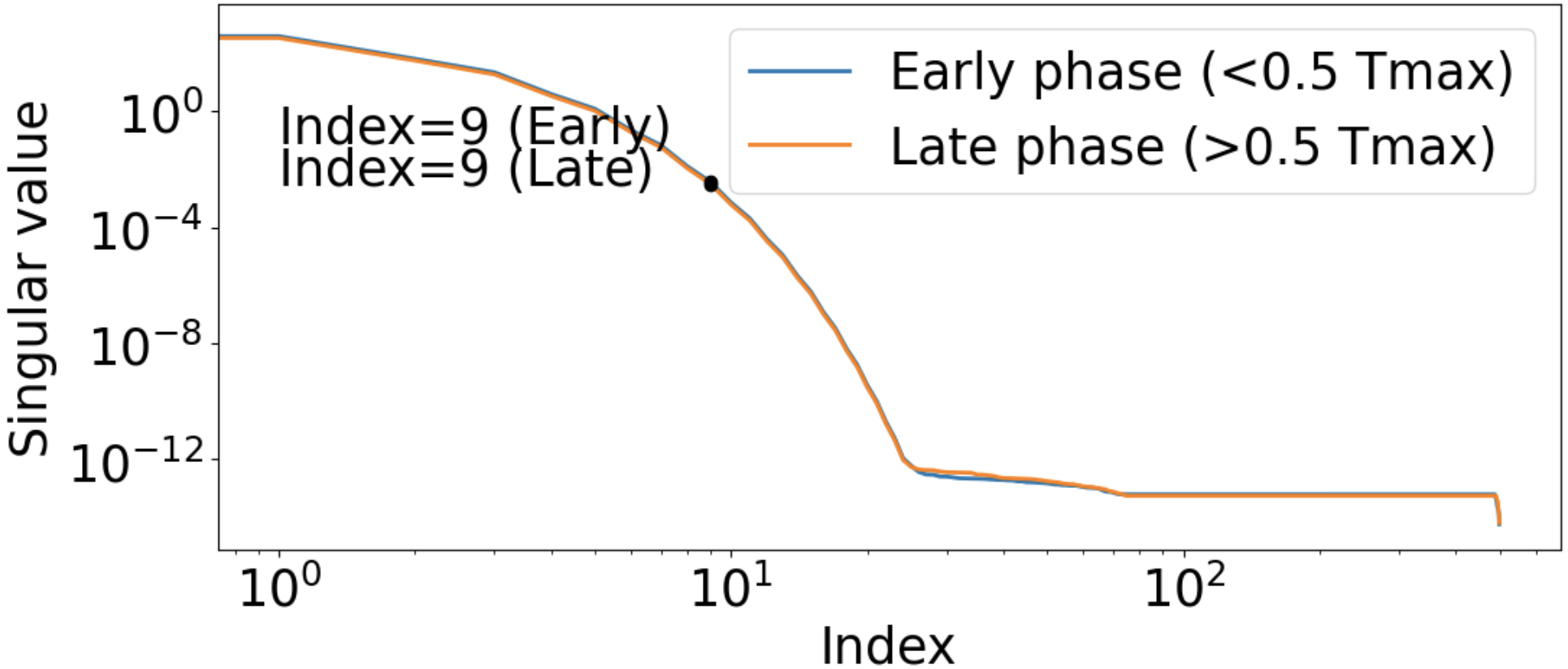}&
    \includegraphics[width = 0.45\textwidth]{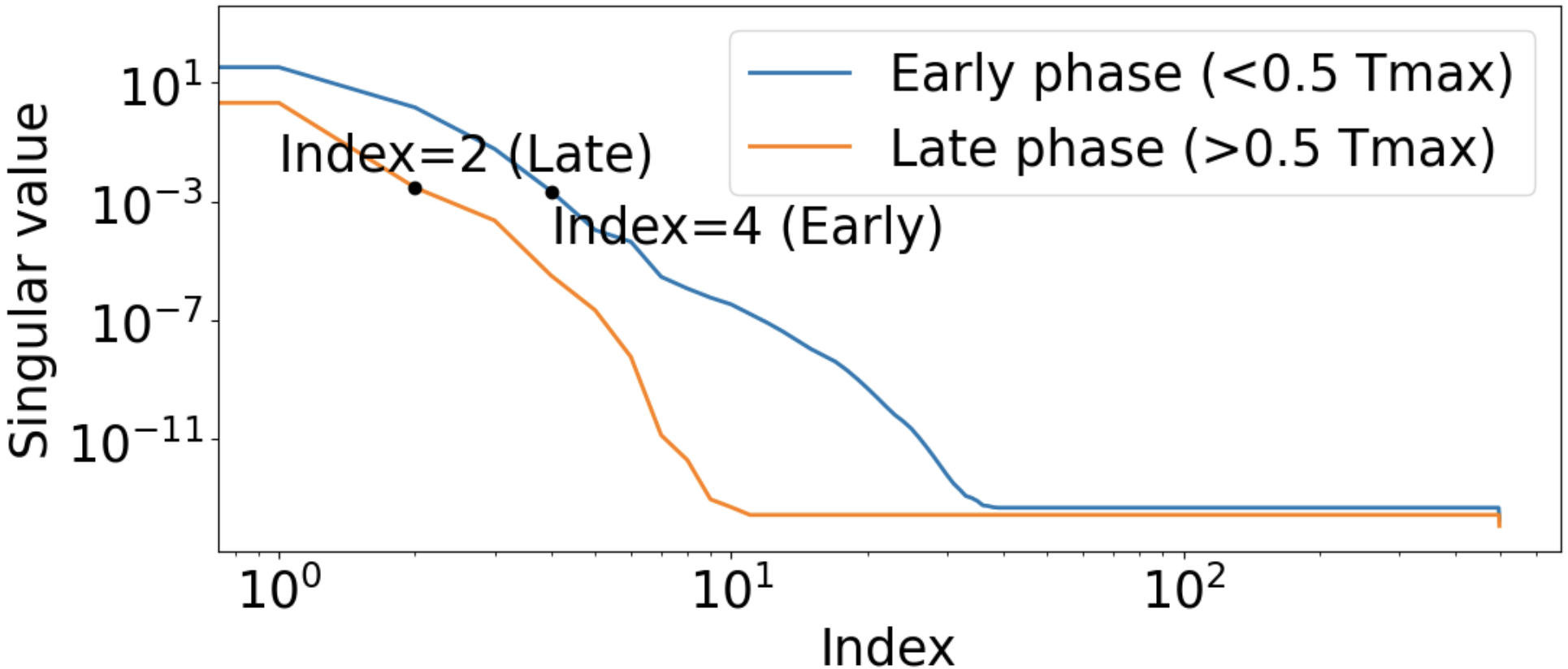}\\
    \qquad\; (a)&\qquad\; (b)
    \end{tabular}
    \caption{Singular value distribution of the solution matrix $u_{jk}=u(x_j,t_k)$ restricted to the first half and the second half of the time interval of (a) the transport equation~\eqref{EQ: TRANS EX2}; (b) the diffusion equation~\eqref{EQ: DIFF EX2}. In both cases, $T_{\max}=5$ is the maximal observation time.}
    \label{FIG: DIMENSION2}
\end{figure}

The results are shown in Figure~\ref{FIG: DIMENSION2}. For the transport equation with constant coefficients, since the initial condition contains only a single sinusoidal mode, the whole solution trajectory contains this single mode with a phase shift and hence lives in a two-dimensional space. For the heat equation with constant coefficients, the whole solution trajectory contains this single mode with a decaying magnitude, which lives in a one-dimensional space. However, the variable coefficients "pump" various modes into the solution trajectory. We see the dimension of the data space of a single solution trajectory behaves similarly to the counterpart with constant coefficients except that the variation of the coefficient in the diffusion equation keeps pumping in modes into the solution and slows down the decay of the singular values in a later time interval.

Finally, we show how the dimension of the solution data space of  a single solution trajectory depends on the initial data. We use two types of initial data. One is initial data with compact support with different regularities. The other one is using initial data that contain a different number of Fourier modes with random amplitude for each mode. For initial data with different regularity, we consider 
\begin{equation}
\label{eq:initial}
u_{\text{square}}(x,0)=\begin{cases} 1,\quad &x\in[-4,0],\\
-1,\quad & x\in(0,4],\\
0,\quad&\text{otherwise}
\end{cases}
\end{equation}
and  $u_{\text{hat}}(x, 0) = \cG (u_{\text{square}}(x, 0))$, $ u_{\text{int}}(x, 0) = \cG(u_{\text{hat}}(x, 0))$, where $\cG$ is the mapping:
\begin{equation}\label{EQ: MAPPING}
\cG f (x):= 
 \int_{-8}^{x} \tilde{f}(s) ds,\quad \tilde{f}(x) = \begin{cases}
 f(2x+4)\quad& x\in [-4, 0],\\
 -f(-2x-4)\quad& x\in (0, 4],\\
 0\quad& \text{otherwise}.
 \end{cases}
\end{equation}
And for the random initial data, we consider %$u(x,0):[-L,L)\to\mathbb{R}$ as follows
\begin{align}
u(x,0) = a_0+\sqrt{2}\sum_{j=1}^{M}\left(a_j\cos\left(\frac{\pi j x}{L}\right)+b_j\sin\left(\frac{\pi j x}{L}\right)\right), \quad x\in [-8,8).
\label{EQ: RANDOMFUNC}
\end{align}
Here $ M$ is the total number of Fourier modes in the initial data and the amplitudes $a_0,a_j,b_j\sim\mathcal{N}(0,1/(2M+1))$, $j=1,2,\dots, M$. 

For initial data with different regularities, we show the percentage of dominant singular values, $\lambda>\varepsilon$, for different threshold $\varepsilon>0$ in Figure~\ref{FIG: DIMENSION_COMPACT}.  Figure~\ref{FIG: DIMENSION_COMPACT} (a) shows the percentage in a log-log plot of the exact solution matrix $u_{jk}=u(x_j, t_k)$ sampled on the space-time grid for the transport equation with constant speed~\eqref{EQ: TRANS EX}. As shown by the argument in Lemma~\ref{LEM: REGULAR},  the less regular the initial data is, the faster the two solution snapshots decorrelate in time and hence the larger the space spanned by the solution trajectory for the transport equation. For the tested initial data given by~\eqref{eq:initial} and~\eqref{EQ: MAPPING}, according to the example given in Remark~\ref{re:example}, the corresponding singular values of the solution matrix decay at an algebraic rate $\lambda_n=\cO(n^{-p-1})$ with $p=0,1,2$, which is verified by the numerical result show in Figure~\ref{FIG: DIMENSION_COMPACT} (a).
 Figure~\ref{FIG: DIMENSION_COMPACT} (b) shows the percentage of dominant singular values in terms of $\log\epsilon$ for the heat equation with constant conductivity~\eqref{EQ: DIFF EX}. Due to exponential decay in time for all eigenmodes in the initial data, the singular value of the solution matrix decays very quickly for all cases, i.e., the space spanned by all snapshots along a single solution trajectory is small for the diffusion equation. As shown by Theorem~\ref{THM: RANK}, the growth can not be more than $|\log\epsilon|^2$. Actually, the numerical results suggest the growth is $c|\log\epsilon|$, where $c$ depends on the regularity of the initial data. The smoother the initial data, the smaller the $c$ is.  

\begin{figure}[!htb]
    \centering
    \begin{tabular}{cc}
    \includegraphics[width = 0.45\textwidth,height=0.25\textwidth
    ]{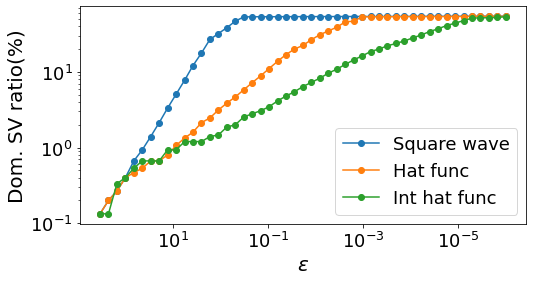}&
    \includegraphics[width = 0.45\textwidth, height=0.25\textwidth]{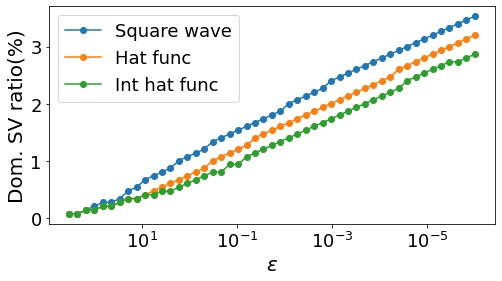}\\
    \qquad\; (a)&\qquad\; (b)
    \end{tabular}
    \caption{Singular value distribution of the solution matrix $u_{jk}=u(x_j,t_k)$ for compact supported initial data with different regularities~\eqref{eq:initial} and~\eqref{EQ: MAPPING} for (a) the transport equation~\eqref{EQ: TRANS EX} in $\log-\log$ plot; (b) the diffusion equation~\eqref{EQ: DIFF EX} in $\log\epsilon$ plot.}
    \label{FIG: DIMENSION_COMPACT}
\end{figure}

Figure~\ref{FIG: DIMENSION_RANDOM} shows the percentage of dominant singular values for solutions of~\eqref{EQ: TRANS EX},~\eqref{EQ: TRANS EX2},~\eqref{EQ: DIFF EX}, and~\eqref{EQ: DIFF EX2} with random initial data constructed as in~\eqref{EQ: RANDOMFUNC}.  We see that the more Fourier modes the initial data contains, the larger the space spanned by the solution trajectory, while the growth rate for the diffusion equation is much slower than that of the transport equation. Also, variable coefficients can introduce Fourier modes into the solution and hence increase the space spanned by the solution trajectory. The increment is more significant when initial data contain fewer Fourier modes.

\begin{figure}[!htb]
    \centering
    \begin{tabular}{cc}
    \includegraphics[width = 0.45\textwidth]{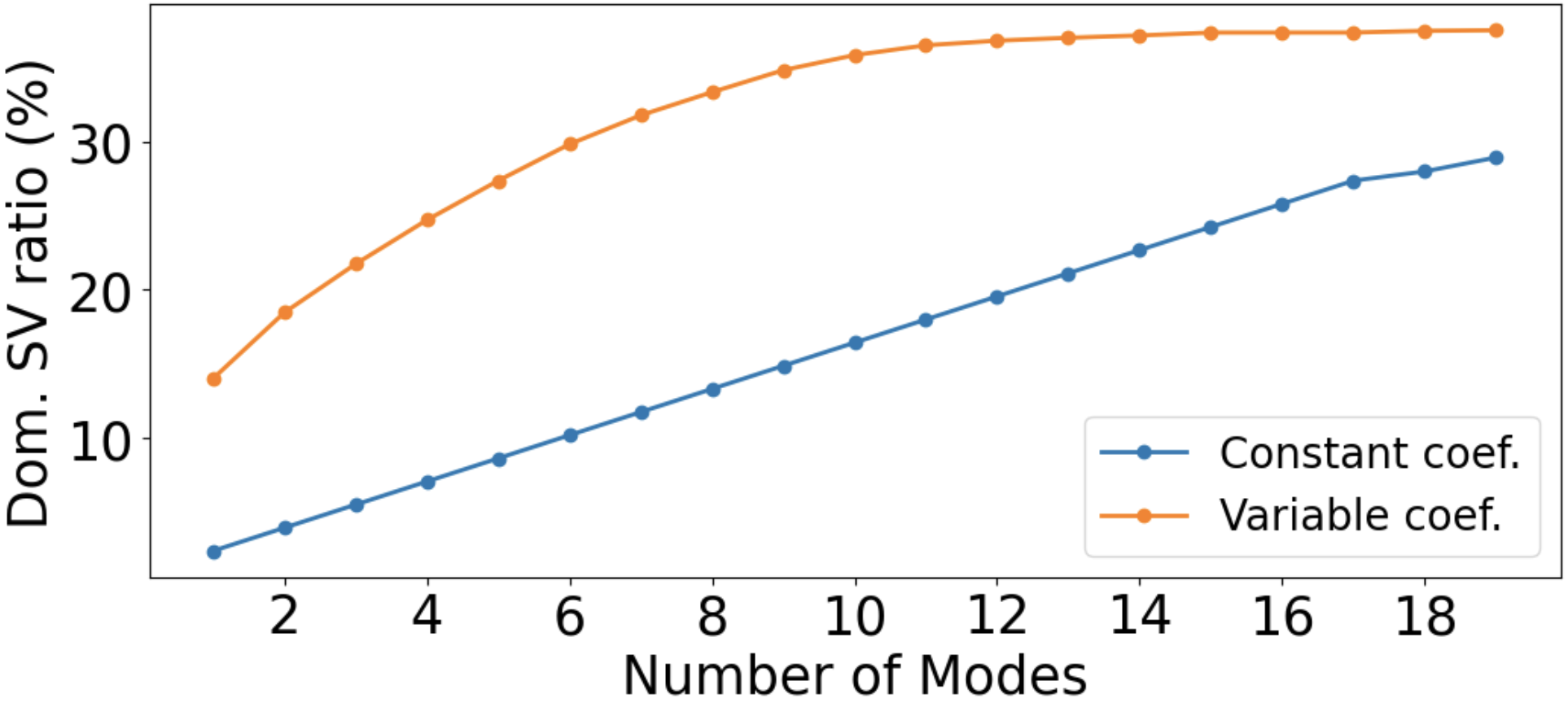}&
    \includegraphics[width = 0.45\textwidth]{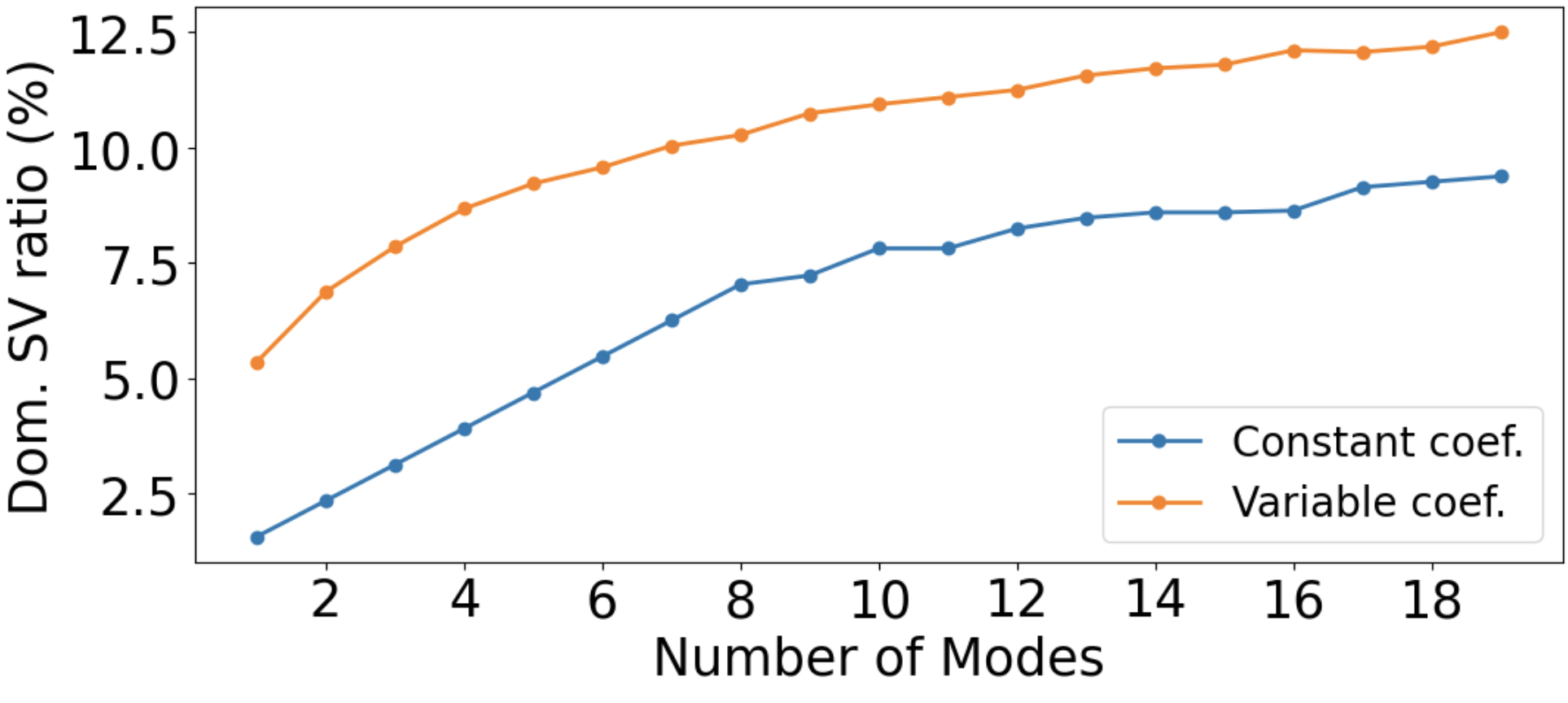}\\
    \qquad\; (a)&\qquad\; (b)
    \end{tabular}
    \caption{Singular value distribution of the solution matrix $u_{jk}=u(x_j,t_k)$ with random initial functions with a varying number of modes for  (a) transport equations (constant coef.~\eqref{EQ: TRANS EX}; variable coef.~\eqref{EQ: TRANS EX2}); (b) heat equations (constant coef.~\eqref{EQ: DIFF EX}; variable coef.~\eqref{EQ: DIFF EX2}). In both figures, the y-axes denote the percentage of dominant singular values ($\lambda_n>1\times10^{-3}$). Each plot is the average of 20 experiments. }
    \label{FIG: DIMENSION_RANDOM}
\end{figure}

\section{PDE identification from a single solution trajectory}\label{SEC: PDE ID}
In this section, we study PDE identification problem based on a  combination  of  candidates  from  a dictionary  of  basic  differential  operators and their functions using a single trajectory. We first focus on the basic question of identifiability and stability. We then propose a data-driven and data-adaptive computational model based on local regression and global consistency for PDE identification with variable coefficients. 

Identifying a differential operator $\cL$ of the form~\eqref{EQ: MODEL} from its solution data $u(x,t)$ can be formulated in the following weak form: choose a filter function $\psi(x)$ and use integration by part (in a weak sense if needed),
\begin{equation}\label{EQ: WEAK}
\begin{aligned}
      \aver{\partial_t u(x,t), \psi(x)} &= -\aver{\sum_{|\alpha|=0}^n p_{\alpha}(x) \partial^{\alpha} u(x,t), \psi(x)} \\
      &= -\aver{\sum_{|\alpha|=0}^n (-1)^{|\alpha|} \partial^{\alpha} (p_{\alpha}(x) \psi(x)), u(x, t)}\\ 
      & =-\aver{\cL^{\ast}\psi(x), u(x, t)},
\end{aligned}
\end{equation}
where $\cL^{\ast}$ is the formal adjoint of $\cL$. The left-hand-side of the weak formulation~\eqref{EQ: WEAK}, denoted by $h(t)$, can be computed from the trajectory data and identification of $\cL$ can be cast in a Galerkin formulation
\begin{equation}\label{EQ: LINEAR}
    \aver{\cL^{\ast}\psi, u(\cdot, t)} = -h(t)\,.
\end{equation}
It shows that the information of a differential operator $\cL$ is projected to the space spanned by snapshots of solution trajectory $u(\cdot, t)$ through its operation on the filter function $\psi(x)$. For example, if $\cL$ is a differential operator with constant coefficients and one chooses $\psi(x;y)=\delta(x-y)$, the PDE identification problem becomes a linear regression problem using the solution and its derivatives sampled at different locations in space and time, which has been the main study in the literature. 

\begin{remark}
Theorem~\ref{THM: RANK} shows that when $\cL$ is a strongly elliptic operator, a single trajectory of the solution stays $\eps$ close in $L^2$ norm to a linear space of dimensions of at most $\cO(|\log\eps|^2)$. It implies the eigenvalues of the compact operator induced by the correlation function between two snapshots $G(t,s)=\int_{\Omega}u(x,t)u(x,s)dx$ has at least an exponential decay as $\lambda_k=\cO(e^{-c\sqrt{k}})$. This implies that, when the Galerkin formulation~\eqref{EQ: WEAK} is discretized and many snapshots along a single trajectory are used as the test functions, the eigenvalues of the resulting linear system also has a fast decay and hence is ill-conditioned, which will affect both accuracy and stability of the identification problem.
\end{remark}

\subsection{PDE identification with constant coefficient}
For PDE identification with constant coefficients, one can transform it into the Fourier domain and show that the underlying differential operator $\cL$ can be identified by one trajectory at two different instants if and only if the solution contains enough Fourier modes.  

Defining the Fourier transform of $u$ with respect to the space variable as $$\widehat{u}(\zeta, t) = (2\pi)^{-d/2}\int_{\Omega} e^{-i\zeta\cdot x} u (x, t)\,dx,$$ the PDE $\partial_t u = \cL u$ is converted to an ODE for each frequency $\zeta\in\mathbb{Z}^d$,
\begin{align}
	\partial_t\widehat{u}(\zeta, t) = -(2\pi)^{-d/2}\sum_{|\alpha|=0 }^np_\alpha(i\zeta)^{\alpha}\widehat{u}(\zeta, t)
\end{align}
whose solution is
\begin{align}
	\widehat{u}(\zeta, t) =\widehat{u}(0,\zeta) \exp\left(-(2\pi)^{-d/2}\sum_{|\alpha|=0}^np_\alpha (i\zeta)^\alpha t\right)\;.
\end{align}
Suppose there is a $\zeta\in\mathbb{R}^d$ such that $\widehat{u}(0,\zeta) \neq 0$, then for any $t_2> t_1>0$, we have
\begin{align}\nonumber
	\frac{\widehat{u}(\zeta, t_2)}{\widehat{u}(\zeta, t_1)} =\exp\left(-(2\pi)^{-d/2}\sum_{|\alpha|~\text{even}}^n p_\alpha (i\zeta)^\alpha (t_2-t_1)\right) \exp\left(-(2\pi)^{-d/2}\sum_{|\alpha|~\text{odd}}^n p_\alpha (i\zeta)^\alpha (t_2-t_1)\right).
\end{align}
By denoting $c_\alpha = p_\alpha i^{|\alpha|}$ when $|\alpha|$ is even, and $c_\alpha = p_\alpha i^{|\alpha|-1}$ when $|\alpha|$ is odd, we associate the  $(\zeta, t)$ pair with the following decoupled system
\begin{align}
		\frac{(2\pi)^{d/2}}{t_2-t_1}\log\left(\left|\frac{\widehat{u}(\zeta, t_2)}{\widehat{u}(\zeta, t_1)}\right|\right)&=  -\sum_{|\alpha|\leq n,~|\alpha|~\text{even}}c_\alpha \zeta^\alpha\label{eq_even},\\
		\frac{(2\pi)^{d/2}}{t_2-t_1}\text{Arg}\left( \frac{\widehat{u}(\zeta, t_2)}{\widehat{u}(\zeta,t_1)}\right) &= -\sum_{|\alpha|\leq n,~|\alpha|~\text{odd}}c_\alpha \zeta^\alpha\label{eq_odd}\,.
\end{align}
It is thus clear that, given a single solution $u(x,t)$ corresponding to an initial data $u(0,t)$, the underlying constant coefficient PDE is \textit{identifiable}, i.e., there exists a unique set of parameters $p_\alpha$ such that $\partial_t u = -\cL u$ if and only if~\eqref{eq_even} and~\eqref{eq_odd} admit unique solutions for $c_\alpha$, which are coefficients of two polynomials. If $t_2-t_1>0$ is small enough, the phase ambiguity in ~\eqref{eq_odd} is removed. Hence one can apply standard polynomial regression results to this problem in the spectral domain.  

\begin{theorem}\label{THM: INTERPOLATION}
    Let $Q = \{\zeta\in\bbZ^d: \widehat{u}_0(\zeta)\neq 0\}$, then if
    $$|Q|\ge\max\left(  \sum_{k=0}^{\lfloor\frac{n}{2}\rfloor}\binom{2k+d-1}{d-1}, \sum_{k=0}^{\lfloor\frac{n-1}{2}\rfloor}\binom{2k+d}{d-1}\right)$$
    and $Q$ is not located on an algebraic polynomial hypersurface of degree $\le n$ {consists of only even order terms or odd order terms}, then the parameters $p_{\alpha}$ are uniquely determined by the solution at two instants $u(x,t_2), u(x,t_1)$ if $|t_2-t_1|$ is small enough.
\end{theorem}
\begin{proof}

Choose Fourier modes $\zeta_k\in Q, k=1, 2, \ldots, K\ge \max\left(  \sum_{k=0}^{\lfloor\frac{n}{2}\rfloor}\binom{2k+d-1}{d-1}, \sum_{k=0}^{\lfloor\frac{n-1}{2}\rfloor}\binom{2k+d}{d-1}\right)$ and $t_2>t_1\ge 0$, \eqref{eq_even} and~\eqref{eq_odd} imply
\begin{align}
 (2\pi)^{d/2} y_e &= -A_e c_e\;,\quad c_e^T = (c_0,c_2,\dots,c_{2\lfloor n/2\rfloor})\,,\\
 (2\pi)^{d/2} y_o &= -A_o c_o\;,\quad c_o^T = (c_1,c_3,\dots,c_{2\lfloor (n-1)/2\rfloor+1})\,,
\end{align}
respectively, where 
\[
(y_e)_k=\frac{1}{t_2-t_1}\log(\left|\widehat{u}(\zeta_k, t_2,)/\widehat{u}(\zeta_k, t_1)\right|), \quad (A_e)_{k\alpha} = \zeta_k^{\alpha}, \quad
 \alpha \mbox{ even },|\alpha|\le n
\]
and 
\[
(y_0)_k=\frac{1}{t_2-t_1}\text{Arg}(\widehat{u}(\zeta_k, t_2)/\widehat{u}(\zeta_k, t_1)), \quad (A_o)_{k\alpha} = \zeta_k^{\alpha}, \quad
 \alpha \mbox{ odd },|\alpha|\le n
\]
By the assumption, $A_e$ and $A_o$ are both of full ranks. Hence $p_{\alpha}$ can be determined uniquely.
\end{proof}
\begin{remark}
Theorem~\ref{THM:4.3} in Section~\ref{sec:local-regression} states that actually there exist solution data on a local patch (in space and time) that can identify PDE with constant coefficients by local regression if the solution has enough Fourier modes. 
\end{remark}
Determining a differential operator $\cL$ in the spectral domain requires observing the solution globally in space. In reality, it may be more practical to observe the solution merely by local sensors. In other words, one can approximate the filter function in~\eqref{EQ: WEAK} by the delta function sampled at certain points $(x_k, t_k), k=1, 2, \ldots, K$ in space and time and identify the PDE through the following least-squares problem
\begin{equation}\label{EQ: LS}
    \argmin_{\mathbf{p}} \|\mathbf{F}\mathbf{p} - \mathbf{u}_t\|_2^2
\end{equation}
where $\mathbf{F}$ is called the \textit{feature matrix} defined by a set of basic  partial differential operators, a linear combination of which can form $\cL$, acted on the observed solution at sampled locations $(t_k,x_k)$, $\mathbf{p}$ represents the unknown coefficient vector, and $$\mathbf{u}_t=[ u_t(x_1,t_1), u_t(x_2,t_2), \cdots, u_t(x_K,t_K)]^T.$$ For example, in one dimension $d=1$, the feature matrix 
$$
\mathbf{F} = \begin{bmatrix}
u(x_1, t_1)&u_x(x_1,t_1)&u_{xx}(x_1,t_1)&\cdots&u_{x}^{(n)}(x_1,t_1)\\
%u(t_1,x_2)&u_x(t_1,x_2)&u_{xx}(t_1,x_2)&\cdots&u_{x}^n(t_1,x_2)\\
\vdots&\vdots&\vdots&\vdots&\vdots\\
u(x_K, t_K)&u_x(x_K,t_K)&u_{xx}(x_K,t_K)&\cdots&u_{x}^{(n)}(x_K,t_K)
%\vdots&\vdots&\vdots&\vdots&\ddots\\
%u(t_J,x_1)&u_x(t_J,x_1)&u_{xx}(t_J,x_1)&\cdots&u_{x}^n(t_J,x_1)\\
%\vdots&\vdots&\vdots&\vdots&\ddots\\
%u(t_J,x_N)&u_x(t_J,x_N)&u_{xx}(t_J,x_2)&\cdots&u_{x}^n(t_J,x_N)
\end{bmatrix}
$$

Assume the solution and its derivatives are sampled on an equally spaced grid $x_k$, $k=1,\dots, K$, at a single observation time $t_k \equiv \tau$, 
$$
\mathbf{F} = \begin{bmatrix}
u(x_1, \tau)&u_x(x_1,\tau)&u_{xx}(x_1,\tau)&\cdots&u_{x}^{(n)}(x_1,\tau)\\
%u(t_1,x_2)&u_x(t_1,x_2)&u_{xx}(t_1,x_2)&\cdots&u_{x}^n(t_1,x_2)\\
\vdots&\vdots&\vdots&\vdots&\vdots\\
u(x_K, \tau)&u_x(x_K,\tau)&u_{xx}(x_K,\tau)&\cdots&u_{x}^{(n)}(x_K,\tau)
%\vdots&\vdots&\vdots&\vdots&\ddots\\
%u(t_J,x_1)&u_x(t_J,x_1)&u_{xx}(t_J,x_1)&\cdots&u_{x}^n(t_J,x_1)\\
%\vdots&\vdots&\vdots&\vdots&\ddots\\
%u(t_J,x_N)&u_x(t_J,x_N)&u_{xx}(t_J,x_2)&\cdots&u_{x}^n(t_J,x_N)
\end{bmatrix},
$$
Since the discrete Fourier transform (DFT) matrix is unitary, the singular values of $\mathbf{F}$ are identical to its discrete Fourier transform
\begin{align}
\widehat{\mathbf{F}} &= \begin{bmatrix}
\widehat{u}(\zeta_1,0)W(\zeta_1,\tau)&(i\zeta_1)\widehat{u}(\zeta_1,0))W(\zeta_1,\tau)&\cdots&(i\zeta_1)^n\widehat{u}(\zeta_1,0))W(\zeta_1,\tau)\\
\vdots&\vdots&\vdots&\vdots\\
\widehat{u}(\zeta_K,0)W(\zeta_K,\tau)&(i\zeta_K)\widehat{u}(\zeta_K,0)W(\zeta_K,\tau)&\cdots&(i\zeta_K)^n\widehat{u}(\zeta_K,0)W(\zeta_K,\tau)
%\vdots&\vdots&\vdots&\ddots\\
%\widehat{u}(0,\zeta_1)W(t_J,\zeta_1)&(i\zeta_1)\widehat{u}(0,\zeta_N)W(t_J,\zeta_1)&\cdots&(i\zeta_1)^n\widehat{u}(0,\zeta_1)W(t_J,\zeta_1)\\
%\vdots&\vdots&\vdots&\vdots\\
%\widehat{u}(0,\zeta_N)W(t_J,\zeta_N)&(i\zeta_N)\widehat{u}(0,\zeta_N)W(t_J,\zeta_N)&\cdots&(i\zeta_N)^n\widehat{u}(0,\zeta_N)W(t_J,\zeta_N)
\end{bmatrix}
\end{align}
where $\widehat{u}$ is the discrete Fourier transform of $u$ and $$W(\zeta_k,\tau)= \exp\left(-(2\pi)^{-1/2}\sum_{\alpha=0}^{n}p_\alpha(i\zeta_k)^\alpha \tau\right), \quad \zeta_k = k,\quad  k=1,2,\dots, K.$$ The matrix $\widehat{\mathbf{F}}$ can be factorized as $$\widehat{\mathbf{F}} = \mathbf{\Lambda} \mathbf{V},$$ where $\mathbf{\Lambda}$ is a $K\times K$ diagonal matrix with $\mathbf{\Lambda}_{kk}=\widehat{u}(\zeta_k,0)W(\zeta_k,\tau)$ and $\mathbf{V}$ is the Vandermonde matrix with $\mathbf{V}_{kj} = (i\zeta_k)^{j}$, $k=1,2,\dots, K$, $j=0,1, 2,\dots, n$.
We can see that the PDE identification problem using pointwise information is a little different from a polynomial regression problem. In addition to the Vandermonde matrix $\mathbf{V}$, we see that initial/input data $u_0$ and sampling can also affect the conditioning of $\mathbf{F}$. 
Let $\overline{\lambda}=\max_{1\le k\le K} |\mathbf{\Lambda}_{kk}|$,  $\underline{\lambda}=\min_{1\le k\le K} |\mathbf{\Lambda}_{kk}|$ and $\overline{\sigma}=\max_{0\le j\le n} |\sigma_j|, \underline{\sigma}=\min_{0\le j\le n} |\sigma_j|$, where $\sigma_j$  are eigenvalues of the corresponding Vandermonde matrix. 
Assume $K=n+1$, denote $L_j(\zeta)=\sum_{k=0}^{n}=\xi_{jk}(i\zeta)^k$ to be the Lagrange basis polynomials, i.e., $L_j(\zeta_k)=\delta_{jk}$ and $\mathbf{\xi}_j=[\xi_{j0}, \xi_{j1}, \ldots, \xi_{jn}]^T$.  We have $\mathbf{V}\mathbf{\xi}_j=\mathbf{e}_j$, where $\mathbf{e}_j$ is the canonical basis in $\mathbb{R}^{n+1}$, and $\overline{\sigma}^{-1}\le \|\mathbf{\xi}_j\|\le \underline{\sigma}^{-1}$. On the other hand, let $\omega_j, j=0,1, \ldots, n$ be the eigenvalues of $\widehat{\mathbf{F}}$ and $\overline{\omega}=\max_{0\le j\le n}|\omega_j|$, $\underline{\omega}=\min_{0\le j\le n} |\omega_j|$. We have $\underline{\lambda} =\min_{0\le j\le n} \|\widehat{\mathbf{F}}\mathbf{\xi}_j\|,
~\overline{\lambda}=\max_{0\le j\le n} \|\widehat{\mathbf{F}}\mathbf{\xi}_j\| $, which implies $\overline{\omega} \ge \overline{\lambda} \underline{\sigma}$ and $\underline{\omega} \le \underline{\lambda} \overline{\sigma}$, and hence $\frac{\overline{\omega} }{\underline{\omega}}\ge \frac{\overline{\lambda} \underline{\sigma}}{\underline{\lambda}\overline{\sigma}}$.
For example, $\mathbf{F}$ may become ill-conditioned if 
\begin{enumerate}
    \item Fourier modes with small amplitudes in the initial/input data are involved;
    \item  the differential operator $\cL$ is elliptic and the solution is sampled at a large time (due to the exponential decay of $|W(\zeta_k,\tau)|$ in time).
\end{enumerate}
In general, using the solution data at all points on a rectangular grid in space and time, which most existing methods are based on, may be too costly or not practical in real applications. In particular, for PDEs with constant coefficients, it is unnecessary. 
Moreover, it may not be a good strategy for PDE identification even if the data are available. For example, at certain sampling locations, the solution may be degenerate, e.g., (nearly) zero in a neighborhood, which is then sensitive to noise, or becomes singular, which can lead to large numerical errors. However, we do not know what type of PDE or initial data \emph{a priori}, so the sampling strategy and PDE identification method should be data-driven and data-adapted. 

\subsection{PDE identification with variable coefficients using a single trajectory}
For the identification of time-dependent PDEs with coefficients varying in space, the unknowns are (coefficient) functions, which means 1) the regression problems at different spatial locations are different; 2) the variations of the coefficients are intertwined with the solution in both frequency and physical domains.

 In the following, we will start with an identifiability and stability study. Then we propose a computational model for PDE identification with variable coefficients that enforces both local regression and global pattern consistency. The main goal is to identify a consistent differential operator that is built from as few terms as possible from the library that can fit observed data well locally by using different linear combinations at different locations. Once the PDE type is determined, a more accurate estimation of coefficients can be achieved by independent local regression and/or appropriate regularization.

\subsection{Identifiability with a single trajectory}
In this section, we focus on the question that whether or not a single solution trajectory $u(x, t)$ can determine the underlying differential operator from a given library, i.e. identifiability of the differential operator $\cL$. In the sequel, we first introduce the general statement in parallel to the identifiability statement for PDE with constant coefficients, which states if there are enough Fourier modes in the initial data and one can observe a single trajectory at two different instants, one can recover those constant coefficients. For PDEs with variable coefficients, one has to recover these unknown functions. Hence more Fourier modes in the initial data and more snapshots along the trajectory depending on the order of the differential operators and space dimensions are needed for identifiability. 
\begin{theorem}
Let $m = \binom{n+d}{d}$. For any given $x\in\Omega$, the parameters $p_{\alpha}(x)$ can be recovered if and only if one can find $m$ instants $t_1,\dots, t_m$ that the matrix $A = (A_{k,\alpha} ) $ is non-singular where $A_{k,\alpha} := \partial^{\alpha} u(x, t_k) $.
\end{theorem}
\begin{proof}
This is directly from the following linear system of $p_{\alpha}(x)$,
\begin{equation}
    \sum_{|\alpha|=0}^n p_{\alpha}(x) \partial^{\alpha} u(x, t_k) = -u_t(x, t_k),\quad k = 1,2,\dots, m.
\end{equation}
\end{proof}

Consider the limiting case that $t_1,\dots, t_m\to 0$. One can then take  $k$-th derivatives of the solution in time at $t=0$. The above lemma becomes the requirement that the matrix $S:= (S_{k, \alpha}), k=1, 2, \ldots, m$ with $S_{k, \alpha} := \partial^{\alpha}\cL^{k-1} u_0(x) $ is non-singular almost everywhere. In the following, we show that if the initial condition is randomly generated, then $S$ is almost surely non-singular if the parameters $p_{\alpha}$ are sufficiently smooth. In the following, we denote the diagonal multi-indices set $D_n$ by
\begin{equation}\nonumber
    D_n : = \{\alpha = (\alpha_1, \dots, \alpha_d)\mid |\alpha|=n\text{ and there exists  $1\le j\le d$ such that }\alpha_j = n\}.
\end{equation}
\begin{theorem}
Assume the coefficients $p_{\alpha}\in C^{mn}(\Omega)$ and let the initial condition $u_0$ be generated by 
\begin{equation}
    u_0(x) = \sum_{j=1}^r w_j e^{i\zeta_j \cdot x}
\end{equation}
where $r > \binom{mn + d}{d}$ and the vectors $\zeta_j\in\bbZ^d$ are not on an algebraic hypersurface of degree $mn$ and the weights $w_j$ are random variables that $w_j\sim \cU[a_j, b_j]$ with $a_j < b_j$. Denote $\bbP$ the induced probability measure on $\Omega\times \prod_{j=1}^m[a_j, b_j]$. If $\sum_{\alpha\in D_n} |p_{\alpha}(x)|^2\neq 0$ almost everywhere in $\Omega$, then the matrix $S$ is non-singular $\bbP$-almost surely.
\end{theorem}
\begin{proof}
Let $x'\in\Omega$ be a fixed point that $\sum_{D_n} |p_{\alpha}(x ')|^2\neq 0$. Notice that the probability $\bbP[S \text{ is non-singular}] = 1 - \bbP[\det S = 0]$, since $S_{k, \alpha}(x ') =  \sum_{j=1}^r w_j \partial^{\alpha} \cL^{k-1} e^{i\zeta_j\cdot x} |_{x=x'}$, the determinant $\det(S)$ can be written in the following form
\begin{equation}\label{EQ: DET EQUATION}
    \sum_{|\beta| = m} F_{\beta}(\{\partial^{\gamma} p_{\alpha}(x ')\}; \{ \zeta_j \}) e^{i(\beta_1\zeta_1+\cdots +\beta_r\zeta_r)\cdot x '}w_1^{\beta_1} w_2^{\beta_2}\cdots w_{r}^{\beta_r} = 0,
\end{equation}
where $\beta = (\beta_1,\beta_2,\dots, \beta_r)$ and each $F_{\beta}$ is a polynomial of $\partial^{\gamma} p_{\alpha}(x)$ with $|\alpha|\le n$, $|\gamma|\le (m-1)n$, and the degree of the polynomial is at most $m-1$, each coefficient of the polynomial is uniquely determined by a polynomial in terms of $\zeta_j$, $1\le j\le r$.

It is clear that the set of $w_i, i=1, 2,\ldots, r$ that satisfy \eqref{EQ: DET EQUATION} has measure zero in $\prod_{j=1}^r[a_j, b_j]$ unless all of the coefficients $F_{\beta} = 0$ at $x'\in \Omega$. If the later is true, it means $\det(S) = 0$ at $x'\in \Omega$ for any $w_j\in [a_j, b_j]$, $j=1,\dots, r$. Let the matrix $S^j:= (S_{k,\alpha}^j)$, where $S_{k,\alpha}^j = \partial^{\alpha}\cL^{k-1} e^{i\zeta_j\cdot x}|_{x=x'}$ and we denote the linear functional $f_j:\bbC^{m\times m}\to \bbC$, $j=1,2,\dots, r$ by 
\begin{equation}
    f_j(X):= \mathrm{trace}(X S^j).
\end{equation}
In the following, we show that $\bigcap_{j=1}^r \ker f_j \neq \{\bzero\}$. By the Lemma 3.9 of~\cite{rudin1991functional}, if $ \bigcap_{j=1}^r \ker f_j = \{\bzero\}$, then any linear functional $f$ on $\bbC^{m\times m}$ can be written as linear combination
\begin{equation}
   f(X) = \sum_{j=1}^r w_j \mathrm{trace}(X S^j ) = \mathrm{trace} (X \sum_{j=1}^r w_j S^j).
\end{equation}
However, if we take a full rank matrix $B\in\bbC^{m\times m}$ and define $\tilde{f}(X) = \mathrm{trace}(X B)$, then $B$ can be represented by $B = \sum_{j=1}^r w_j S^j$, which violates the condition that $\det (\sum_{j=1}^r w_j S^j) = 0$.
Hence there exists a nonzero matrix $\bc= (c_{k,\alpha})\in \bbC^{m\times m}$ such that
\begin{equation}\label{EQ: MODES}
     \sum_{|\alpha|\le n}\sum_{k=1}^m c_{k, \alpha} \partial^{\alpha} \cL^{k-1} e^{i\zeta_j\cdot x}|_{x=x'} = 0,\quad \quad 1\le j\le r .
\end{equation}
The differential operator $\cL' :=  \sum_{|\alpha|\le n}\sum_{k=1}^m c_{k,\alpha}\partial^{\alpha} \cL^{k-1}$ is uniquely determined by the derivatives $\partial^{\gamma} p_{\alpha}$ and has the order of at most $mn$. Without loss of generality, we denote $\cL'$ by 
\begin{equation}
    \cL' := \sum_{|\gamma|=0}^{mn} q_{\gamma}(x) \partial^{\gamma}.
\end{equation} 
We first show $\cL'$ is non-trivial. Let $1\le m'\le m$ be the largest index that $c_{m',\alpha'}\neq 0$ for some $\alpha'$, then according to the assumption, almost everywhere in $\Omega$, at least one of the leading order terms  $c_{m',\alpha'}p_{\alpha}^{m'-1}(x)\partial^{(m'-1)\alpha+\alpha'}$ for certain $\alpha\in D_n$ in $c_{m',\alpha'}\partial^{\alpha'}\cL^{m'-1}$ is nonzero. 
In the next, we denote $(i\zeta)^{\gamma} := \prod_{k=1}^d (i\zeta_{k})^{\gamma_k}$, where $\zeta = (\zeta_{1},\dots, \zeta_{d})$ and $\gamma = (\gamma_1,\dots, \gamma_d)$. Since $r > \binom{mn+d}{d}$ and the vectors $\zeta_j\in \bbZ^d$, $1\le j\le r$  are not located on an algebraic polynomial hypersurface of degree $\le mn$, then the equation
\begin{equation}
   \cL' e^{i\zeta_j\cdot x}|_{x=x'} =  \left( \sum_{|\gamma|=0}^{mn} q_{\gamma}(x') (i\zeta_j)^{\gamma} \right) e^{i\zeta_j\cdot x'}= 0,\quad 1\le j\le r
\end{equation}
implies $\sum_{|\gamma|=0}^{mn} q_{\gamma}(x ') (i\zeta_j)^{\gamma}=0$ which gives a contradiction.
\end{proof}

\subsubsection{Ergodic Orbits}
In the following, we consider the solution to~\eqref{EQ: HYPERBOLIC} with infinite observation time $T=\infty$. The previous dimensionality analysis of the solution space would not work due to the non-compactness. 
For the sake of simplicity, we assume $\bc(x)$ is measure-preserving and has no singular points, that is
\begin{equation}
    \nabla \cdot \bc(x) = 0, \quad |\bc(x)|\neq 0.
\end{equation}
For dimension $d=2$, the ergodic properties of the measure-preserving dynamic system
\begin{equation}\label{EQ: DYNAMIC}
    \dot{X}(t) = \bc(X),\quad X(0) = x_0
\end{equation}
on torus have been studied by~\cite{sternberg1957differential, saito1951measure, akutowicz1958ergodic}. Let $$(a_1, a_2) = \frac{1}{|\Omega|}\int_{\Omega} \bc(x) dx, $$
it has been proved that if $a_1/a_2$ is irrational, then the flow $X(t)$
is ergodic and can be regarded as a rectilinear flow on a two-dimensional Euclidean torus by some suitable choice of coordinate change. The dynamics on a high dimensional torus are studied by~\cite{saito1965dynamical,kozlov2007dynamical}.
\begin{theorem}[\cite{saito1951measure}] For $d=2$, every orbit of~\eqref{EQ: DYNAMIC} is ergodic if and only if $a_1, a_2$ are linearly independent with respect to integral coefficients. 
\end{theorem}
\begin{theorem}[\cite{saito1965dynamical}]
For $d\ge 3$ and assume every orbit of~\eqref{EQ: DYNAMIC} is Lyapunov stable in both directions in time,
then every orbit is ergodic if and only if $a_1, a_2,\dots, a_d$ are linearly independent with respect to integral coefficients. 
\end{theorem}
Using the ergodic property, we can derive the following simple corollary to characterize the solution data by tracking the locally unique value in time. 
\begin{corollary}
Let $u_0\in C^1(\Omega)$, suppose every orbit of~\eqref{EQ: DYNAMIC} is ergodic and there exists a point $x_0\in\Omega$ such that for any $y\in\Omega$, $\mathrm{dist}(y,x_0)\in(0, \delta)$, $u_0(y)\neq u_0(x_0)$, 
then $\bc(x)$ is uniquely determined from $u(x,t)$, $t\in [0,\infty)$.
\end{corollary}
\begin{proof}
Let $X(t)$ be the orbit
\begin{equation}
    \dot{X}(t) = \bc(X),\quad X(0) = x_0.
\end{equation}
Then $\bc(X(t))$ is uniquely determined by tracing the location of the local unique value $u_0(x_0)$. By the ergodic property of $X(t)$ and continuity of $\bc(x)$, for any $z\in\Omega$ one can find a sequence $X(t_j)\to z$, $j\to\infty$, which reconstructs $\bc(z)=\lim_{j\to\infty} \bc(X(t_j))$.
\end{proof}

\subsection{Possible instability for identification of elliptic operator}
In this section, we study the stability issue for PDE identification and use the elliptic operator as an example to show possible instability. First, we show local instability when one has a short observation time for a single trajectory. Suppose each coefficient $p_{\alpha}$ is sufficiently smooth, we consider a small perturbation $p_{\alpha}(x)\to p_{\alpha}(x) + \mu f_{\alpha}(x)$ that $|\mu|\ll 1$, then the Frech\'et derivative $w(x,t) = \partial_{\mu} u(x,t)$ satisfies the equation 
\begin{equation}
\begin{aligned}
        \partial_t w(x, t) &= -\sum_{|\alpha|=0}^n p_{\alpha}(x) \partial^{\alpha} w(x,t) + \sum_{|\alpha|=0}^n f_{\alpha}(x) \partial^{\alpha} u(x, t), \\ w(x, 0) &= 0.    
\end{aligned}
\end{equation}
When $\cL = \sum_{|\alpha|=0}^n p_{\alpha}(x)\partial^{\alpha}$ is a dissipative operator, there exists a constant $c > 0$ such that 
\begin{equation}
\begin{aligned}
        \|w(x, t)\|^2_{L^2(\Omega\times[0,T])} &\le c\|  \sum_{|\alpha|=0}^n f_{\alpha} (x)\partial^{\alpha} u(x, t) \|^2_{L^2(\Omega\times [0,T])} \\
        &= c \sum_{0\le |\alpha|, |\beta|\le n} \int_{\Omega} \left( \int_{0}^T \partial^{\alpha} u(x, t){\partial^{\beta} u(x, t)} dt\right) f_{\alpha}(x) {f_{\beta}(x)} dx \,.
\end{aligned}
\end{equation}
Define the functions $K_{\alpha\beta}(x)$ by 
\begin{equation}
    K_{\alpha\beta}(x) := \int_{0}^T \partial^{\alpha} u(x, t) {\partial^{\beta} u(x, t)} dt
\end{equation}
and denote the normalized set $F = \{\{ f_{\alpha} \}_{|\alpha|=0}^n\subset C^n(\Omega;\bbR)\mid \sum_{\alpha}\|f_{\alpha}\|^2_{L^2(\Omega)}=1 \}$, then the local instability amounts to find the minimum
\begin{equation}
    \min_{f_{\alpha}\in F} \sum_{0\le |\alpha|, |\beta|\le n} \int_{\Omega} K_{\alpha\beta}(x) f_{\alpha}(x){f_{\beta}(x)} dx 
\end{equation}
where $K(x) = (K_{\alpha\beta}(x))$ is a positive definite matrix for each $x\in\Omega$. Given any fixed $x\in\Omega$,
\begin{equation}
    \min_{f_{\alpha}(x)}  \sum_{0\le |\alpha|, |\beta| \le n} K_{\alpha\beta}(x) f_{\alpha}(x){f_{\beta}(x)} = {\lambda}_m (x) \sum_{|\alpha|=0}^n |f_{\alpha}(x)|^2\,,
\end{equation}
where $m = \binom{n+d}{d}$ and $\lambda_1\ge \lambda_2\ge \cdots \ge \lambda_m \ge 0$ are the eigenvalues of $K_{\alpha\beta}(x)$, the equality holds when $f = (f_{\alpha}(x))$ is the eigenvector of $K(x)$ for $\lambda_m$. Therefore
\begin{equation}\label{EQ: W}
    \min_{f_{\alpha}\in F} \|w\|^2_{L^2(\Omega\times[0,T])} \le c \int_{\Omega} {\lambda}_m(x) \sum_{|\alpha|=0}^n |f_{\alpha}(x)|^2 dx = c \int_{\Omega} {\lambda}_m(x) dx\,.
\end{equation}
It should be noted that if $\partial^{\alpha} u$ are continuous, then $K_{\alpha\beta}(x)$ varies continuously, therefore ${\lambda}_m(x)$ is also a continuous function over $\Omega$. On the other hand,
use Taylor expansion for short time that $T\ll 1$, the matrix entries of $K$ are
\begin{equation}\label{EQ: K}
\begin{aligned}
K_{\alpha\beta}(x) &= \sum_{k=0}^{m-1} \frac{T^{k+1}}{(k+1)!}(-1)^k \left( \sum_{l=0}^k \binom{k}{l} \partial^{\alpha}\cL^l u_0(x) \partial^{\beta} \cL^{k-l} u_0(x) \right) + \cO(T^{m+1}) \\
&= \sum_{l=0}^{m-1} \partial^{\alpha}\cL^l u_0(x) \left( \sum_{k=l}^{m-1}  \frac{T^{k+1}}{(k+1)!}(-1)^k \binom{k}{l} \partial^{\beta} \cL^{k-l} u_0(x) \right) + \cO(T^{m+1}) \\
&=  \sum_{l=0}^{m-2} \partial^{\alpha}\cL^l u_0(x) \left( \sum_{k=l}^{m-1}  \frac{T^{k+1}}{(k+1)!}(-1)^k \binom{k}{l} \partial^{\beta} \cL^{k-l} u_0(x) \right) \\
&\quad +  \frac{T^{m}}{m!}(-1)^{m-1}  \partial^{\alpha}\cL^{m-1} u_0(x) \partial^{\beta} u_0(x) + \cO(T^{m+1}).
\end{aligned}
\end{equation}
For each $0\le l\le m-2$ in the summation, the matrix formed by the entry
\begin{equation}
    K_{\alpha\beta}^l(x) := \partial^{\alpha}\cL^l u_0(x) \left( \sum_{k=l}^{m-2}  \frac{T^{k+1}}{(k+1)!}(-1)^k \binom{k}{l} \partial^{\beta} \cL^{k-l} u_0(x) \right) 
\end{equation}
is rank one, therefore the summation of the terms $0\le l\le m-2$ of~\eqref{EQ: K} is at most rank $m-1$, then by the Theorem VI.3.3 of~\cite{bhatia2013matrix},  
\begin{equation}
    \lambda_{m}(x) \le \frac{T^m}{m!}\|M(x)\| + \cO(T^{m+1}),
\end{equation}
where $M(x)$ is the rank one matrix defined by $M_{\alpha\beta} (x):= \partial^{\alpha}\cL^{m-1} u_0(x) \partial^{\beta} u_0(x) $, then the norm $\|M(x)\|$ is bounded by 
\begin{equation}
    \|M(x)\|\le \left(\sum_{|\alpha|=0}^n |\partial^{\alpha} u_0(x)|^2\right)^{1/2} \left(\sum_{|\alpha|=0}^n |\partial^{\alpha}\cL^{m-1} u_0(x)|^2\right)^{1/2}.
\end{equation}
Combine above estimate with~\eqref{EQ: W} and Cauchy-Schwartz inequality, we obtain 
\begin{equation}\label{EQ: ESTIMATE INST}
     \min_{f_{\alpha}\in F} \|w\|^2_{L^2(\Omega\times[0,T])} \le \frac{cT^{m}}{m!}  \|\cL^{m-1} u_0\|_{W^{n,2}(\Omega)} \| u_0\|_{W^{n,2}(\Omega)}  + \cO(T^{m+1}).
\end{equation}
\begin{remark}
In fact, one can show $\lambda_k \le \cO(T^{k})$, $1\le k\le m$ by following the similar argument. Furthermore, from the estimate~\eqref{EQ: ESTIMATE INST}, one can observe that $\cL^{m-1} u_0$ plays a role. Indeed, if $\cL^{m-1} u_0 = 0$, then the solution is simply
\begin{equation}
    u(x, t) = \sum_{l=0}^{m-1} (-1)^l\frac{t^l}{l!} \cL^l u_0(x)
\end{equation}
which means the matrix formed by $A_{\alpha, k}:=\partial^{\alpha} u(x, t_k)$, $k=1,2,\dots, m$,
\begin{equation}
    \partial^{\alpha} u(x, t_k) = \sum_{l=0}^{m-1} (-1)^l\frac{t_k^l}{l!} \partial^{\alpha}\cL^l u_0(x)
\end{equation}
has a rank at most $m-1$. Hence the identification is non-unique.
\end{remark}

Now we show instability for high-frequency perturbations for elliptic operators. For the sake of simplicity, in the following we let $\cL = -\sum_{\alpha = 0}^n p_{\alpha}(x) \partial^{\alpha}$ be an elliptic differential operator in 1D where $p_{\alpha}(x)$ are constant functions. Consider the perturbation  $\widetilde{\cL} = \cL- \delta e^{i q  \cdot x } \partial^{\alpha'}$, where $q \in \bbZ$ is some frequency and $\alpha'$ is certain index. Denote $u$ and $\tilde u$ the solutions to~\eqref{EQ: MODEL} with respect to the operators $\cL$ and $\widetilde\cL$, respectively. Then we have
\begin{equation}
    u(x, t) = \sum_{k\in\bbZ }\phi_k(t) e^{ik\cdot x},
\end{equation}
where $\lambda_k = -\sum_{\alpha=0}^n p_{\alpha} (i k)^{\alpha}$, $\phi_k(t) = c_k e^{-\lambda_k t}$, and the constant $c_k$ is the Fourier coefficient of $u_0$ . In the following, we study the instability of identification for large $|q|$ under the assumptions that $\Re \lambda_k \ge C_1 \aver{k}^{n}$ and $|c_k|\le C_2 \aver{k}^{-\beta}$, $\beta > \frac{1}{2}$, where $C_1, C_2$ are two positive constants and $\aver{k}:= (1+k^2)^{1/2}$. Using Fourier transform, we may write the solution $\tilde{u}$ in the form
\begin{equation}
    \tilde{u}(x, t) = \sum_{k\in\bbZ}  \widetilde \phi_k(t) e^{ik\cdot x},
\end{equation}
where $ \widetilde  \phi_k(t)$ satisfies the coupled ODE, 
\begin{equation}
   \frac{d}{dt} \widetilde  \phi_k(t) = -\lambda_k \widetilde  \phi_k(t) +  \delta  (i(k-q))^{\alpha'}  \widetilde  \phi_{k-q}(t).
\end{equation}
The solution can be written as 
\begin{equation}
    \widetilde \phi_k(t) = e^{- \lambda_k t} \left(c_k + \int_0^t e^{\lambda_k s}  \delta  (i(k-q))^{\alpha'}  \widetilde  \phi_{k-q}(s) ds \right),
\end{equation}
since $c_k =  \widetilde \phi_k(0)$. Therefore using Cauchy-Schwartz inequality, we obtain the estimate
\begin{equation}\label{EQ: INS BOUND}
\begin{aligned}
    \int_0^T |k|^{2\alpha'} |\widetilde \phi_k(t)|^2 dt& \le 2 |c_k|^2  \int_0^T |k|^{2\alpha'} e^{-2 \Re \lambda_k t} dt  \\ &\quad  + 2\delta^2 |k|^{2\alpha'} \int_0^T  e^{-2 \Re \lambda_k t} \int_0^t e^{2\Re \lambda_k s} ds \int_0^t |k-q|^{2\alpha'} |\widetilde\phi_{k-q}(s)|^2 ds dt  \\
    &\le \frac{|k|^{2\alpha'}}{\Re \lambda_k} (C_2)^2 \aver{k}^{-2\beta}+  \frac{|k|^{2\alpha'} \delta^2 T}{\Re \lambda_k }  \int_0^T |k-q|^{2\alpha'} |\widetilde\phi_{k-q}(s)|^2 ds \,.
\end{aligned}
\end{equation}
\begin{lemma}\label{LEM: INS INEQ}
Suppose $0\le \alpha' \le \frac{n}{2}$, $\beta > \frac{1}{2}$, and $\delta^2 T < \gamma C_1$ for some $\gamma \in (0,1)$, then 
\begin{equation}
    \sum_{k\in\bbZ} \int_0^T |k|^{2\alpha'} |\widetilde \phi_k(t)|^2 dt \le \frac{(C_2)^2}{C_1}\frac{1}{1-\gamma} \sum_{k\in\bbZ} \aver{k}^{2\alpha' -2\beta - n}.
\end{equation}
If $kq \le 0$, we would have 
\begin{equation}
     \int_0^T |k|^{2\alpha'} |\widetilde \phi_k(t)|^2 dt \le \frac{1}{1-\gamma} \aver{k}^{2\alpha' -2\beta -n}.
\end{equation}
\end{lemma}
\begin{proof}
The first inequality can be easily proved by summation of~\eqref{EQ: INS BOUND} on both sides over $k\in\bbZ$. For the second inequality, by iteratively using~\eqref{EQ: INS BOUND}, we find that 
\begin{equation}
     \int_0^T |k|^{2\alpha'} |\widetilde \phi_k(t)|^2 d t \le \frac{(C_2)^2}{C_1} \sum_{l=0}^{\infty} \gamma^l \aver{k - l q}^{2\alpha' - 2\beta - n}.
\end{equation}
If $k q \le 0$, we would have $|k - l q|\ge |k|$ for all $l\ge 0$, then 
\begin{equation}
    \sum_{l=0}^{\infty} \gamma^l \aver{k - l q}^{2\alpha' - 2\beta - n} \le \sum_{l=0}^{\infty} \gamma^l \aver{k}^{2\alpha'  - 2\beta - n} = \aver{k}^{2\alpha'-2\beta-n} \frac{1}{1-\gamma}.
\end{equation}
\end{proof}
\begin{theorem}
Assuming the conditions of Lemma~\ref{LEM: INS INEQ}, the following inequality holds for any $q\in\bbZ$,
\begin{equation}\nonumber
    \|u(x, t) - \tilde{u}(x ,t)\|_{L^2(\Omega\times [0,T])}^2 \le \frac{\gamma}{2(1-\gamma)} \left( \aver{q}^{-n} (2^n   + \frac{(C_2)^2}{C_1})C_3 + \aver{q}^{2\alpha' - 2\beta - n} 2^{n+2\beta - 2\alpha'}  C_4 \right),
\end{equation}
where $C_3 = \sum_{k\in\bbZ} \aver{k}^{2\alpha' - 2 \beta - n}$ and $C_4 =  \sum_{k\in\bbZ} \aver{k}^{- n}$. 
\end{theorem}
\begin{proof}
Let $w_k = \phi_k(t) - \widetilde \phi_k(t)$, then 
\begin{equation}
    w_k(t) = e^{-\lambda_k t} \delta (i(k-q))^{\alpha'} \int_0^t e^{\lambda_k s} \widetilde \phi_{k-q}(s) ds 
\end{equation}
and note that $\delta^2 T \le \gamma C_1$, 
\begin{equation}
\begin{aligned}
   & \|u(x, t) - \tilde{u}(x ,t)\|_{L^2(\Omega\times [0,T])}^2 = \sum_{k\in\bbZ } \int_0^T |w_k(t)|^2 dt \\
    &\le \sum_{k\in\bbZ}   \frac{\gamma C_1}{2\Re \lambda_k} |k-q|^{2\alpha'} \int_0^T |\widetilde\phi_{k-q}(s)|^2 ds  \\
    &\le \sum_{k\in\bbZ} \frac{\gamma}{2\aver{k + q}^{n}}|k|^{2\alpha'} \int_0^T |\widetilde\phi_{k}(s)|^2 ds \\
    & = \sum_{k q \le 0} \frac{\gamma}{2\aver{k + q}^{n}}|k|^{2\alpha'} \int_0^T |\widetilde\phi_{k}(s)|^2 ds +  \sum_{k q > 0} \frac{\gamma}{2\aver{k + q}^{n}}|k|^{2\alpha'} \int_0^T |\widetilde\phi_{k}(s)|^2 ds \\
    &\le \frac{\gamma}{(1-\gamma)} \sum_{k q \le 0} \frac{1}{2\aver{k + q}^{n}} \aver{k}^{2\alpha' - 2\beta - n} + \frac{\gamma}{2\aver{q}^n}  \sum_{k\in\bbZ} |k|^{2\alpha'} \int_0^T |\widetilde\phi_{k}(s)|^2 ds \\
    & \le \frac{\gamma}{(1-\gamma)} \sum_{k q \le 0} \frac{1}{2\aver{k + q}^{n}} \aver{k}^{2\alpha' - 2\beta - n} +   \frac{(C_2)^2}{C_1}\frac{\gamma}{2(1-\gamma)\aver{q}^n}  \sum_{k\in\bbZ}\aver{k}^{2\alpha' - 2\beta - n} \\
    & \le \frac{\gamma}{2(1-\gamma)} \left( 2^n \aver{q}^{-n} C_3 + 2^{n+2\beta - 2\alpha'} \aver{q}^{2\alpha' - 2\beta - n} C_4 + \aver{q}^{-n} (C_2)^2 C_3/C_1   \right).
\end{aligned}
\end{equation}
The last inequality uses the fact that $|k+q|\le \frac{|q|}{2}$ implies $|k| \ge  \frac{|q|}{2}$ when $kq \le 0$ and \emph{vice versa}, therefore
\begin{equation}
\begin{aligned}
    \sum_{k q \le 0, |k + q|\le |q|/2} \frac{1}{\aver{k + q}^{n}} \aver{k}^{2\alpha' - 2\beta - n} &\le \sum_{k\in\bbZ} \frac{1}{\aver{k+q}^n} \aver{\frac{q}{2}}^{2\alpha' -2 \beta -n}, \\
  \sum_{k q \le 0, |k + q| > |q|/2} \frac{1}{\aver{k + q}^{n}} \aver{k}^{2\alpha' - 2\beta - n} &\le \sum_{k\in\bbZ} \aver{\frac{q}{2}}^{-n} \aver{k}^{2\alpha' - 2\beta -n}.
\end{aligned}
\end{equation}
\end{proof}
\begin{remark}
From the assumption in the above theorem, one can view $\gamma=\cO(\delta^2)$. The result shows that the high-frequency component of perturbation will have a limited impact on the solution. In particular, both the order of the differential operator and the smoothness of the initial data affect the instability estimate. The above analysis can also possibly be extended to high dimensional and smooth coefficient cases.
\end{remark}

\section{Local regression and global consistency enforced PDE identification method}
\label{SEC: algorithm}
For PDE identification in practice, the first and most important task is to robustly identify the PDE type using a combination of the fewest possible candidates from a (relatively large) dictionary with minimal data, e.g. a single trajectory (corresponding to unknown and uncontrollable initial data) with local (in space and time) measurements. When the PDE has variable coefficients, this amounts to different regression problems at different measurement locations. Since a variable coefficient may be degenerate or small at certain locations, e.g., certain components of a velocity field, independent local regressions may render different types of PDEs at different locations. The key idea is to enforce consistency and sparsity across local regressions to enhance both robustness and accuracy. Here we propose a local regression and global consistency enforced strategy, the Consistent and Sparse Local Regression (CaSLR) method, which enforces global consistency of the differential operator and involves as few terms as possible from the dictionary while having the flexibility to fit local measurements as well as possible. Numerical tests show that CaSLR can identify PDE successfully even with a small amount of local data. 

\subsection{Proposed PDE identification method}
Suppose the solution data can be observed/measured by local sensors in a neighborhood at different locations, i.e., local patches,  that can be used to approximate the solution derivatives and their functions corresponding to those terms in the dictionary. Once the measured/computed solution and its derivatives are available at certain locations we propose the following local regression and global consistency-enforced PDE identification method.

Assume that the unknown PDE takes the following form:
\begin{align}
u_t(x,t) = \sum_{k=1}^Kc_{k}(x,t)f_k(x,t),\label{eq_PDEform}
\end{align}
where $\mathcal{F} = \{f_k:\Omega\times [0,T]\to\mathbb{R}\}_{k=1}^K$ is a dictionary of features that contains partial derivatives of $u$ with respect to the space (e.g., $u_x,u_{xxx}$, etc.), functions of these terms (e.g., $uu_x$ and $u^2$, $\sin (u)$ etc.); and  $c_k:\Omega\times [0,T]\to\mathbb{R}$, $k=1,2,\dots, K$ are the respective coefficients.  We assume that the dictionary is rich enough that it is over-complete, i.e., the underlying PDE is expressed as~\eqref{eq_PDEform} when at least one of the feature's coefficients is null. One also has to assume the measurements can resolve local variation of the coefficients, i.e., one can approximate the PDE coefficients by constants in each small patch neighborhood (in space and time) centered at $(x_j,t_j), j=1, 2, \ldots, J$. Otherwise, the measurement data is not enough to identify the PDE. Denote $\Omega_j$ to be the local neighborhood centered at $(x_j,t_j)$ and $\hat{\mathbf{c}}_j=(\hat{c}_1^j,\ldots, \hat{c}_K^j))$, we define the local regression error in each patch as 
\begin{align}
\mathcal{E}_{loc}^j(\hat{\mathbf{c}}_j) = \sum_{(x_{j,m},t_{j,m})\in \Omega_j}\left(u_t(x_{j,m},t_{j,m}) -{\sum_{k=1}^K} \hat{c}_k^jf_k(x_{j,m},t_{j,m})\right)^2.
\label{eq_loc}
\end{align}
Define the global regression error as 
\begin{equation}
\label{eq:global}
\mathcal{E}(\widehat{\mathbf{c}}) = \sum_{j=1}^J \mathcal{E}_{loc}^j(\hat{\mathbf{c}}_j),  \quad \widehat{\mathbf{c}}=[\hat{\mathbf{c}}_1, \ldots, \hat{\mathbf{c}}_J]\,.
\end{equation}

For each $l=1,2,\dots, K$, we search for $l$ terms from the dictionary $\mathcal{F}$ whose linear combination using $\hat{\mathbf{c}}_j$  minimizes the global fitting error. In particular, we optimize
\begin{equation}
\begin{aligned}
  &\widehat{\mathbf{c}}^{l} = \arg\min_{\widehat{\mathbf{c}}}\mathcal{E}(\widehat{\mathbf{c}})\\
  &\text{subject to: }\quad \|\widehat{\mathbf{c}}\|_{\text{Group}-\ell_0} = l
\end{aligned}
\end{equation}
using the Group Subspace Pursuit (G-SP) algorithm proposed in~\cite{he2022gppde}. As a generalization of the well-known subspace pursuit~\cite{dai2009subspace}, G-SP sequentially chooses $K$ groups of  variables from the pool of candidates that best fit the residual obtained from the previous iteration. 
Here the group sparsity is defined by
\begin{equation}\nonumber
\|\widehat{\mathbf{c}}\|_{\text{Group}-\ell_0} = \|(\|\tilde{\mathbf{c}}_1\|_1,\dots,\|\tilde{\mathbf{c}}_K\|_1)\|_0,\\
\end{equation}
where $\tilde{\mathbf{c}}_k=(\widehat{c}_k^1,\dots,\widehat{c}_k^J)\in\mathbb{R}^J, k=1,2,\dots,K$. As we increase the sparsity level $l$, $\mathcal{E}(\widehat{\mathbf{c}}^{l})$  decreases. Setting $\widehat{\mathbf{c}}^0 = [0,\dots,0]$, we employ the model score
\begin{align}
S^{l} =  \mathcal{E}(\widehat{\mathbf{c}}^{l})+\rho\frac{l}{K}\label{eq_modelscore}
\end{align}
for $l=1,2,\dots, K-1$, where $\rho>0$ is a penalty parameter for using more terms from the dictionary. We decide that the candidate with $l^*$ features is the optimal if $S^{l^*}=\min_{l=1,\dots,K-1}S^l$. The metric $S^k$ in~\eqref{eq_modelscore} evaluates a candidate model with $l$ features by considering two factors: the fitting error and the model complexity penalty controlled by the parameter $\rho$.  In this work, we fix $\rho$ to be the mean of 
$\{\mathcal{E}(\widehat{\mathbf{c}}^{k})\}_{k=1}^K$ so that the two components of~\eqref{eq_modelscore} are balanced. We find it applicable for both clean and noisy data.  

Once the PDE operator type is identified, one may use various regression techniques to each patch locally to refine the recovery of those coefficients. In case there are sufficiently many sensors densely distributed in the space, one may employ different types of regularizations to further improve the reconstruction, e.g., TV-norm. We leave such explorations to future works and focus on situations where few sensors are used. 

\subsection{Identification guarantee by local regression}
\label{sec:local-regression}
In this section, we show that local patch regression using operators with constant coefficients approximation indeed can identify the underlying PDE with variable coefficients under these conditions: 1) the variable coefficients are bounded away from zero and vary slowly on the patch, 2) the solution data contain diverse information content, on the patch. The first condition basically requires the mathematical problem to be well-posed and the second condition says the local regression system determined by the solution data is well-conditioned.
    
First, we define the admissible set for the coefficients $\cC_{\eps} \subset \bbR$ by 
\begin{equation}
    \cC_{\eps} := (-\infty, \eps)\cup\{0\}\cup  (\eps, \infty),\quad \eps > 0, 
\end{equation}
which means the non-zero coefficients in the unknown PDE should be above a certain threshold $\eps$ on a local patch. The local regression-based identification problem is to find what terms in the PDE operator those nonzero coefficients correspond to. We also introduce the condition constant of the solution data corresponding to the set of candidates in the dictionary $\cF=\{f_k\}_{k=1}^K$ in a neighborhood $B\subset \Omega$,
\begin{equation}\label{eq:constant}
   {\cK}_{B}^{\cF} := \frac{\inf_{\|\hat \bc\|_{\infty}= 1}  \left\|\sum_{k=1}^{K} \hat c_{k}  f_k(y, \tau)\right\|_{L^{\infty}(B)} }{ \sup_{\|\hat \bc\|_{\infty}= 1}  \left\|\sum_{k=1}^{K} \hat c_{k}  f_k(y, \tau)\right\|_{L^{\infty}(B)} },
\end{equation}
where $\hat \bc = (\hat c_{k})\in \bbR^{K}$, $k=1,2,\dots, K$. We say a vector $\hat\bc \in \cC_{\eps}$ if each element of $\hat\bc$ is in $\cC_{\epsilon}$.
\begin{theorem}\label{THM: 4.1}
Suppose the coefficients $c_k(x,t)$ in~\eqref{eq_PDEform} are admissible at some point $(x_0, t_0)$ and take $B$ as a neighborhood of $(x_0, t_0)$. Let $\cA = \{k\mid c_k\neq 0\}$ be the index set for nonzero coefficients at $(x_0, t_0)$.
Let $\{\overline{c}_{k}^{\ast}\}_{k=1}^K$ be an optimal set of constant coefficients such that
\begin{equation}
    \{\overline{c}_{k}^{\ast}\}_{k=1}^K = \argmin_{\overline{c}_{k} \in\,\cC_{\eps}}  \left\| \sum_{k=1}^K \left(\overline{c}_{k} - c_{k}(y, \tau) \right)  f_k(y, \tau)\right\|_{L^{\infty}(B)} ,
\end{equation}
If we denote $\cA^{\ast} = \{k \mid \overline{c}_k^{\ast}\neq 0\}$, then $\cA =\cA^{\ast}$ if the patch size $R = \text{diam}(B)$ satisfies
\begin{equation}\label{EQ: INEQ}
 {\cK}_B^{\cF} > \frac{2  L R}{\eps},
\end{equation}
where $L$ is the Lipschitz constant for all $c_{k}(x,t)$.
\end{theorem}
\begin{proof}
We prove this by contradiction. Let $\tilde{c}_{k} = c_{k}(x_0, t_0)$ for $k\in\cA$ and $\tilde{c}_{k} = 0$ for $k\in\cA^{\complement}$, then 
\begin{equation}
\begin{aligned}
      &\left\| \sum_{k=1}^K \left(\overline{c}_{k}^{\ast} - c_{k}(y, \tau) \right)  f_k(y, \tau)\right\|_{L^{\infty}(B)}= \left\|  \sum_{k=1}^K \left(\overline{c}_{k}^{\ast} - \tilde{c}_{k} + \tilde{c}_{k}- c_{k}(y,\tau) \right)  f_k(y, \tau)\right\|_{L^{\infty}(B)} \\
      &\ge \left\| \sum_{k=1}^K (\overline{c}_{k}^{\ast} -\tilde{c}_{k} )   f_k (y, \tau)\right\|_{L^{\infty}(B)}   -  \left\| \sum_{k\in\cA} \left(\tilde{c}_{k} - c_{k}(y, \tau) \right)  f_k(y, \tau)\right\|_{L^{\infty}(B)} \\
      &\ge  \left\| \sum_{k=1}^K (\overline{c}_{k}^{\ast} -\tilde{c}_{k} )   f_k (y, \tau)\right\|_{L^{\infty}(B)}   -  LR  \sup_{\|\hat \bc\|_{\infty}= 1}  \left\|\sum_{k\in\cA} \hat c_{k}  f_k(y, \tau)\right\|_{L^{\infty}(B)} .
\end{aligned}
\end{equation}
On the other hand, if the conclusion does not hold, then there exists an index $k'$ such that $|\overline{c}_{k'}^{\ast} - \tilde{c}_{k'}|\ge \eps$ and 
\begin{equation}
  \left\| \sum_{k=1}^K (\overline{c}_{k}^{\ast} -\tilde{c}_{k} )   f_k(y, \tau)\right\|_{L^{\infty}(B)}   \ge \inf_{\|\hat \bc\|_{\infty}= \eps}  \left\|\sum_{k=1}^{K} \hat c_{k}  f_k(y, \tau)\right\|_{L^{\infty}(B)}.
\end{equation}
Note the linear scaling,
\begin{equation}
\begin{aligned}
   \inf_{\|\hat \bc\|_{\infty}= \eps}  \left\|\sum_{k=1}^{K} \hat c_{k}  f_k(y, \tau)\right\|_{L^{\infty}(B)} &=     \eps \inf_{\|\hat \bc\|_{\infty}= 1} \left\|\sum_{k=1}^{K} \hat c_{k}  f_k(y, \tau)\right\|_{L^{\infty}(B)}.
\end{aligned}
\end{equation}
Therefore, if
\begin{equation}
    \eps{\cK}_{B}^{\cF} > 2  LR  ,
\end{equation}
we would obtain that 
\begin{equation}
\begin{aligned}
  \left\| \sum_{k=1}^K \left(\overline{c}_{k}^{\ast} - c_{k}(y, \tau) \right)  f_k(y, \tau)\right\|_{L^{\infty}(B)}  > \left\| \sum_{|\alpha|=0}^n \left(\tilde{c}_{k}- c_{k}(y,\tau) \right)  f_k (y, \tau)\right\|_{L^{\infty}(B)}
\end{aligned}
\end{equation}
which is a contradiction with the optimal choice of $\overline{c}^{\ast}_{k}$.
\end{proof}
From the above result, we can immediately see the following.
\begin{corollary}\label{Cor:4.2}
For a differential operator, $\cL$ with constant coefficients, local regression on a neighborhood $B$ can identify the PDE exactly if $\cK_B^{\cF}>0$.
\end{corollary}
\begin{lemma}
For a linear differential operator $\cL$ with constant coefficients, if the solution $u$ contains enough Fourier modes, we have $\cK_{B}^{\cF} > 0$ for $B = \Omega\times (t_0, t_1)$.
\end{lemma}
\begin{proof}
Suppose  $\cK_B^{\cF} = 0$ on $B= \Omega\times (t_0, t_1)$. It implies that $\sum_{k=1}^{K} \hat{c}_{k}  f_k(y, \tau) = 0$ for any  $(y,\tau)\in B$ and some $\hat{\bc}=(\hat{c}_k)\in \mathbb{R}^K, \|\hat{\bc}\|_{\infty}=1$. Therefore, we must have 
\begin{equation}
    \sum_{k=1}^{K} \hat{c}_{k}  f_k(y, \tau) = 0
\end{equation}
for all $y\in\Omega$ and $\tau\in(t_0, t_1)$. Expanding the solution in terms of the Fourier series
\begin{equation}
    u(y, \tau) = \sum_{v\in\bbZ^d} d_{v}(\tau ) e^{iv\cdot y}, 
\end{equation}
Since each $f_k$ is a partial derivative, $f_k(e^{iv\cdot y}) = P_k(iv) e^{iv\cdot y}$ for certain polynomial $P_k$, we have $ d_{v}(\tau) \sum_{k=1}^{K} \hat{c}_{k}  P_k(iv) = 0$. Each $d_v(\tau)$ satisfies an autonomous ODE $d_v^{'}= (\sum_{k=1}^{K} c_{k}  P_k(iv))d_v$ from~\eqref{eq_PDEform},  which means either $d_{k}(\tau) \equiv 0$ on $(t_0, t_1)$ or 
\begin{equation}
    \sum_{k=1}^{K} \hat{c}_{k}  P_k(iv)  = 0.
\end{equation}
Since the above algebraic equation only permits finitely many solutions of $v$, we prove the result.
\end{proof}
The above lemma indicates that if $\cL$ is a linear differential operator with constant coefficients and the solution has sufficiently many Fourier modes, then we can always find a local patch $B\subset \Omega\times (t_0, t_1)$ such that $\cK_{B}^{\cF} > 0$. Combined with Corollary~\ref{Cor:4.2}, we have the following Corollary.

\begin{theorem}\label{THM:4.3}
A linear differential operator $\cL$ with constant coefficients can be identified exactly by local regression if the solution $u$ contains sufficiently many Fourier modes.
\end{theorem}

\begin{remark}
The numerator in the definition of $\cK_B^{\cF}$ measures the minimal support of the set  $\cS(B):=\{(f_1(y,\tau),\ldots, f_K(y,\tau))\in\mathbb{R}^K, (y,\tau)\in B\}$ on the unit $\ell^{\infty}$ sphere. The denominator can be relaxed to $\max_{(y,\tau)\in B}\sum_{k\in\cA} | f_k(y, \tau)| $.

If the solution varies little on $B$, e.g., $R$ is small, the set $\cS(B)$ is close to a constant vector in $\mathbb{R}^K$, hence $\cK_B^{\cF}$ is small. On the other hand, the coefficients need to vary slowly on $B$ so that $LR$ is small. So some scale separation between the coefficients and the solution is needed on $B$.

In general, the larger $K$ is, the smaller $\cK_B^{\cF}$ will be.
In other words, the larger the dictionary, the harder the identification of the underlying PDE. 
\end{remark}

\subsection{Data driven and data adaptive measurement selection}\label{SEC:patch}
In practice, using the solution data indiscriminately for PDE identification may not be a good strategy. 

On each patch, the local regression is approximated by a PDE with constant coefficients as studied. For a stable local identification, the solution should contain some variation as shown in the previous section.  On the other hand, the data on those patches whose measurement resolution can not resolve the rapid variation of the solution leads to significant error, e.g., truncation error in numerical differentiation, and hence should not be used either. 
To improve CaSLR's robustness and accuracy, we propose the following process to  select patches containing reliable and accurate information.

First, we propose to use the following numerical estimate of the local Sobolev semi-norm to filter out those patches in which the solution may be singular or oscillate rapidly, 
\begin{align}
\label{eq:sobolev}
    \beta(x_j,t_j) = \sqrt{\frac{1}{m_j}\sum_{m=1}^{m_j}\sum_{p=1}^{P_{\max}}(\partial_x^{p}u(x_{j,m},t_{j,m}))^2}
\end{align}
where $P_{\max}$ is the maximal order of partial derivatives in the dictionary. We remove those local regressions at $(x_j,t_j)$ in {~\eqref{eq_loc}} if $\beta(x_j,t_j) < \underline{\beta}$ or 
$\beta(x_j,t_j)>\bar{\beta}$ for some thresholds $\underline{\beta},\bar{\beta}>0$. In this work, we fix $\underline{\beta}$ to be the 1st percentile and $\bar{\beta}$ the 99th percentile of all the collected local Sobolev semi-norms.

However, when the measurement data contain noise, the constant solution may not be detected using numerical approximation of~\eqref{eq:sobolev}.  Here we design the following criterion to detect whether the solution is almost constant in a neighborhood. 
For an arbitrary $(\bx,t)\in\Omega\times[0,T_{\max}]$, consider a neighborhood centered at $(\bx,t)$ denoted by  $B_{(\bx,t)}=\prod_{i=1}^{d}[x_i-r_i,x_i+r_i]\times[t-r,t+r]$ with radius in each dimension $r_i>0$, $i=1,2,\dots,d$, $r>0$. Suppose we observe a discrete set of data (including $(\bx,t)$) with noise in this neighborhood: $\widehat{B}_{(\bx,t)}=\{\widehat{u}(\by,s):=u(\by,s)+\varepsilon(\by,s):~\varepsilon(\by,s)\overset{\text{i.i.d.}}{\sim}\mathcal{N}(0,\sigma^2),~(\by,s)\in B_{(\bx,t)}\}$. Furthermore, we assume that $u$ is locally Lipschitz continuous in $B((\bx,t))$ with constant $L>0$. Let $|\widehat{B}_{(\bx,t)}|$ be the cardinality of $\widehat{B}_{(\bx,t)}$ and $R = \max\{r_1,\dots,r_d,r\}$. We note that
\begin{align*}
\widehat{u}(\bx,t) -\overline{\widehat{u}}_{\widehat{B}}
\sim\mathcal{N}(\mu_{(\bx,t)},(1-|\widehat{B}_{(\bx,t)}|^{-1})\sigma^2)
\end{align*}
where $\mu_{(\bx,t)}=u(\bx,t)-\overline{u}_{\widehat{B}}$, 
$\overline{u}_{\widehat{B}}$ and $\overline{\widehat{u}}_{\widehat{B}}$ is the average of $u$ and $\widehat{u}$ over $\widehat{B}$ respectively.
By the assumption of Lipschitz continuity of $u$, $|\mu_{(\bx,t)}|\leq \sqrt{D}LR$, where $D:=d+1$. 

Hence, to estimate $\sigma^2$, we  consider a collection of points $\{(\bx_1,t_1),\dots, (\bx_N,t_N)\}\subset\Omega\times[0,T_{\max}]\}$ with non-intersecting boxes $B_{(\bx_n,t_n)}$, $n=1,2,\dots, N$ centered at each of them, respectively. Within each $B_{(\bx_n,t_n)}$, we observe noisy data over a finite discrete subset $\widehat{B}_{(\bx_n,t_n)}\subset B_{(\bx_n,t_n)}$ having cardinality $B>0$. Denote $\zeta_n=\widehat{u}(\bx_n,t_n) -\overline{\widehat{u}}_{\widehat{B}_{(\bx_n,t_n)}}$, then by the computations detailed in Appendix~\ref{append_detail},
\begin{align}
\widehat{\sigma}^2=\frac{B\sum_{m=1}^N(\zeta_m-\frac{1}{N}\sum_{n}\zeta_n)^2}{(N-1)(B-1)}\label{eq_est_sigma}
\end{align}
is a biased estimator for $\sigma^2$ with
\begin{align}
|\mathbb{E}[\widehat{\sigma}^2-\sigma^2]|\leq \frac{DNBL^2R^2}{(N-1)(B-1)}\label{eq_est_mean}
\end{align}
and
\begin{align}
\text{Var}\,\widehat{\sigma}^2\leq \frac{2\sigma^4}{N-1}+\frac{NB\sigma^2\gamma}{(N-1)^2(B-1)},\label{eq_est_var}
\end{align}
where $\gamma : = 4DL^2R^2$. From the results above, we observe that  to reduce the estimator's bias, we should use small patches, i.e., small $R$, and to control the estimator's variance, we need to use more patches, i.e., larger $N$. Notice that the upper bound of the bias in~\eqref{eq_est_sigma} is closely related to the magnitude of $u$'s local variation in $B(\bx,t)$ expressed as $\gamma$. This implies that when $u$ has mild variation in $B(\bx,t)$, the estimator~\eqref{eq_est_sigma} becomes approximately unbiased. This was expected since the statistical characteristics of the additive noise become more identifiable if the underlying function does not introduce extra variations. We note that the upper bound of the variance~\eqref{eq_est_var} consists of two terms. The first term reflects the fact that when sufficiently many samples are provided, the estimator becomes more stable, while the second term also shows the influence of the magnitude of $u$'s variation.  

Now for any two distinct points $(\bx,t),(\by,s)\in\Omega\times[0,T_{\max}]$, we know that  $\widehat{u}(\bx,t)-\widehat{u}(\by,s)\sim\mathcal{N}(u(\bx,t)-u(\by,s),2\sigma^2)$, hence under the hypothesis $\mathcal{H}_0:$ $u(\bx,t)=u(\by,s)$, we would have
\begin{align}
    \mathbb{P}\left(\frac{|\widehat{u}(\bx,t)-\widehat{u}(\by,s)|}{\sqrt{2}\sigma}>\alpha_{0.90}~|~\mathcal{H}_0\right)<0.1\,,
\end{align}
where $\alpha_{0.90}=1.644853$. By the estimation~\eqref{eq_est_sigma}, if
\begin{align}
|\widehat{u}(\bx,t)-\widehat{u}(\by,s )|>\sqrt{2}\alpha_{0.90}\widehat{\sigma}\,,\label{eq_compare}
\end{align}
we reject $\mathcal{H}_0$; otherwise, we accept $\mathcal{H}_0$. In case of PDE systems of multiple unknown functions $\{u_j\}_{j=1}^J$, we consider the hypothesis $\mathcal{H}_0: u_j(\bx,t)=u_j(\by,s)$ at least for one $j=1,\dots,J$, and reject $\mathcal{H}_0$ if  $\min_{j=1,\dots,J}\{|\widehat{u}_j(\bx,t)-\widehat{u}_j(\by,t)|\}>\sqrt{2}\alpha_{0.90}\widehat{\sigma}$; otherwise, we accept $\mathcal{H}_0$. 

The statistical test described above helps to determine if the underlying function restricted to a patch shows variation when the collected data is contaminated by Gaussian noise. Specifically, for each pair of sampled points in a patch, we conduct the comparison~\eqref{eq_compare}: if any of them yields a rejection of $\mathcal{H}_0$, we keep the patch; otherwise, we discard the data collected in the patch.

\begin{remark}
For robust identification of PDE with variable coefficients using local patch data, it is desirable to have local sensors that support local patches whose patch size $R$ is small enough to resolve the variation of the coefficients while the measurement resolution $h$ is fine enough so that $\frac{R}{h}$ is large enough and hence can detect a good bandwidth of modes in the solution.
At the same time, one needs to select patch data that contain as many modes as possible that can be resolved by patch resolution.
\end{remark}

\subsection{Numerical Experiments}\label{SEC: NUM}
In this section, we present numerical examples to demonstrate our local regression and global consistency-enforced PDE identification method. In some of the numerical examples, we use the following transition function in time with slope $s\in\mathbb{R}$ and critical point $t_c\in\mathbb{R}$ as
\begin{align}
    \tau(t;s,t_c) = 0.5 + 0.5\tanh( s(t-t_c))),\quad t\in\mathbb{R},
\end{align}
which models a smooth emergence (when $s>0$) or a smooth decaying (when $s<0$) behavior; when $|s|$ increases, the transition becomes sharper. The critical point $t_c$ signifies the boundary point between two phases. Moreover, we use the Jaccard score~\cite{jaccard1912distribution} 
\begin{equation}
J(S_0,S_1) = \frac{|S_0\cap S_1|}{|S_0\cup S_1|}\label{eq_jaccard}
\end{equation}
as a metric for identification accuracy. Here $S_0$ and $S_1$ denote two sets and the Jaccard score measures the similarity between them. In our case, $S_0$ represents the set of indices of the true features in a given dictionary and $S_1$ denotes the set of indices of features in an identified PDE model. We note that $J(S_0,S_1) \in [0,1]$, and $J(S_0,S_1) = 1$ if and only if $S_0=S_1$. A higher Jaccard score indicates that the sets to be compared are more similar.

In the following, we consider a general setup for PDE identifications using a collection of local sensors. For each experiment, we put sensors randomly in space, and each of them collects data at some fixed intervals in time.  Each sensor collects data in a cubic neighborhood in space and time whose side length  is $ 2r+1$ in each spatial dimension and $2r_t+1$ in time. Here $r>0$ and $r_t>0$ are referred to as the sensing radius and time duration respectively. 
\\ \\
{\bf Example 1:} Transport equation. \\
First, we examine a transport equation with a periodic boundary condition whose speed varies both in space and time:
\begin{equation}\label{eq: trans}
    \begin{aligned}
u_t(x,t) &= (1+0.5\sin(\pi x)\tau(t;-10,0.5))u_x(x,t)\;,\quad(x,t)\in[-1,1)\times (0,1],\\
u(x,0)& = \sin(4\pi(x+0.1))+ \sin(6\pi x) +\cos(2\pi (x-0.5)) + \sin(2\pi (x+0.1)).
\end{aligned}
\end{equation}
We solved it numerically over a grid with 100 points in space and 5000 points in time, and Figure~\ref{fig_exp1} (a) shows the solution. 
\begin{figure}[!htb]
\begin{tabular}{ccc}
    \begin{minipage}[p]{0.3\textwidth}
    \includegraphics[width = \textwidth]{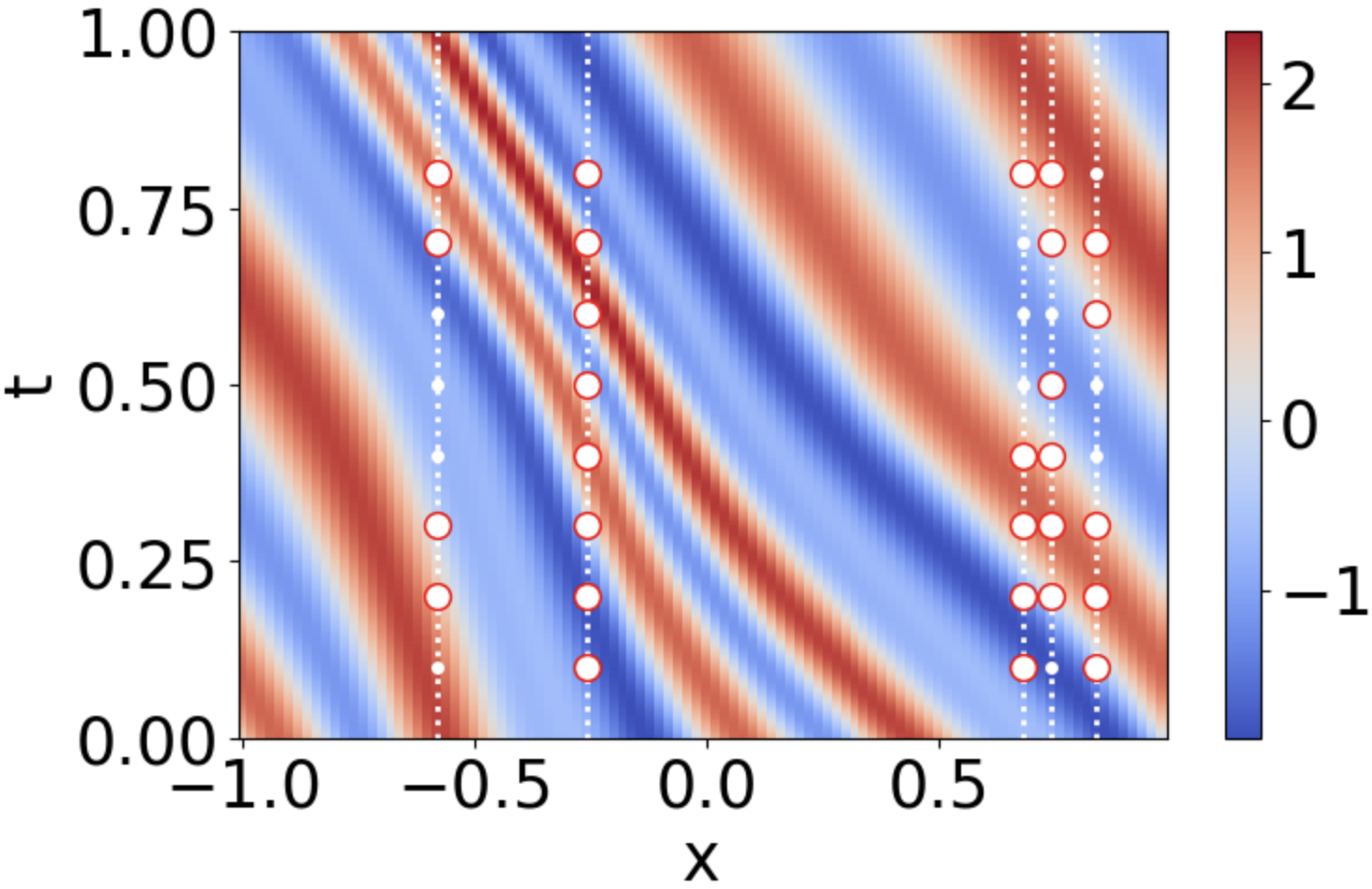}
    \end{minipage}&
    \begin{minipage}[p]{0.3\textwidth}
    \includegraphics[width = \textwidth]{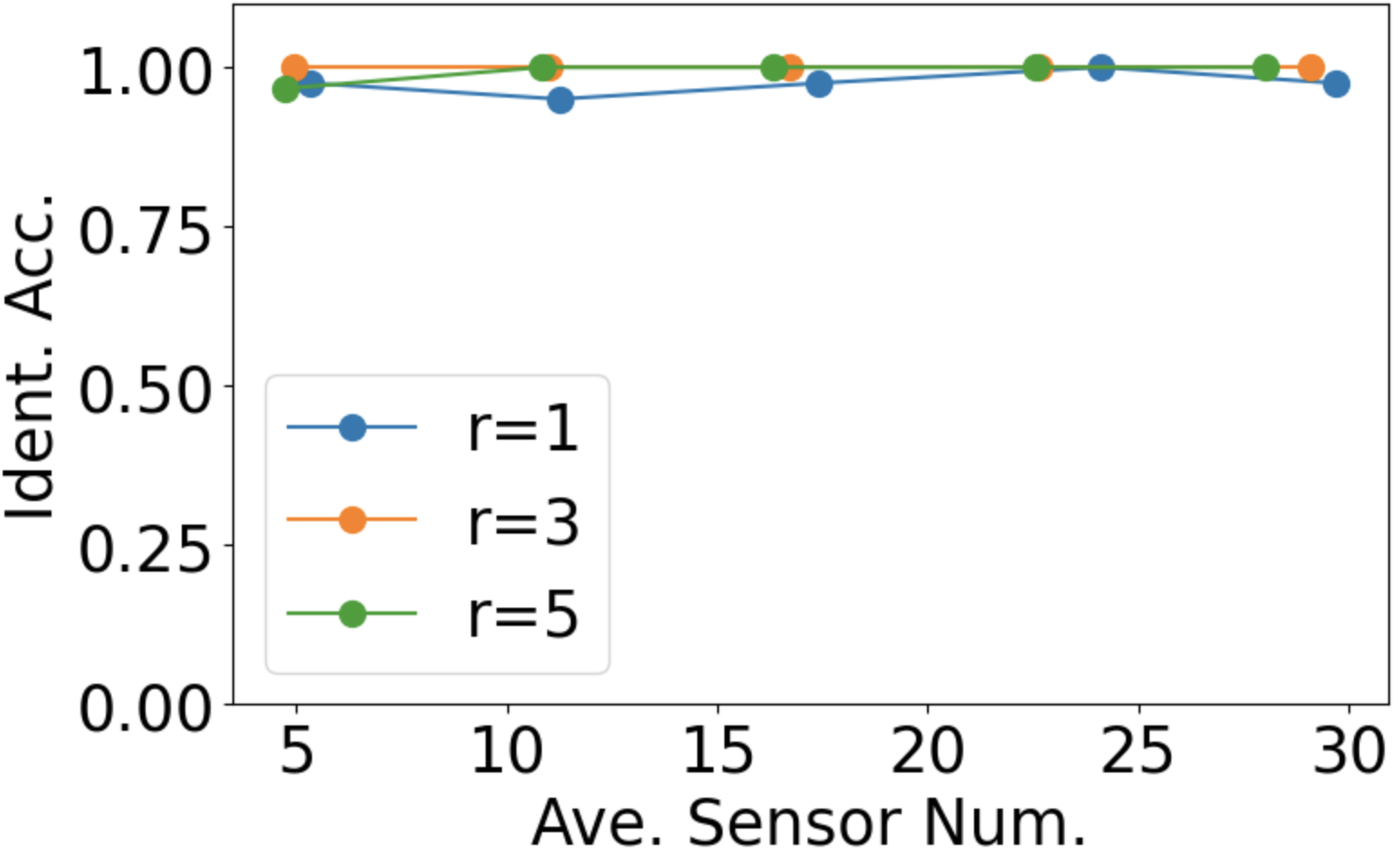}
    \end{minipage}&
    \begin{minipage}[p]{0.3\textwidth}
    \includegraphics[width = \textwidth]{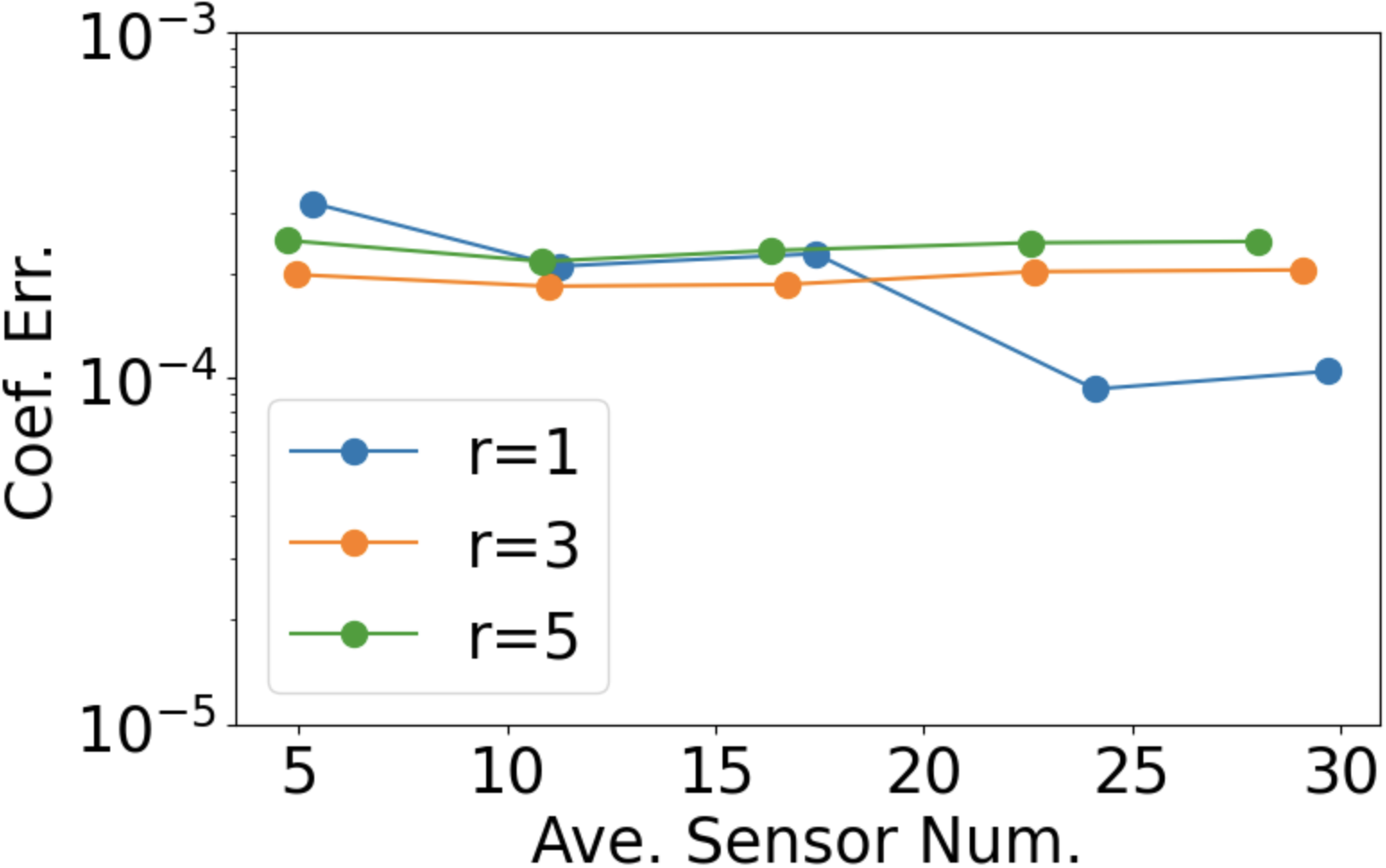}
    \end{minipage}\\
    (a)&(b)&(c)
\end{tabular}
\caption{Identification of transport equation~\eqref{eq: trans} with space-time varying speed.}\label{fig_exp1}
\end{figure}
 Based on local patch data and using a dictionary of size $59$ including up to $4$-th order partial derivatives of the solution and products of them up to $3$ terms, plus $\sin(u)$, $\cos(u)$, $\sin(u_x)$, and $\cos(u_x)$,  we employed the proposed patch selection scheme described in Section~\ref{SEC:patch}  and applied CaSLR to identify the PDE type as well as to reconstruct the associated local coefficients. For a series of combinations of different numbers of sensors and sensing radii, we conducted 20 independent experiments and recorded the mean values of identification accuracy, shown in (b), and reconstruction error, shown in (c). A set of sensor distributions is shown in (a) as red markers. We observe that the identification of a transport equation is highly reliable even with a single sensor whose patch size is small. %fixed at a single location whose sensing ranges are narrow.
This considerable robustness is closely related to the fact that the solution data contain diverse Fourier modes over the time-space domain. As we will see in Section~\ref{exp_patchfilt}, our data-driven patch selection helps to preserve such performances even when the solution has compact support at all snapshots.
\\ \\
{\bf Example 2:} KdV type equation. \\
Next, we test on the KdV type equation with space-time varying coefficients and a periodic boundary condition in the following form
\begin{equation}\label{eq: kdv}
    \begin{aligned}
u_t(x,t) &= (3+200t\sin(\pi x))u(x,t)u_x(x,t)\\&\quad\quad +\frac{5+\sin(\frac{400\pi t}{3})}{100}u_{xxx}(x,t),\quad (x,t)\in[-1,1)\times (0,1.5\times10^{-2}],\\
u(x,0)& = \sin(4\pi(x+0.1))+ 2\sin(5\pi x) +\cos(2\pi (x-0.5)) + \sin(3\pi x) +\cos(6\pi x).
\end{aligned}
\end{equation}
\begin{figure}[!htb]
\centering
\begin{tabular}{ccc}
    \begin{minipage}[p]{0.3\textwidth}
    \includegraphics[width = \textwidth]{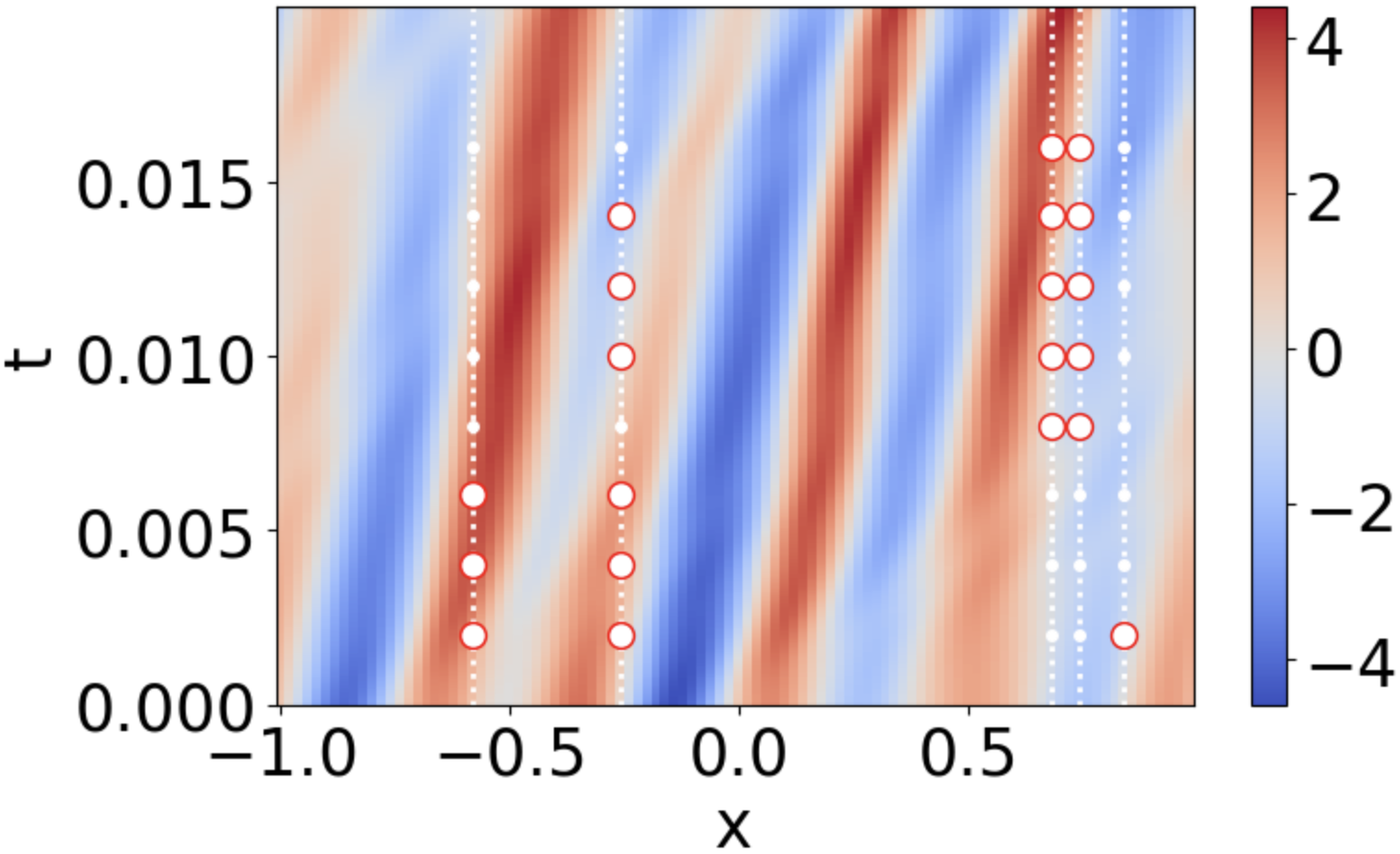}
    \end{minipage}&
    \begin{minipage}[p]{0.3\textwidth}
    \includegraphics[width = \textwidth]{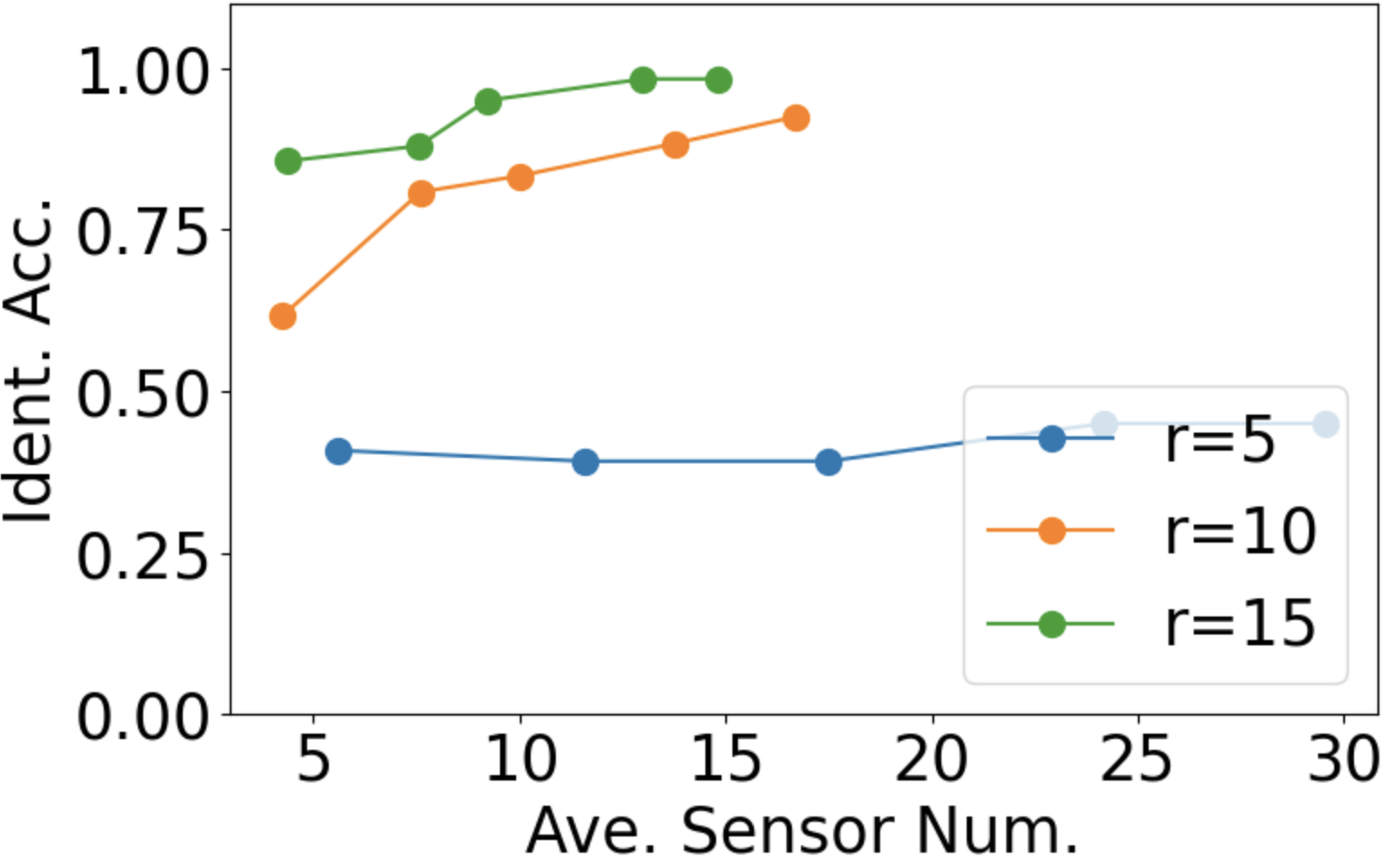}
    \end{minipage}&
    \begin{minipage}[p]{0.3\textwidth}
    \includegraphics[width = \textwidth]{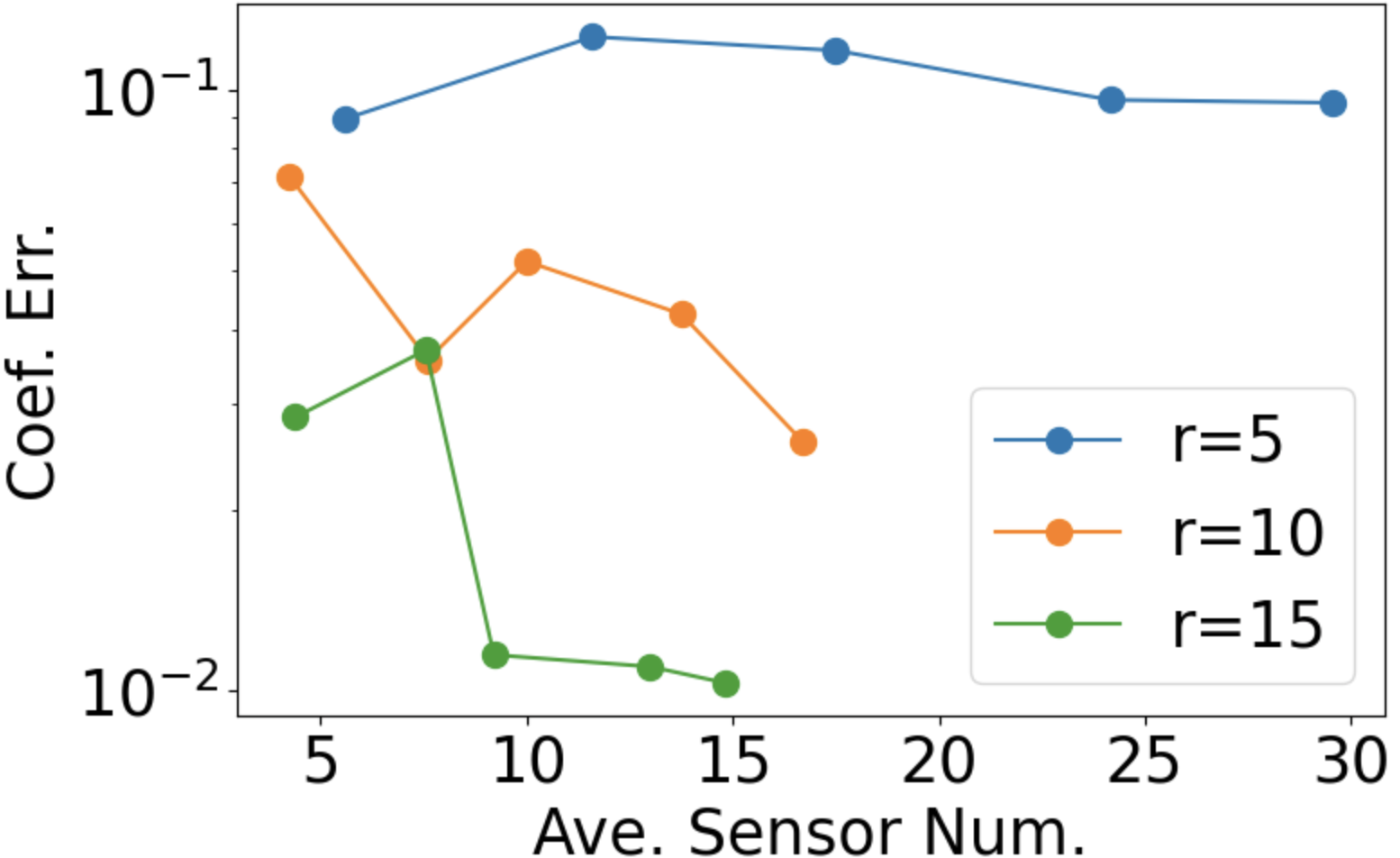}
    \end{minipage}\\
    (a)&(b)&(c)
\end{tabular}
\caption{Identification of KdV type equation~\eqref{eq: kdv} with space-time varying coefficients. }\label{fig_exp2}
\end{figure}
Figure~\ref{fig_exp2} (a) shows its numerical solution with a set of exemplary sensors indicated by red markers. Using a dictionary of size $59$ as specified in the previous example, the identification and reconstruction results of CaSLR are reported in (b) and (c), respectively. As there are more terms and of different orders in the PDE, more Fourier modes, in particular, relatively low-frequency modes for lower order terms in the PDE, need to be included in each local patch measurement for accurate and robust identification. That is why 
patches of size larger than that in the transport equation example are needed.
\\\\
{\bf Example 3:} Schr\"{o}dinger equation.\\
We perform a test on the Schr\"{o}dinger equation $$i\psi_t = \frac{1}{2}\psi_{xx} - V \psi$$ with a space-time dependent potential function $V = -10- 2\sin(40\pi t)\cos(\pi x)$ and periodic boundary condition. Let $\psi = u + iv$, then we obtain the following system.
\begin{align}\label{EQ: SCH}
\begin{cases}
u_t(x,t) &= \frac{1}{2}v_{xx}(x,t)-V(x,t)v(x,t)\;,\\
v_t(x,t) &= -\frac{1}{2}u_{xx}(x,t)+V(x,t)u(x,t)\;
\end{cases}
\end{align}
for $(x,t)\in[-1,1)\times (0,0.05]$ and we take $u(x,0) = 5+f_1(x)$ and $v(x,0) = 3+f_2(x)$, $x\in[-1,1)$, for the initial conditions, where $f_1$ and $f_2$ are random functions~\eqref{EQ: RANDOMFUNC} with $4$ modes.  We solved the system numerically over a grid with $800$ space points and $10000$ time points, then down-sampled the data to a grid of size $200\times 5000$.
\begin{figure}[!htb]
\centering
\begin{tabular}{cc}
    \begin{minipage}[p]{0.3\textwidth}
    \includegraphics[width = \textwidth]{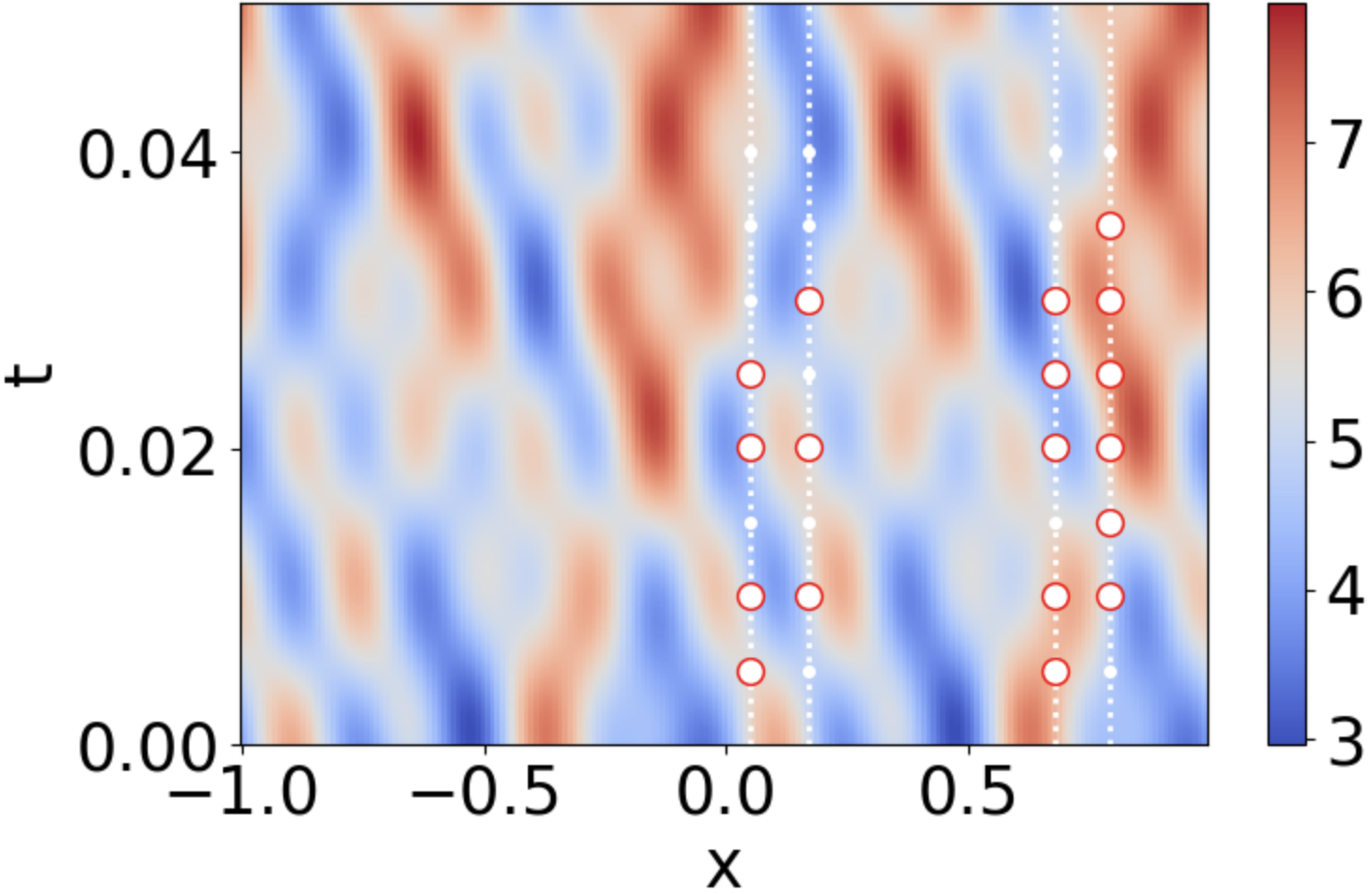}
    \end{minipage}&
    \begin{minipage}[p]{0.3\textwidth}
    \includegraphics[width = \textwidth]{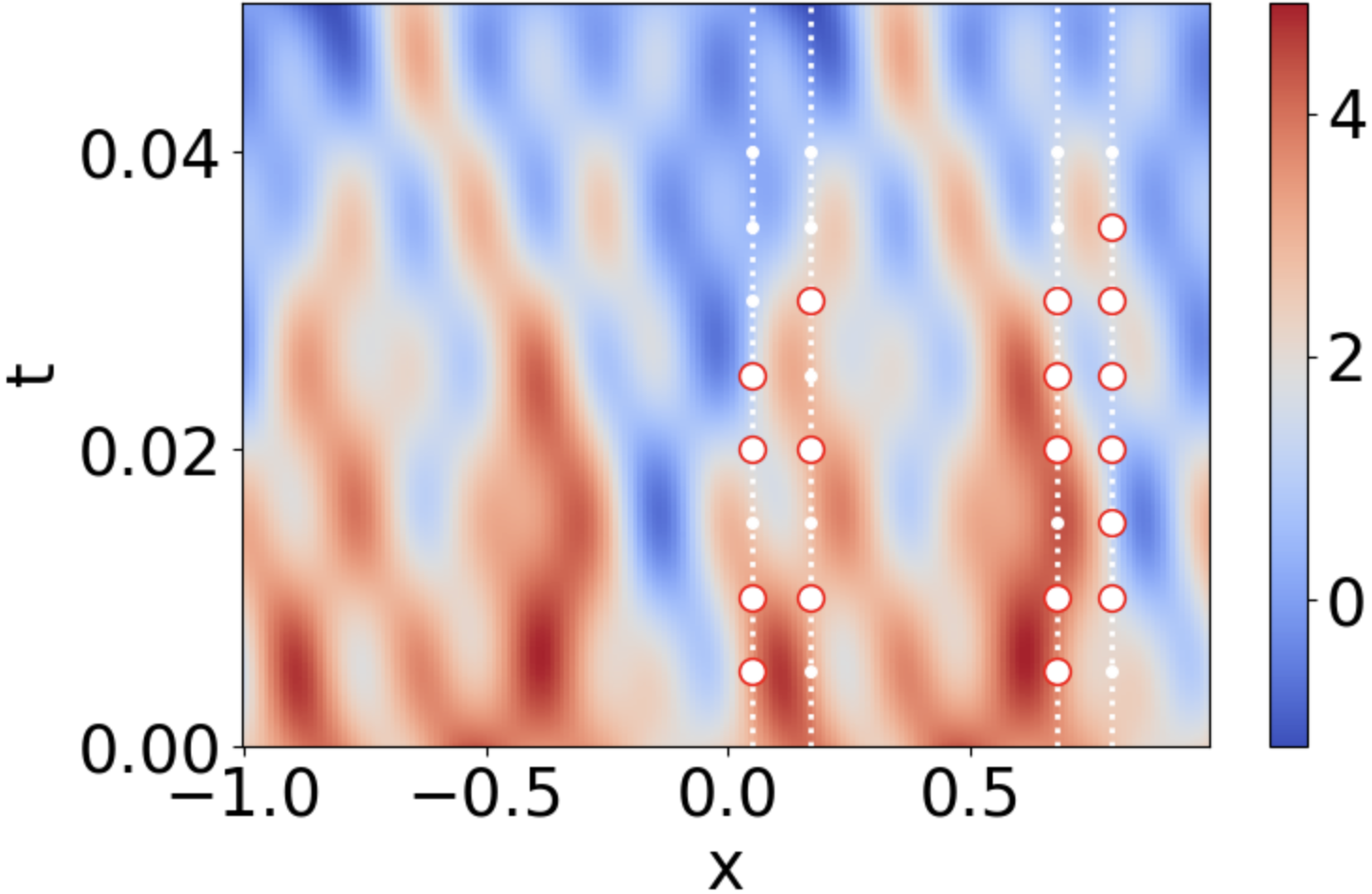}
    \end{minipage}\\
    \quad (a)&\quad (b)\\
    \begin{minipage}[p]{0.3\textwidth}
    \includegraphics[width = \textwidth]{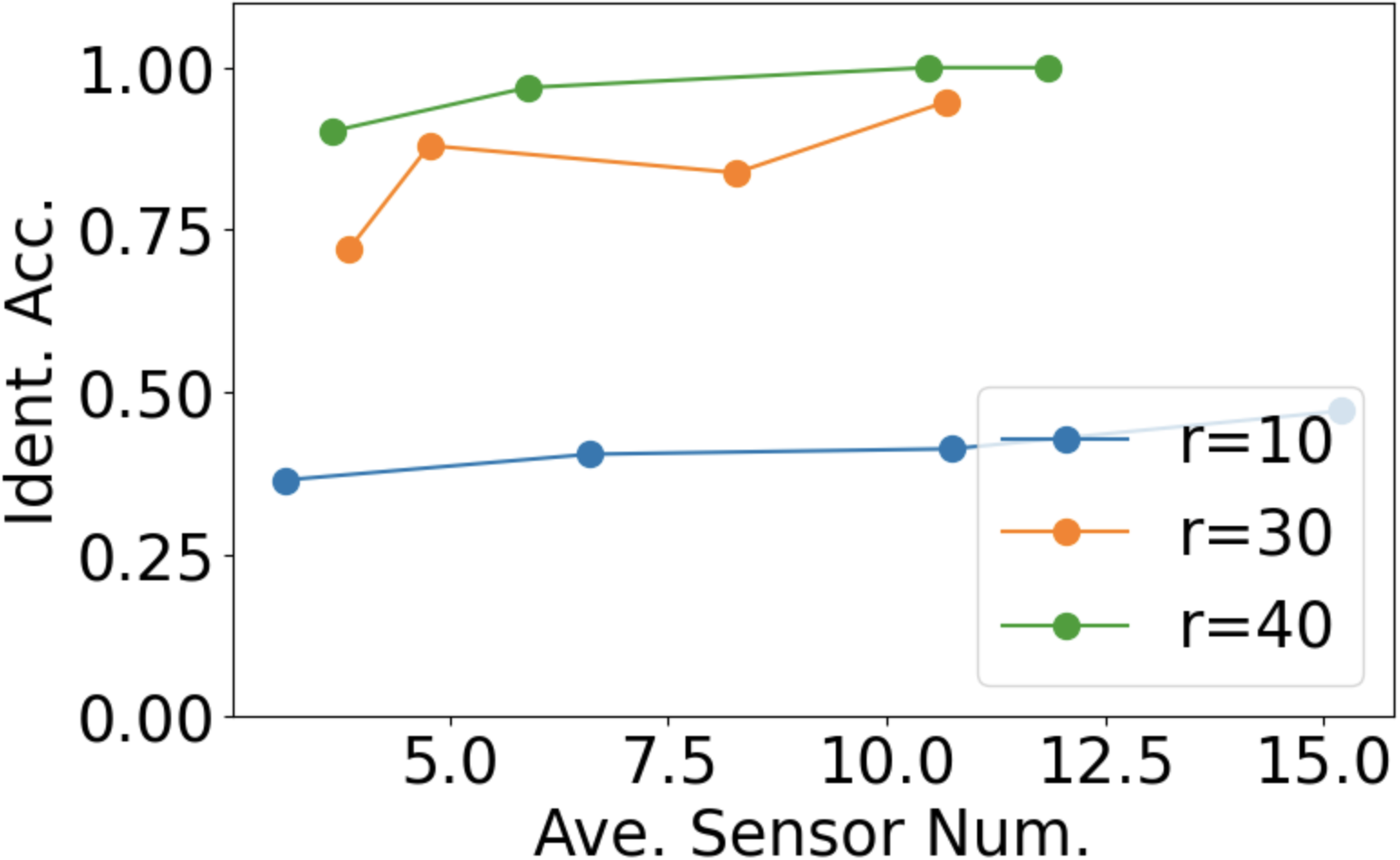}
    \end{minipage}&
    \begin{minipage}[p]{0.3\textwidth}
    \includegraphics[width = \textwidth]{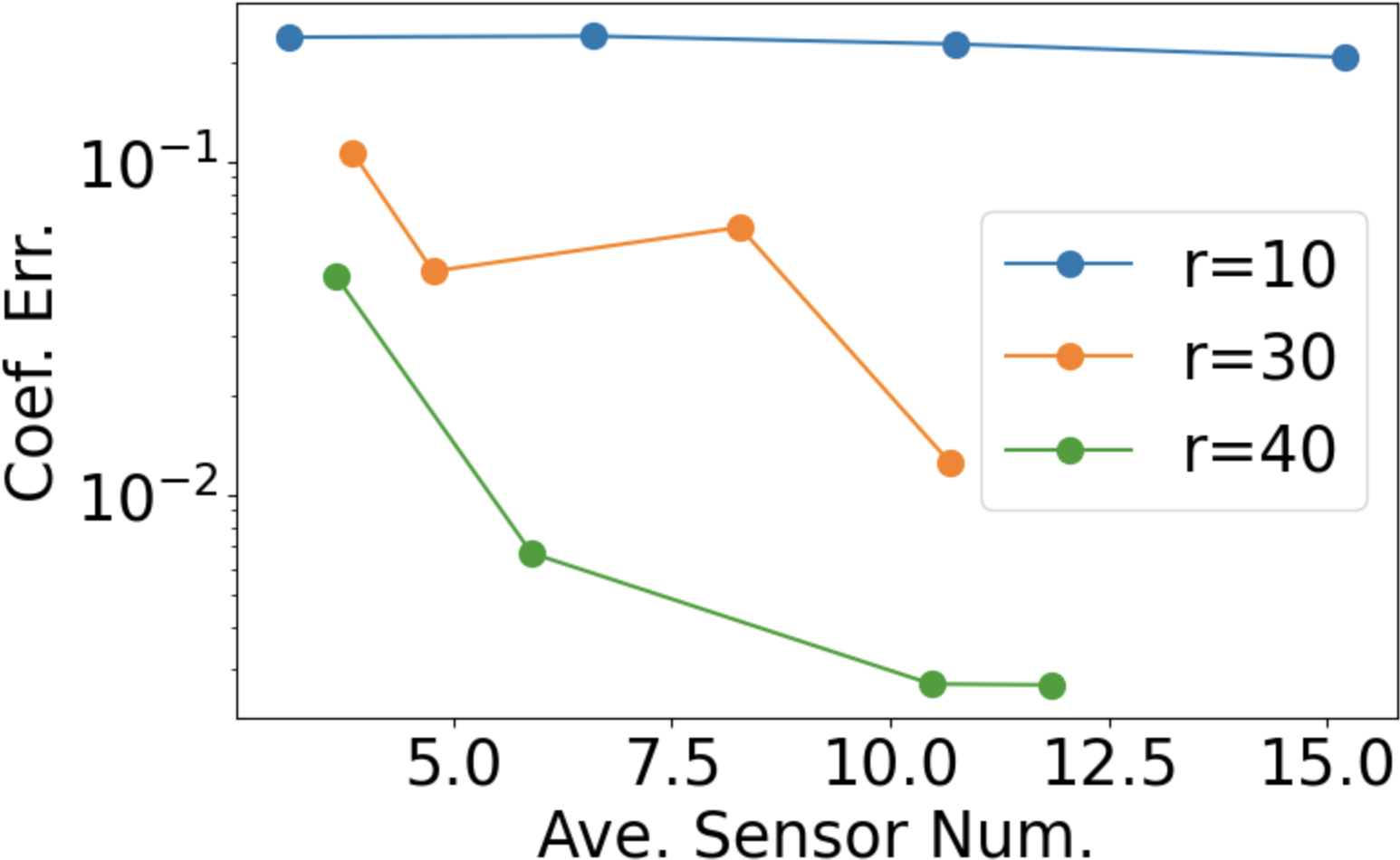}
    \end{minipage}\\
    \quad (c)&\quad (d)
\end{tabular}
\caption{Identification of the Schr\"{o}dinger equation~\eqref{EQ: SCH} with a time-space dependent potential. The solution data for (a) $u$ and (b) $v$ with a set of sensors selected by the proposed scheme. (c) The identification accuracy. (d) The reconstruction error. For each radius, 20 experiments were conducted and the average values were reported.} \label{fig_exp3}
\end{figure}
Figure~\ref{fig_exp3} (a) shows the solutions with a set of sensors selected by our filtering scheme. Notice that these sensors concentrate near the regions where both $u$ and $v$ have significant variations. For this example, considering the computational cost, we consider a dictionary of size 44, where the partial derivatives of orders up to $3$ and products of up to $2$ terms are included. For the same reason, as discussed in Example 2 above, when a relatively large patch size is used, our method is successful in identifying the correct system with an accurate reconstruction of the local coefficients.
\\ \\
%\subsubsection{2D circular flow}
{\bf Example 4:} 2D circular flow.\\
Consider a 2D circular flow characterized by the PDE with spatially-dependent coefficients:
\begin{equation}\label{EQ:2DFLOW}
\begin{aligned}
u_t(x,y,t) &= -yu_x(x,y,t) +xu_y(x,y,t),\quad (x,y,t)\in \mathbb{R}^2\times(0,2\pi]\\
u(x,y,0)&=f(x,y), \quad (x,y)\in\mathbb{R}^2\,
\end{aligned}
\end{equation}
for some initial condition $f$. We note that~\eqref{EQ:2DFLOW} can be transformed into a transport equation with a constant angular speed in the polar coordinate, and it admits an exact solution $u(x,y,t) = f(\sqrt{x^2+y^2}\cos(\arctan(y/x)-t), \sqrt{x^2+y^2}\sin(\arctan(y/x)-t))$. In this example, we take $f(x,y) = \cos(4\sqrt{x^2+y^2})\cos(2\arctan(y/x))$, and the initial data over a uniform grid of size $64\times64$ is shown in Figure~\ref{fig_exp_circflow} (a). The dictionary consists of 27 features, where partial derivatives up to order 2 and products up to 2 terms are included. For a range of values of $R>0$, we randomly locate $5$ sensors with a sensing radius of $r$ ($r=3,5,7$) on a circle of radius $R$ centered at the origin and repeat the experiments 20 times. Figure~\ref{fig_exp_circflow} (b) and (c) show the average identification accuracy and the reconstruction errors as the radius of the circles $R$ increases from $3$ to $30$ grid points. Observe that  when the sensors are located close to the origin, the PDE type is hard to identify since the coefficients are all close to zero and the solution is almost constant 1. The misleading low reconstruction error means that all recovered coefficients including those for wrong features are close to zeros. When the sensors are located far away from the origin because the solution changes very slowly and those low-frequency Fourier modes can not be detected when the local patch is too small. We see that both PDE identification and coefficient approximations are improved when the patch size becomes larger. When the sensors are located away from the origin at a moderate distance, the PDE identification is most successful for each patch size. 
\begin{figure}[!htb]
\centering
\begin{tabular}{ccc}
    \begin{minipage}[p]{0.3\textwidth}
    \includegraphics[width = \textwidth,height=0.70\textwidth]{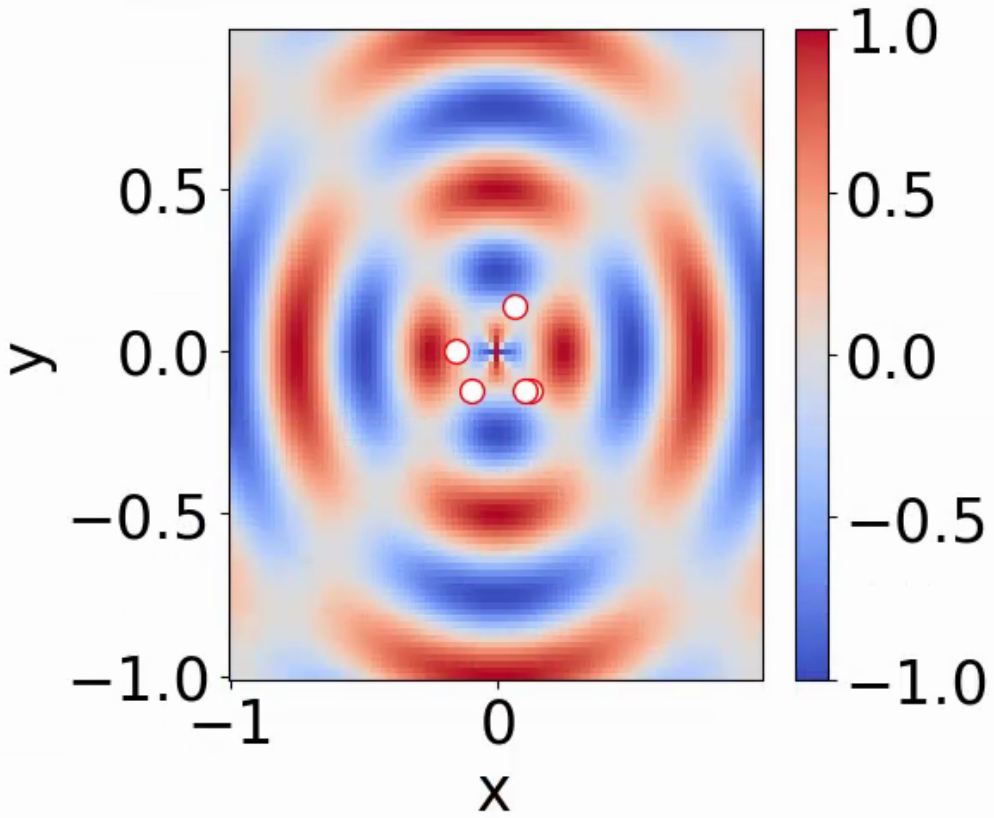}
    \end{minipage}&
    \begin{minipage}[p]{0.3\textwidth}
    \includegraphics[width = \textwidth]{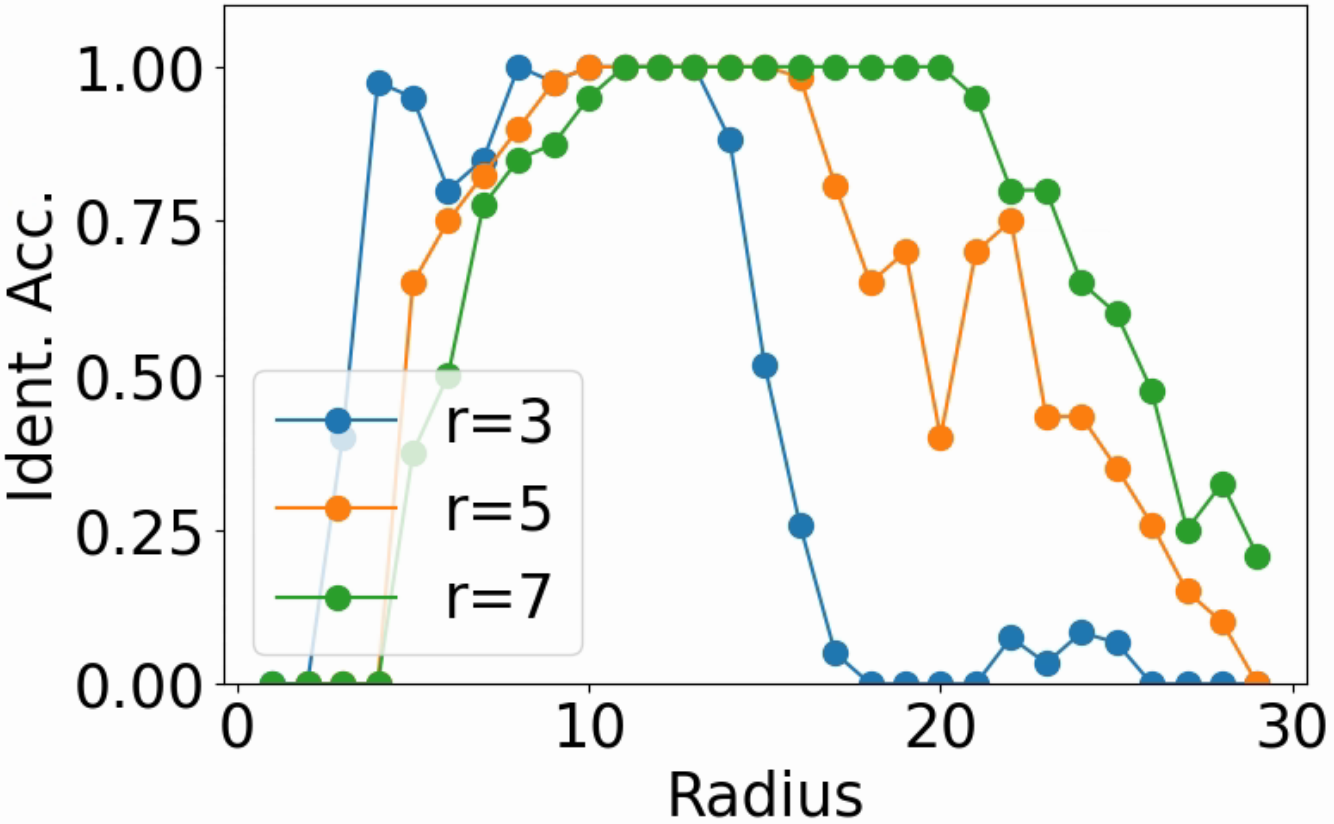}
    \end{minipage}&
    \begin{minipage}[p]{0.3\textwidth}
    \includegraphics[width = \textwidth]{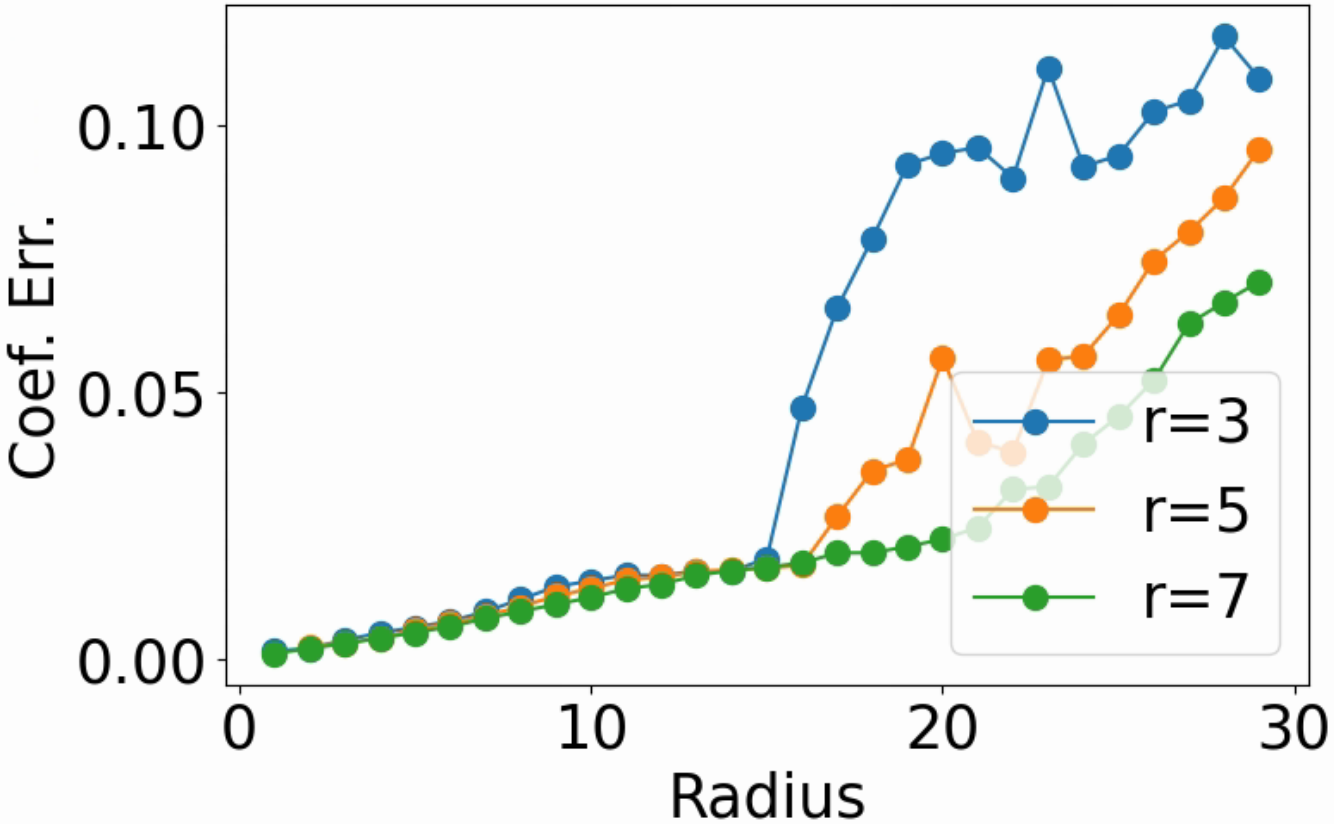}
    \end{minipage}\\
    (a)&(b)&(c)
\end{tabular}
\caption{Identification of a 2D circular flow. (a) The solution data at $t=0$ and 5 randomly distributed sensors. (b) Identification accuracy of 5 sensors with different sensing radii randomly located on a concentric circle whose radius varies from $1$ to $30$ grid points. (c) Reconstruction errors. For each radius, 20 experiments were conducted and the average values were reported. }\label{fig_exp_circflow}
\end{figure}
\subsubsection{Identification with random initial conditions}
In this set of experiments, we demonstrate identification accuracy with respect to different random initial conditions. We consider four different PDEs: 
\begin{enumerate}
    \item [I.] transport equation with constant speed: $u_t(x,t)=2u_x(x,t)$;
    \item [II.] transport equation with variable speed: $u_t =(2+\sin(2\pi t))u_x(x,t)$;
    \item [III.] heat equation with constant coefficient: $u_t(x,t)=0.5u_{xx}(x,t)$;
    \item [IV.]  heat equation with variable coefficient:  $u_t(x,t)=(0.5+0.25\sin(2\pi t))u_{xx}(x,t)$.
\end{enumerate}
All cases are over the domain $(x,t)\in[-1,1)\times (0,0.5]$. Here we take a grid of $200\times 5000$ points. For each equation in the above, we set the random function~\eqref{EQ: RANDOMFUNC} as the initial conditions for a range of mode numbers. For I., we also consider exact data and exact features, i.e., exact partial derivatives of $u$. We randomly select 5 spatially located sensors and each of them can collect data at 10 uniformly distributed time points; the sensing range has a size of $7$ points in space and $31$ points in time. Then we conduct $20$ experiments using the corresponding random initial conditions.
\begin{figure}[!htb]
\centering
\begin{tabular}{cc}
    \begin{minipage}[p]{0.35\textwidth}
    \includegraphics[width = \textwidth]{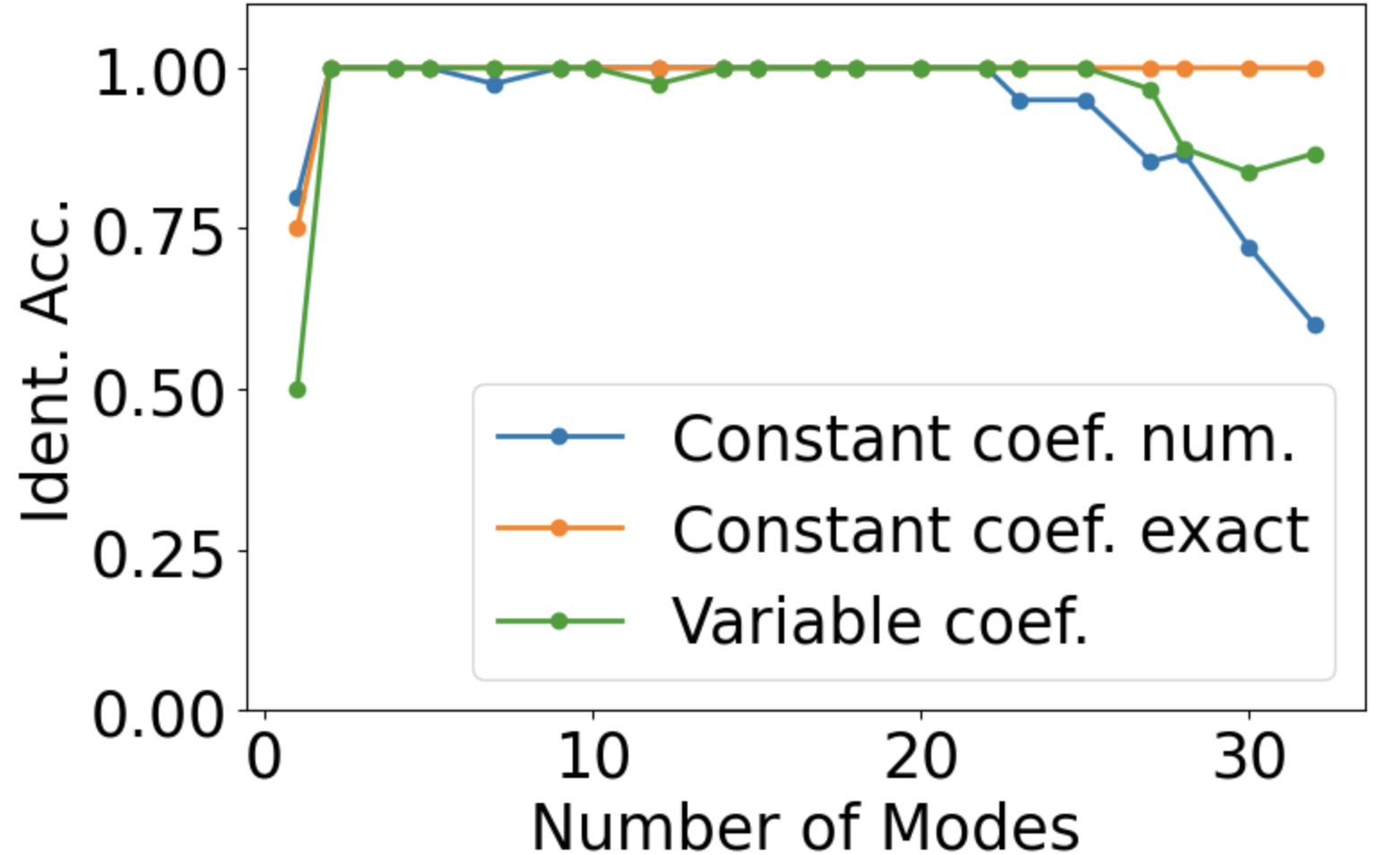}
    \end{minipage}&
    \begin{minipage}[p]{0.35\textwidth}
    \includegraphics[width = \textwidth]{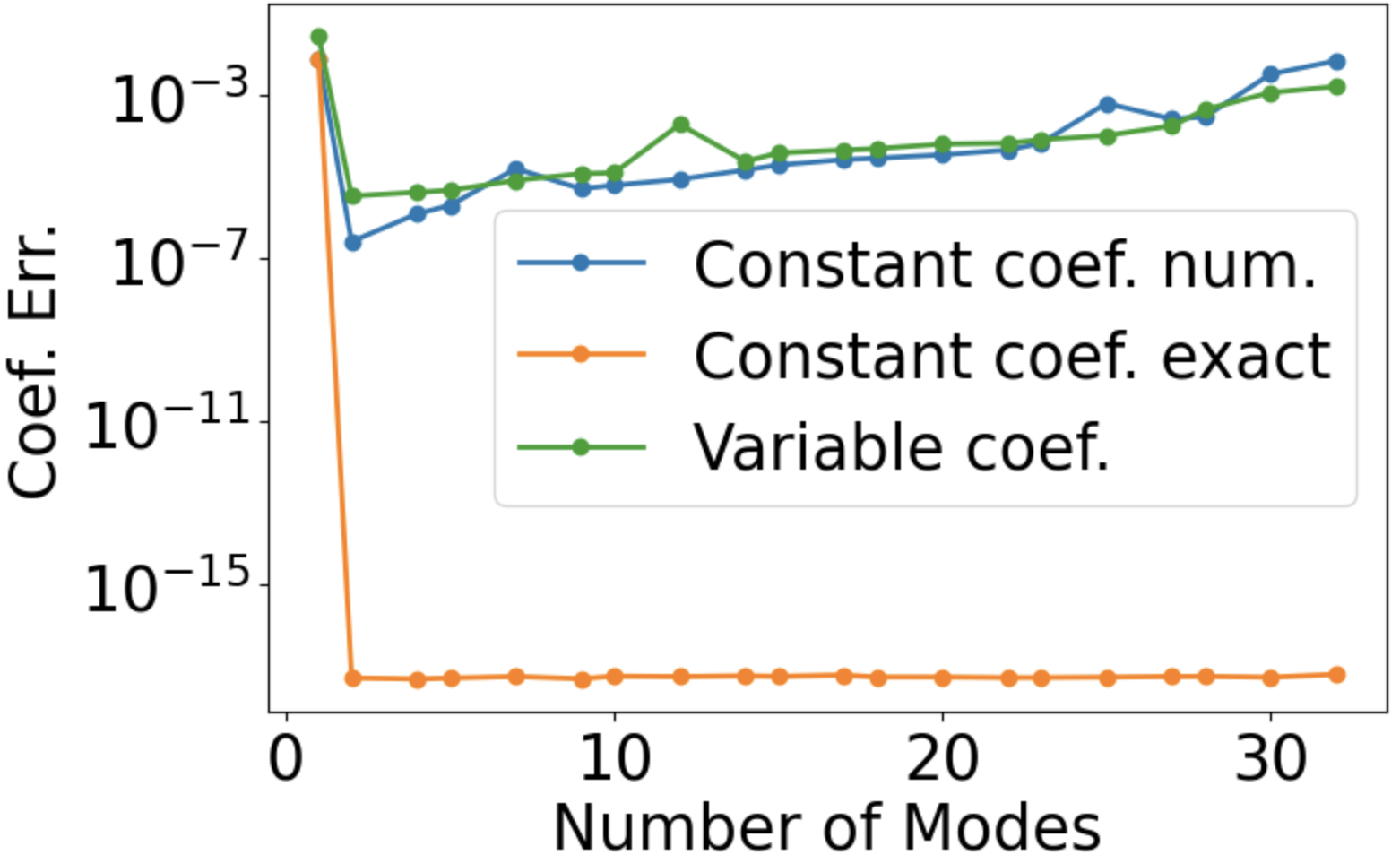}
    \end{minipage}\\\quad\quad (a)&\quad\quad (b)\\
    \begin{minipage}[p]{0.35\textwidth}
    \includegraphics[width = \textwidth]{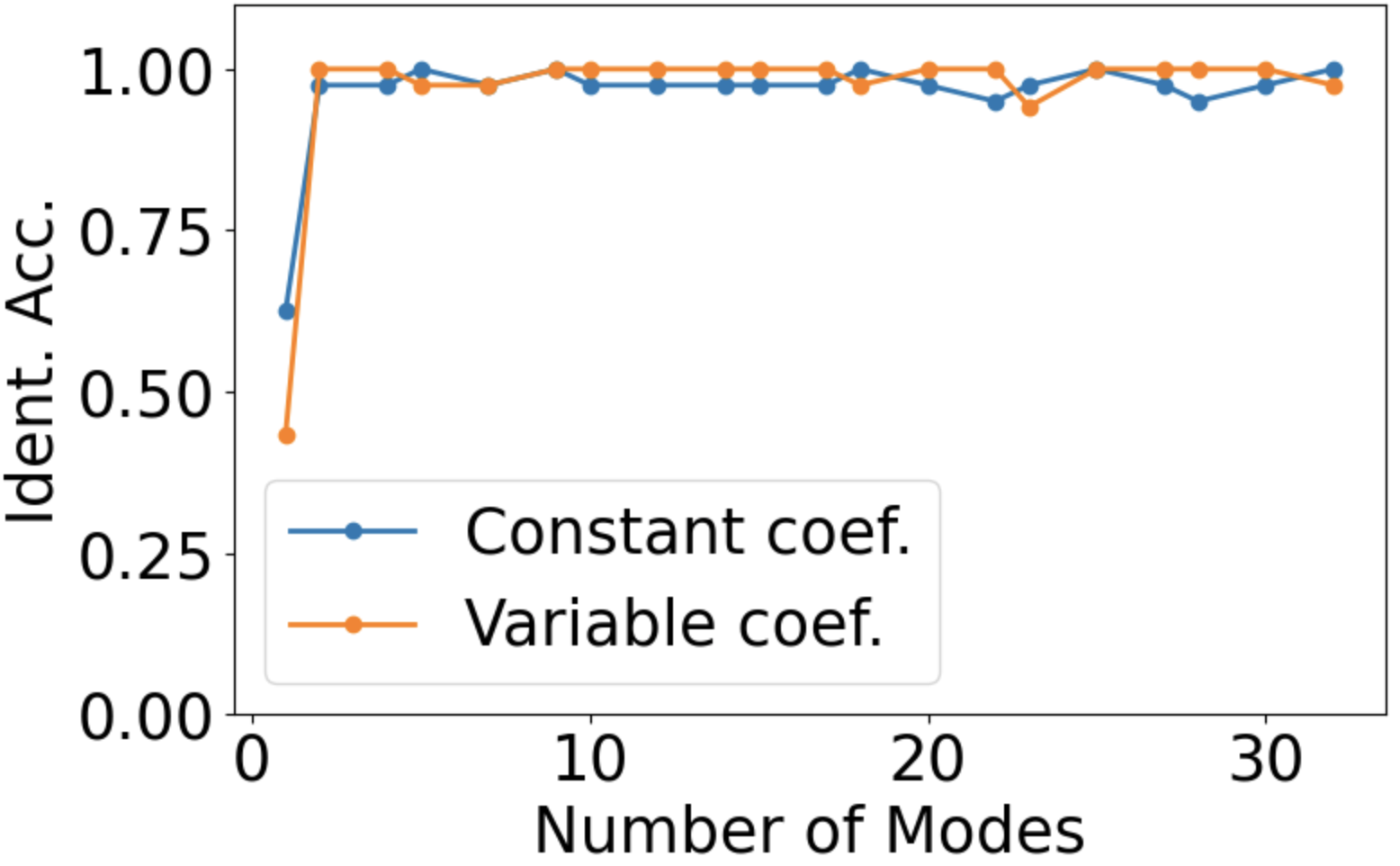}
    \end{minipage}&
    \begin{minipage}[p]{0.35\textwidth}
    \includegraphics[width = \textwidth]{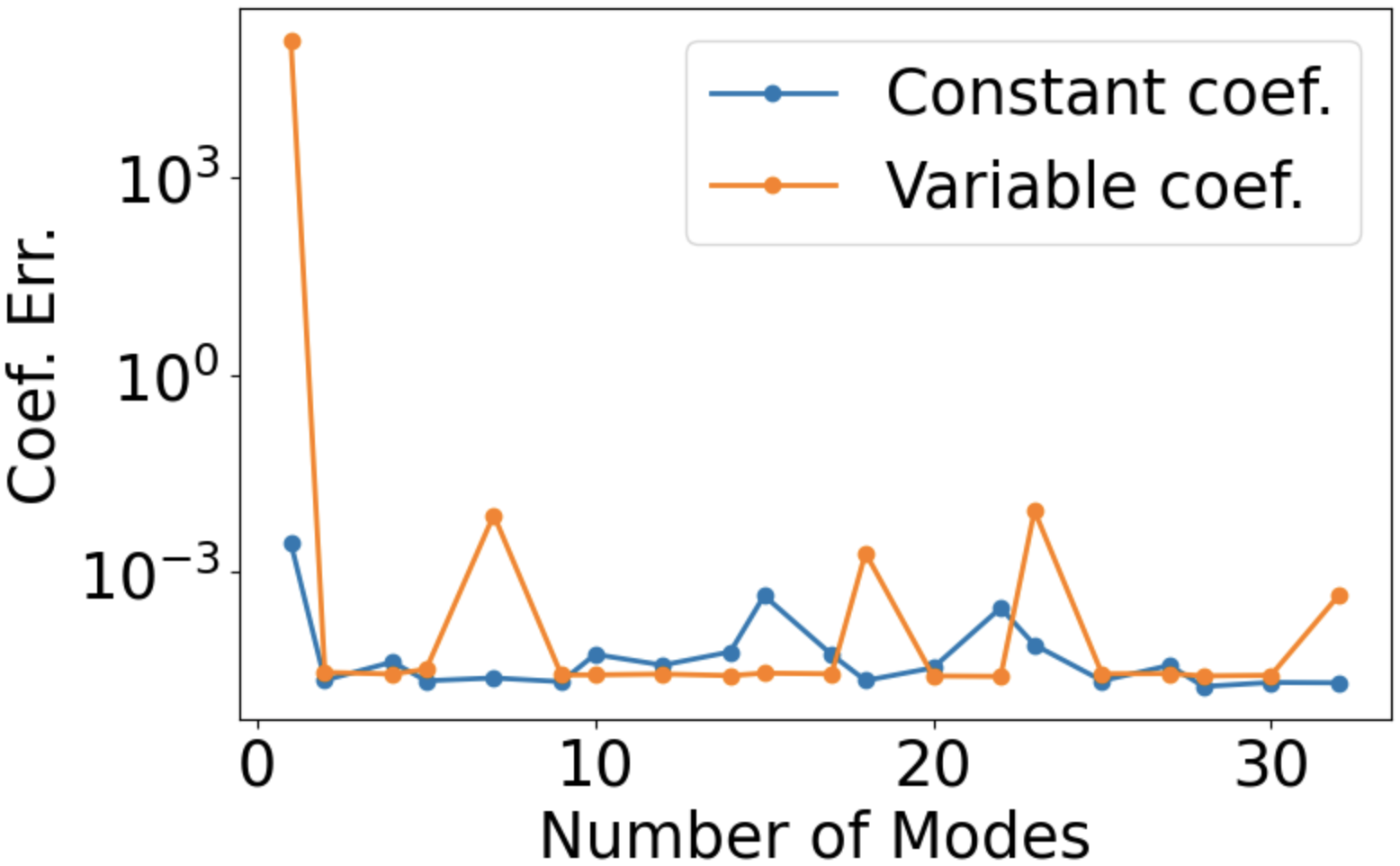}
    \end{minipage}\\
    \quad\quad (c)&\quad\quad (d)
\end{tabular}
\caption{Identification accuracy and reconstruction errors when the initial conditions are random functions with varying numbers of modes. (a) The identification accuracy  and  (b) the reconstruction errors for transport equations. (c) The identification accuracy and (d) the reconstruction errors for heat equations.}\label{fig:random_init}
\end{figure}
Figure~\ref{fig:random_init} (a) and (b) record the average identification accuracy and reconstruction error for the transport equations, and (c) and (d) record those for the heat equations. We observe that, in all cases, the identification accuracy immediately becomes correct once the initial condition has at least 2 different modes since there is only one term in the PDE.  For transport equations, as we can see, when the initial data contain high-frequency modes that can not be resolved by the grid, significant numerical errors will degrade the result, which is further justified by the fact that there is no such issue when using exact data (including derivatives). This is not so obvious for diffusion equations since the solution is smoothed out in time quickly.  We also note that the reconstruction errors for the equation with variable coefficient are greater than that with constant coefficient in general. This is due to the approximation of PDE on each local patch by a PDE with constant coefficients. 

This set of experiments emphasizes the importance of the variability, i.e., diversity of Fourier modes, in solution data. On one hand, if the solution data do not have enough diverse modes or variation, the identification can fail even for PDE identification with constant coefficients. On the other hand, if the solution's rapid variation can not be resolved by the measurement/computation resolution, 
the identification is compromised due to significant numerical errors in the feature approximations.

\subsubsection{Effects of the proposed data-driven patch filtering}\label{exp_patchfilt}
In this set of experiments, we demonstrate the effects of our data-driven patch filtering proposed in Section~\ref{SEC:patch} specifically when the data contains noises. In Figure~\ref{fig_patch_filt} (a), (b), and (c), we show solution data over the domain $[-1,1]$ of 
\begin{enumerate}
    \item [I.] transport equation with variable speed $u_t(x,t)=1000 t \sin(4\pi t/0.03)u_x(x,t)$ for $t\in [0,0.03]$;
    \item [II.] heat equation $u_t(x,t) = 0.5u_{xx}(x,t)$ for $t\in[0,0.03]$;
    \item [III.] inviscid Burgers' equation $u_t = 1.1u(x,t)u_x(x,t)$ for $t\in[0,0.6]$.
\end{enumerate}
 All cases have periodic boundary conditions and take the bump function as the initial conditions. Upon the solution data $u$, we also add the noise with intensity $p\%$ of the standard deviation of $u$. For (a), we added $5\%$; for (b) and (c), we added $0.5\%$. No denoising process was applied. 

For all cases, we randomly select 10 spatially located sensors and each of them can collect data within a rectangle ($7$ neighboring points in each space dimension and $11$ neighboring points in time) centered at 10 points uniformly distributed over the observation duration. They are marked as white dots in Figure~\ref{fig_patch_filt}. 
\begin{figure}[!htb]
\begin{tabular}{ccc}
    \begin{minipage}[p]{0.3\textwidth}
    \includegraphics[width = \textwidth]{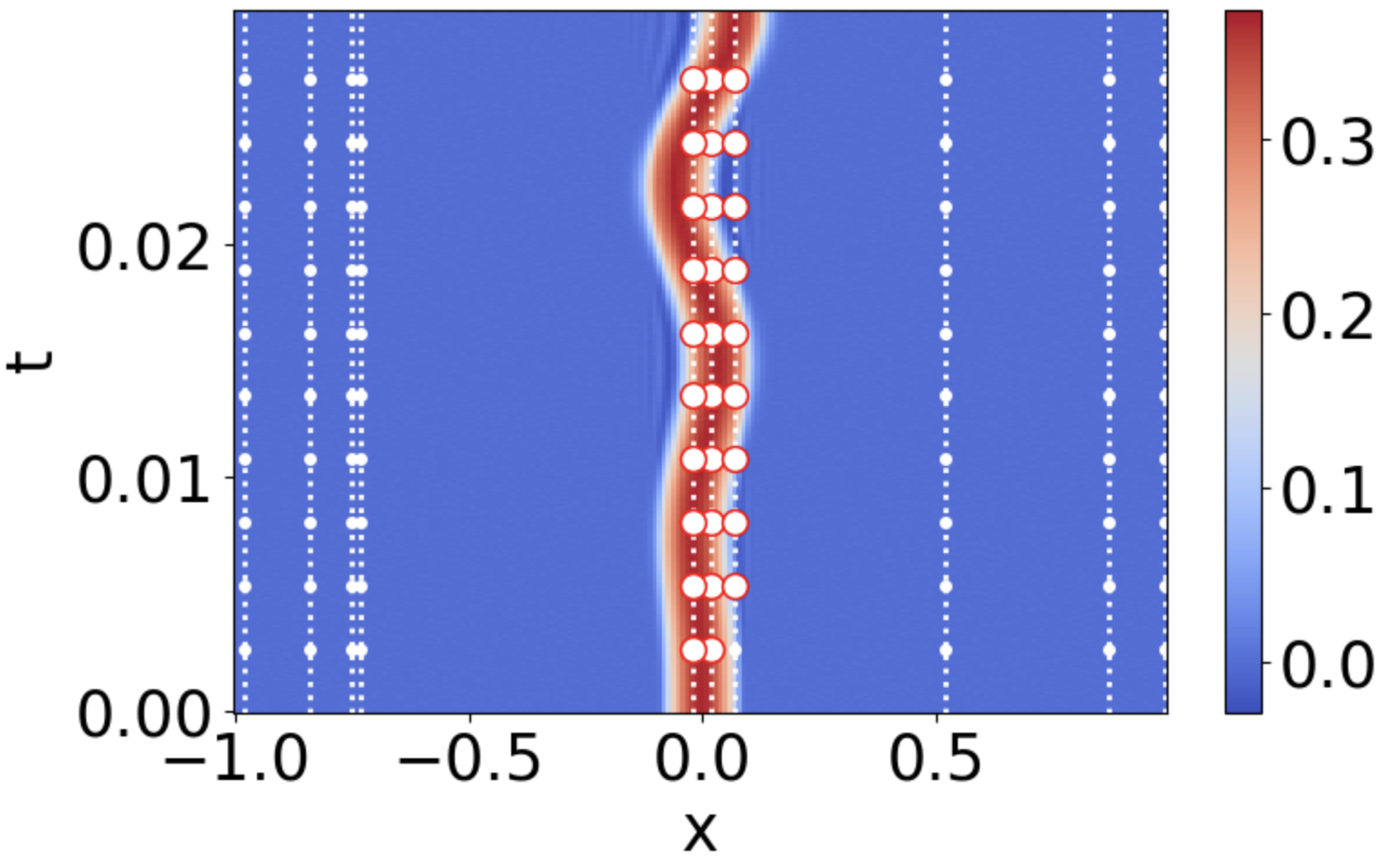}
    \end{minipage}&
    \begin{minipage}[p]{0.3\textwidth}
    \includegraphics[width = \textwidth]{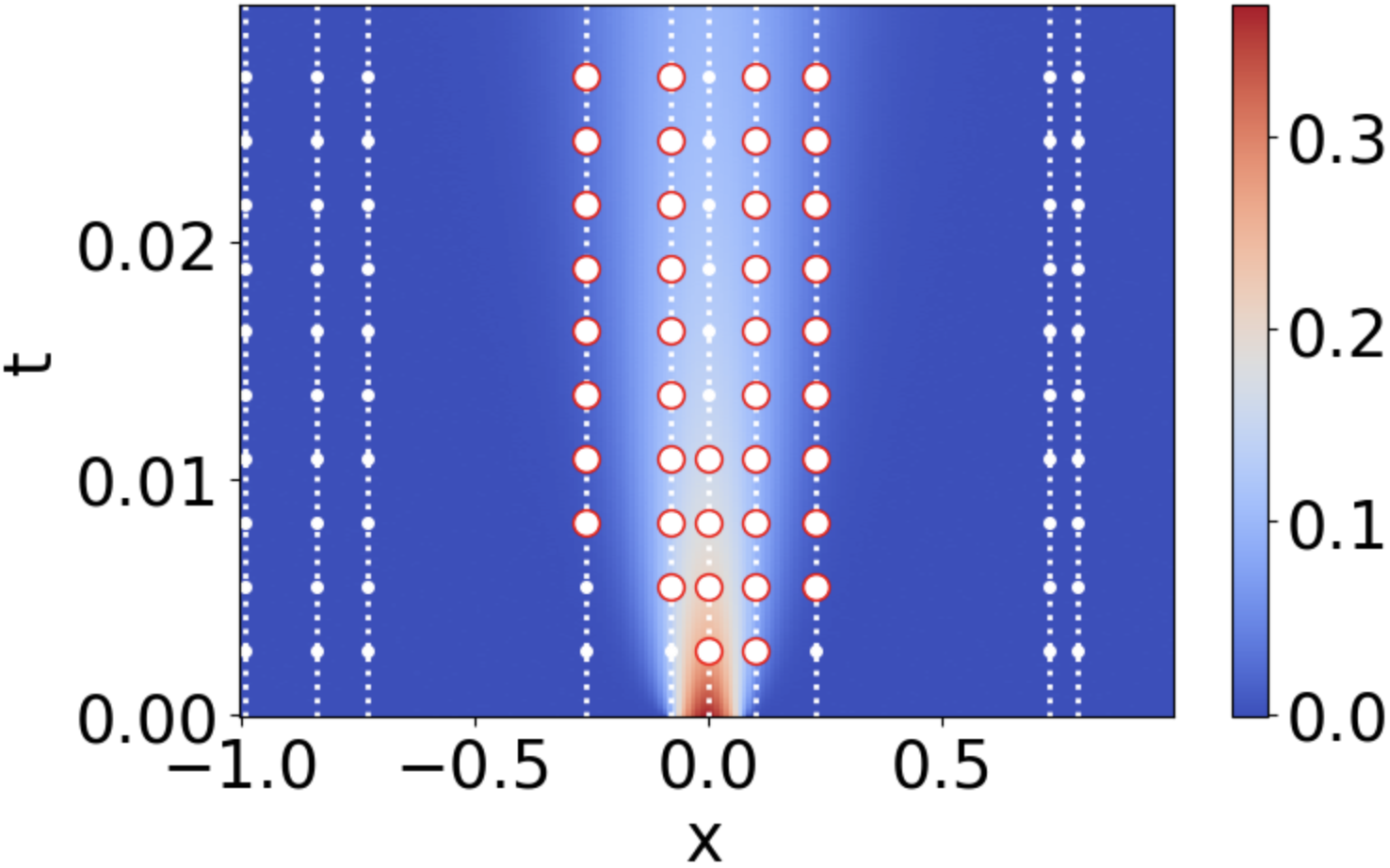}
    \end{minipage}&
    \begin{minipage}[p]{0.3\textwidth}
    \includegraphics[width = \textwidth]{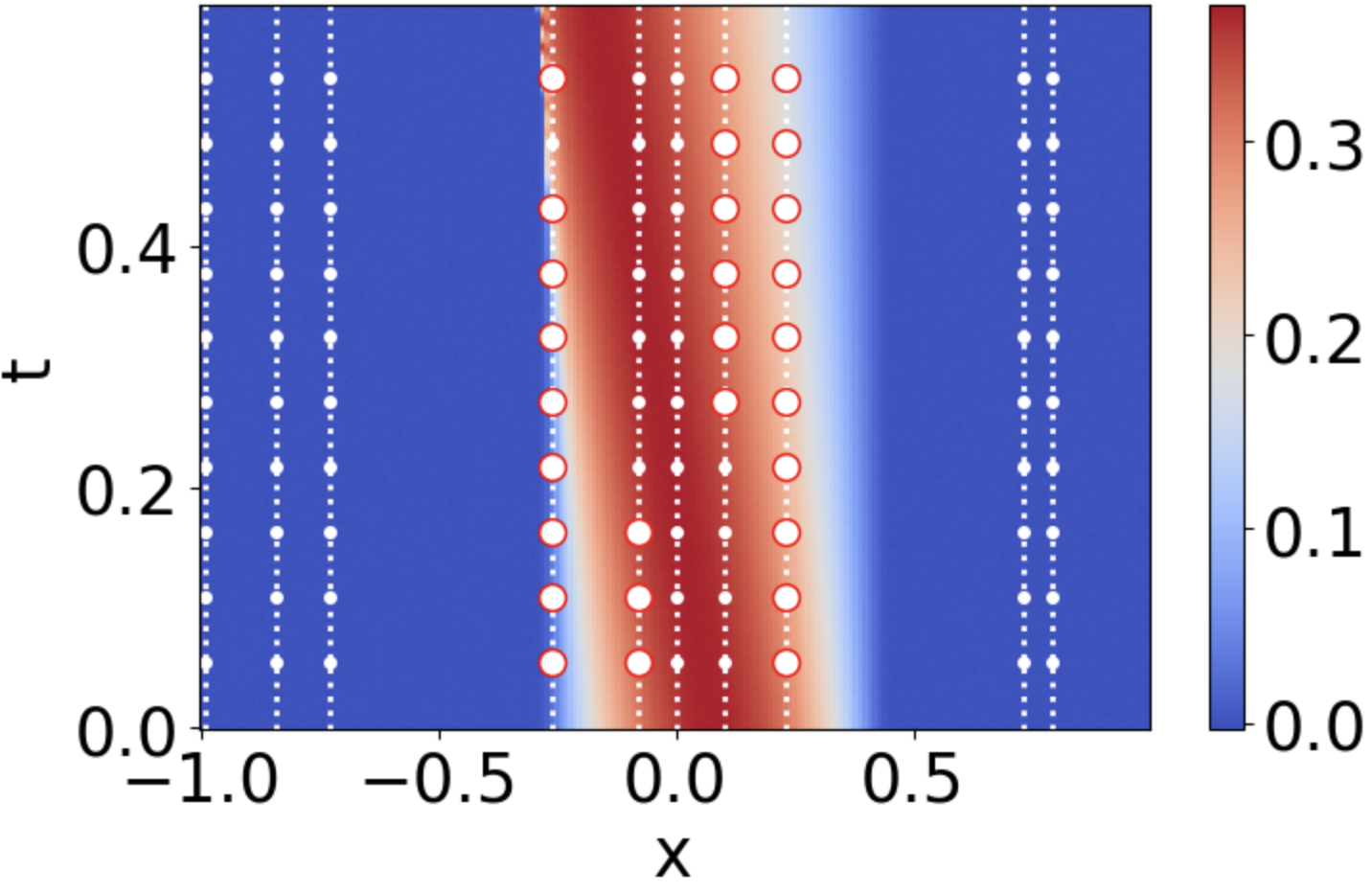}
    \end{minipage}
    \\
    (a)&(b)&(c)
\end{tabular}
\caption{Data-driven patch selection with noisy data.}\label{fig_patch_filt}
\end{figure}
\begin{table}[!htb]
\centering
\caption{Comparison of average identification accuracy between CaSLR with patch trimming and without.}\label{tab_patch_filt}
\begin{tabular}{cccc}
\hline\hline
&Transport&Heat&Burgers\\\hline
With Trim&$0.85$&$0.85$&$0.45$\\
Without Trim&$0.80$&$0.80$&$0.20$
\end{tabular}
\end{table}

The red circled dots show the centers of patches selected by our data-driven approach, and based on these limited data, CaSLR correctly identified the PDE types in all cases. Noticeably, the selected patches avoid those regions where the solution is close to a constant. For example, in (a), the selected patches meander along with the compactly supported solution; in (b), these patches are located near transitioning area between the non-zero solution data and the flat region, where the derivative  values are non-trivial; in (c), similarly to (b), the selected patches are located near the  transitioning area. Under the criterion that the collected data should be numerically stable, we note that the patch centered at the singularity of the solution is not chosen. We also emphasize that, if we were to use all the patches in (a), then the PDE type may be wrongly identified. In Table~\ref{tab_patch_filt}, we record the average identification accuracy of the proposed method applied to the equations above when the sensors are randomly located, and we observe that with the patch trimming, the accuracy is improved in all cases. Under the influence of noise, the true dynamics can be submerged by random variations. These examples show that our proposed data-drive patch selection scheme provides robustness to some extent against the noise for PDE identification.

\section{Conclusion}
We have studied a few basic questions for PDE learning from observed solution data. Using various types of linear evolution PDEs, we have shown 1) how the approximate dimension (richness) of the data space spanned by all snapshots along a solution trajectory depends on the differential operator and initial data, and 2) the identifiability of a differential operator from its solution data. Moreover, we propose a  Consistent and Sparse Local Regression(CaSLR) method, which enforces global consistency and involves as few terms as possible from the dictionary using local measurement data from a single solution trajectory, for general PDE identification. 

\section*{Acknowledgment}
H. Zhao’s research is partially supported by NSF grant DMS-2012860.
\appendix
\section{Proof of~(\ref{eq_est_mean}) and~(\ref{eq_est_var})}\label{append_detail}
In the following, we assume $N > 1$. Recalling that the variance of a random variable $X$ can be expressed as $\mathbb{E}[X^2]-(\mathbb{E}[X])^2$, we get
\begin{align*}
\sum_{m=1}^N\mathbb{E}\left[(\zeta_m-\frac{1}{N}\sum_{n=1}^N\zeta_n)^2\right]&=\sum_{m=1}^N\mathbb{E}\left[\zeta_m^2\right]-\frac{1}{N}\mathbb{E}\left[(\sum_{n=1}^N\zeta_n)^2\right]\\
&=\frac{N(B-1)}{B}\sigma^2+\sum_{m=1}^N\mu_m^2-\frac{B-1}{B}\sigma^2-\frac{1}{N}(\sum_{n=1}^N\mu_n)^2\\
&=\frac{(N-1)(B-1)}{B}\sigma^2+\sum_{m=1}^N\mu_m^2-\frac{1}{N}(\sum_{n=1}^N\mu_n)^2.
\end{align*}
Hence by defining the estimator:
\begin{align*}
\widehat{\sigma}^2 = \frac{B\sum_{m=1}^N\left(\zeta_m-\frac{1}{N}\sum_{n=1}^N\zeta_n\right)^2}{(N-1)(B-1)}
\end{align*}
and noting the Lipschitz assumption, we have
\begin{align*}
\mathbb{E}[\widehat{\sigma}^2-\sigma^2]&=\frac{B\left(\sum_{m=1}^N\mu_m^2-\frac{1}{N}(\sum_{n=1}^N\mu_n)^2\right)}{(N-1)(B-1)}\leq\frac{B\sum_{m=1}^N\mu_m^2}{(N-1)(B-1)}\leq\frac{DNBL^2R^2}{(N-1)(B-1)}.
\end{align*}
As for the variance of the estimator, we  notice that since $\zeta_n$, $n=1,\dots, N$ are independent Gaussian random variables, if we denote $S:=\sqrt{\frac{B-1}{B}}\sigma$, $\sum_{n=1}^N\zeta_n^2/S^2$ has a noncentral chi-squared distribution whose mean is $N+\sum_{n=1}^N\mu_n^2/S^2$, and variance is $2(N+2\sum_{n=1}^N\mu_n^2/S^2)$; and $(\sum_{n=1}^N\zeta_n)^2/(NS^2)$ also has a noncentral chi-squared distribution whose mean is $1+(\sum_{n=1}^N\mu_n)^2/(NS^2)$, and variance is $2(1+2(\sum_{n=1}^N\mu_n)^2/(NS^2))$. First, we compute the covariance 
\begin{align*}
 &\text{Cov}(\sum_{n=1}^N\zeta_n^2,(\sum_{n=1}^N\zeta_n)^2)=\mathbb{E}[\sum_{n=1}^N\zeta_n^2(\sum_{m=1}^N\zeta_m)^2]-\mathbb{E}[\sum_{n=1}^N\zeta_n^2]\mathbb{E}[(\sum_{n=1}^N\zeta_n)^2]\\
 &=\mathbb{E}[\sum_{n=1}^N\zeta_n^2(\sum_{m=1}^N\zeta_m)^2]-(NS^2 +\sum_{n=1}^N \mu_n^2)(NS^2 +(\sum_{n=1}^N\mu_n)^2 ).
\end{align*}
Focusing on the first term, we have
\begin{align*}
&\mathbb{E}[\sum_{n=1}^N\zeta_n^2(\sum_{m=1}^N\zeta_m)^2] = \sum_{n=1}^N\mathbb{E}[\zeta_n^2\sum_{m=1}^N\zeta_m^2]+\sum_{n=1}^N\mathbb{E}[\zeta_n^2\sum_{i\neq j}\zeta_i\zeta_j]\\
&=\sum_{n=1}^N\mathbb{E}[\zeta_n^4]+\sum_{n=1}^N\left(\mathbb{E}[\zeta_n^2]\sum_{m=1,m\neq n}^N\mathbb{E}[\zeta_m^2]\right)+2\sum_{n=1}^N\left(\mathbb{E}[\zeta_n^3]\sum_{m=1,m\neq n}^N\mathbb{E}[\zeta_m]\right)+\\
&\sum_{n=1}^N\left(\mathbb{E}[\zeta_n^2]\sum_{i\neq n,j\neq n,i\neq j}\mathbb{E}[\zeta_i]\mathbb{E}[\zeta_j]\right)\\
&=\sum_{n=1}^N(\mu_n^4+6\mu_n^2S^2+3S^4)+\sum_{n=1}^N\left((\mu_n^2+S^2)\sum_{m=1,m\neq n}^N(\mu_m^2+S^2)\right)+\\
&2\sum_{n=1}^N\left((\mu_n^3+3\mu_nS^2)\sum_{m=1,m\neq n}^N\mu_m\right)+\sum_{n=1}^N\left((\mu_n^2+S^2)\sum_{i\neq n,j\neq n,i\neq j}\mu_i\mu_j\right)\\
&=\sum_{n=1}^N\mu_n^2\left(\sum_{n=1}^N\mu_n\right)^2+\left(2(N-1)\sum_{n=1}^N\mu_n^2+6\left(\sum_{n=1}^N\mu_n\right)^2+(N-2)\sum_{n\neq m}\mu_n\mu_m\right)S^2+\\
&N(N+2)S^4.
\end{align*}
Hence, we have
\begin{align*}
&\text{Cov}(\sum_{n=1}^N\zeta_n^2,(\sum_{n=1}^N\zeta_n)^2)\\
  &=\left((N-2)\sum_{n=1}^N\mu_n^2+(6-N)\left(\sum_{n=1}^N\mu_n\right)^2+(N-2)\sum_{n\neq m}\mu_n\mu_m\right)S^2+2NS^4\\
  &=4\left(\sum_{n=1}^N\mu_n\right)^2S^2+2NS^4 .
\end{align*}
Now we note that
\begin{align*}
    &\text{Var}\left[\sum_{m=1}^N\left(\zeta_m-\frac{1}{N}\sum_{n=1}^N\zeta_n\right)^2\right]=\text{Var}\left[\sum_{m=1}^N\zeta^2_m\right]+\frac{1}{N^2}\text{Var}\left[(\sum_{n=1}^N\zeta_n)^2\right]-\frac{2}{N}\text{Cov}(\sum_{m=1}^N\zeta_m^2,(\sum_{m=1}^N\zeta_m)^2)\\
    &=2(NS^4+2\sum_{n=1}^N\mu_n^2S^2)+2S^4(1+2(\sum_{n=1}^N\mu_n)^2/(NS^2))-\frac{8}{N}\left(\sum_{n=1}^N\mu_n\right)^2S^2-4S^4.
\end{align*}
After simplification, we get
\begin{align*}
    &\text{Var}\left[\sum_{m=1}^N\left(\zeta_m-\frac{1}{N}\sum_{n=1}^N\zeta_n\right)^2\right]=2(N-1)S^4+4\left(\sum_{n=1}^N\mu_n^2-\frac{\left(\sum_{n=1}^N\mu_n\right)^2}{N}\right)S^2.
\end{align*}
Considering the Lipschitz assumption, we obtain
\begin{align*}
    &\text{Var}\left[\sum_{m=1}^N\left(\zeta_m-\frac{1}{N}\sum_{n=1}^N\zeta_n\right)^2\right]\leq 2(N-1)S^4+4NDL^2R^2S^2.
\end{align*}
Therefore, denoting $\gamma =4DL^2R^2$, then we get
\begin{align*}
   \text{Var}[\widehat{\sigma}^2]&\leq\frac{2(N-1)(B-1)^2\sigma^4+\gamma N B(B-1)\sigma^2}{(N-1)^2(B-1)^2}\\
   &=\frac{2\sigma^4}{N-1}+\frac{NB\sigma^2\gamma}{(N-1)^2(B-1)}.
\end{align*}

%%%%%%%%%%%%%%%%%%%%%%%%%%%%%%%%%%%%%%%%%%%%%%%%%%%%%%%%%%%%%%%%%%%%%%%%%%%%%%%%
\bibliographystyle{plain}
\bibliography{main}
\end{document}